\documentclass[11pt]{article}
\renewcommand{\baselinestretch}{1.2}  

\usepackage{mathtools}
\usepackage{empheq}
\usepackage{amsfonts}
\usepackage{amsgen,amsmath,amstext,amsbsy,amsopn,amssymb,subfigure,amsthm,stmaryrd}
\usepackage[dvips]{graphicx}
\usepackage{comment}
\usepackage{enumitem}
\usepackage{natbib}
\usepackage{booktabs}
\usepackage{blkarray}
\usepackage{algorithm}
\usepackage{algpseudocode}
\usepackage{url}
\usepackage{longtable}
\usepackage{mathtools}
\usepackage{multirow}
\DeclarePairedDelimiter\ceil{\lceil}{\rceil}
\DeclarePairedDelimiter\floor{\lfloor}{\rfloor}

\usepackage[usenames]{color}
\newcommand{\red}{\color{red}}
\newcommand{\blue}{\color{blue}}

\renewcommand{\baselinestretch}{1.2}  

\usepackage{mathtools}
\usepackage{empheq}
\usepackage{amsfonts}
\usepackage{amsgen,amsmath,amstext,amsbsy,amsopn,amssymb,subfigure,amsthm,stmaryrd}
\usepackage[dvips]{graphicx}
\usepackage{comment}
\usepackage{enumitem}
\usepackage{natbib}
\usepackage{booktabs}
\usepackage{blkarray}
\usepackage{algorithm}
\usepackage{algpseudocode}
\usepackage{url}
\usepackage{longtable}
\usepackage{mathtools}
\usepackage{multirow}

\usepackage[usenames]{color}

\allowdisplaybreaks

\newtheorem{Theorem}{Theorem}[section]

\newtheorem{Lemma}{Lemma}[section]
\newtheorem{Remark}{Remark}[section]
\newtheorem{Corollary}{Corollary}[section]

\newtheorem{Assumption}{Assumption}[section]

\allowdisplaybreaks

\DeclareMathAlphabet\mathbfcal{OMS}{cmsy}{b}{n}

\newcommand{\be}{\begin{equation}}
\newcommand{\ee}{\end{equation}}
\newcommand{\bea}{\begin{eqnarray}}
\newcommand{\eea}{\end{eqnarray}}
\newcommand{\beas}{\begin{eqnarray*}}
	\newcommand{\eeas}{\end{eqnarray*}}

\newcommand{\I}{{\mathbf{I}}}

\newcommand{\Cov}{{\rm Cov}}

\newcommand{\bbeta}{\boldsymbol{\beta}}
\newcommand{\bgamma}{\boldsymbol{\gamma}}
\newcommand{\bU}{\boldsymbol{U}}

\newcommand{\tr}{{\rm tr}}

\newcommand{\argmin}{\mathop{\rm arg\min}}

\makeatletter
\newcommand*{\rom}[1]{\expandafter\@slowromancap\romannumeral #1@}
\makeatother 

\allowdisplaybreaks

\begin{document}

	\title{Testing for high-dimensional network parameters in auto-regressive models}
	
	\author{Lili Zheng$^1$ and Garvesh Raskutti$^1$}
	
	\date{}
	
	\maketitle

	\footnotetext[1]{Department of Statistics, University of Wisconsin-Madison}
	
	\bigskip
	
	\begin{abstract}
	High-dimensional auto-regressive models provide a natural way to model influence between $M$ actors given multi-variate time series data for $T$ time intervals. While there has been considerable work on network estimation, there is limited work in the context of inference and hypothesis testing. In particular, prior work on hypothesis testing in time series has been restricted to linear Gaussian auto-regressive models. From a practical perspective, it is important to determine suitable statistical tests for connections between actors that go beyond the Gaussian assumption. In the context of \emph{high-dimensional} time series models, confidence intervals present additional estimators since most estimators such as the Lasso and Dantzig selectors are biased which has led to \emph{de-biased} estimators. In this paper we address these challenges and provide convergence in distribution results and confidence intervals for the multi-variate AR(p) model with sub-Gaussian noise, a generalization of Gaussian noise that broadens applicability and presents numerous technical challenges. The main technical challenge lies in the fact that unlike Gaussian random vectors, for sub-Gaussian vectors zero correlation does not imply independence. The proof relies on using an intricate truncation argument to develop novel concentration bounds for quadratic forms of dependent sub-Gaussian random variables. Our convergence in distribution results hold provided $T = \Omega((s \vee \rho)^2 \log^2 M)$, where $s$ and $\rho$ refer to sparsity parameters which matches existed results for hypothesis testing with i.i.d. samples. We validate our theoretical results with simulation results for both block-structured and chain-structured networks.
	\end{abstract}
	
\section{Introduction}

Vector autoregressive models arise in a number of applications including macroeconomics (see e.g.\cite{ang2003no},\cite{hansen2003structural},\cite{shan2005does}), computational neuroscience (see e.g.\cite{goebel2003investigating},\cite{seth2015granger},\cite{harrison2003multivariate}, \cite{bressler2007cortical}), and many others (see e.g.\cite{michailidis2013autoregressive},\cite{fujita2007modeling}). Recent years has seen substantial development in the theory and methodology of high-dimensional auto-regressive models with respect to parameter estimation (see e.g. \cite{song2011large},\cite{basu2015regularized},\cite{davis2016sparse},\cite{medeiros2016l1}, \cite{MarkRasWillett18}). In particular if there are $M$ dependent time series (e.g. voxels in the brain, actors in a social network, measurements at different spatial locations), \emph{time series network} models allow us to model temporal dependence between actors/nodes in a network.

More precisely, consider the following time series auto-regressive network model with lag $p$,
\begin{equation}\label{AR(p)}
X_{t+1}=\sum_{j=1}^p A^*(j)X_{t+1-j}+\epsilon_{t},
\end{equation}
where $\{X_t\}_{t=0}^T\in \mathbb{R}^M$ is the time series data we have access to, $\{A^*(j) \in \mathbb{R}^{M \times M}, j=1,\dots,p\}$ are the network parameters of interest and $\epsilon_{t}\in \mathbb{R}^M$ is zero-mean noise. We are considering the high-dimensional setting where the number of nodes $M$ in the network is much larger than the sample size $T$. Prior work in \cite{basu2015regularized} has addressed the question of how to estimate the network parameter $A^*$ with Gaussian noise $\epsilon_t$ under sparsity assumptions and various structural constraints. In this paper, we focus on \emph{inference and hypothesis testing} for the parameter $A^*$ given the data $(X_t)_{t=0}^{T}$. 

In high-dimensional statistics, there has recently been a growing body of work on confidence intervals and hypothesis testing under structural assumptions such as sparsity. Since the widely used Lasso estimator for sparse linear regression is asymptotically biased, one-step estimators based on bias-correction have been studied in works such as \cite{zhang2014confidence}, \cite{van2014asymptotically} and \cite{javanmard2014confidence} which are referred to as LDPE, de-sparsifying and de-biasing estimator respectively. Low-dimensional components of these estimators have asymptotic normality and thus can be used for constructing hypothesis testing and confidence intervals. 

In this paper, we adopt the framework of Ning and Liu (\cite{ning2017general}) 
who propose a high dimensional test statistic based on score function, called the decorrelated score function which we briefly describe here. Formally, consider a statistical model $\mathcal{P}=\{\mathbb{P}_{\bbeta}:\bbeta\in \Omega\}$ with high-dimensional parameter vector ${\bbeta}=(\theta,\bgamma^\top)^\top\in \mathbb{R}^d$. Suppose we are interested in the scalar parameter $\theta$ and $\bgamma\in \mathbb{R}^{d-1}$ is the nuisance parameter. Suppose data $\{\bU_i, i=1,\dots,n\}$ are i.i.d. data following distribution $\mathbb{P}_{\bbeta}$, then the negative log-likelihood function is defined as
\begin{equation*}
    \ell(\theta,\bgamma)=-\frac{1}{n}\sum_{i=1}^n \log f(\bU_i;\theta,\bgamma).
\end{equation*}
It is known that the score function $\sqrt{n}\nabla_{\theta}\ell(0,\bgamma^*)$ is asymptotically normal if the true parameter $\bbeta^*=(0,\bgamma^*)$. If $\bgamma^*$ is substituted by some estimator $\hat{\bgamma}$, the estimation induced error can be approximated as the following:
\begin{equation*}
    \sqrt{n}\nabla_{\theta}\ell(0,\hat{\bgamma})-\sqrt{n}\nabla_{\theta}\ell(0,\bgamma^*)\approx\sqrt{n}\nabla_{\theta\bgamma}^2\ell(0,\bgamma^*)(\hat{\bgamma}-\bgamma^*),
\end{equation*}
when $\hat{\bgamma}-\bgamma^*$ is small enough.
Although $\hat{\bgamma}-\bgamma^*$ converge to 0 with properly chosen $\hat{\bgamma}$, e.g. Lasso estimator, $\sqrt{n}\nabla_{\theta\bgamma}^2\ell(0,\bgamma^*)(\hat{\bgamma}-\bgamma^*)$ would not vanish if $\mathbb{E}_{\bbeta}\left(\nabla_{\theta\bgamma}^2\ell(0,\bgamma^*)\right)\neq 0$. This fact motivates the decorrelated score function:
\begin{equation*}
    S(\theta,\bgamma)=\nabla_{\theta}\ell(\theta,\bgamma)-\I_{\theta\bgamma}\I_{\bgamma\bgamma}^{-1}\nabla_{\bgamma}\ell(\theta,\bgamma),
\end{equation*}
with Fisher information matrix $\I=\mathbb{E}_{\bbeta}\left(\nabla^2\ell(\bbeta)\right)$. One can check that 
\begin{equation*}
    \mathbb{E}\left(\nabla_{\bgamma}S(\theta,\bgamma)\right)=0.
\end{equation*}
Both $\bgamma$ and $\I_{\theta\bgamma}\I_{\bgamma\bgamma}^{-1}$ are substituted by some estimator, and it is shown in \cite{ning2017general} that the decorrelated score function is asymptotically normal. 

In the linear regression case, the test statistic generated by the decorrelated score function in \cite{ning2017general} is equivalent to that constructed by de-biased estimator in \cite{van2014asymptotically}. However, \cite{ning2017general} allow a more general form, and thus is easier to adapt to the time series case. In fact Neykov et al.\cite{neykov2018unified} consider amongst other examples, high-dimensional time series with Gaussian error innovations. While Gaussian error innovations are widely used, many time series models include data that has bounded range or discrete data, for which the Gaussian distribution is not a natural fit. In this paper, we address the more general and technically challenging setting in which the noise $\epsilon_t$ is sub-Gaussian.

One of the important technical challenges in going from the Gaussian to the sub-Gaussian case is that dependent Gaussian vectors can be rotated to be independent, while such a result does not hold for sub-Gaussian vectors. Prior work in~\cite{wong2016lasso} addresses this challenges by imposing stationarity and $\beta$-mixing conditions. In order to avoid these conditions, we develop novel concentration bounds for sub-Gaussian random vectors.

In this paper, we investigate the hypothesis testing and confidence region with respect to a low-dimensional component of parameter matrices $\{A^*(j),j=1,\dots,p\}$ for sub-Gaussian data, using the testing framework in \cite{ning2017general}. Our major contributions are as follows:
\begin{itemize}
    \item Extending theoretical results in \cite{ning2017general} for high-dimensional hypothesis testing from Gaussian to sub-Gaussian temporal dependent data (VAR model), both under null and alternative hypothesis. We also show that our techniques lead to similar results to Neykov et al.\cite{neykov2018unified} in the Gaussian case but under less restrictive conditions;
    \item A novel concentration bound for quadratic forms of sub-Gaussian time series data. Note that unlike Gaussian vectors which can be rotated to be independent, sub-Gaussian vectors can not which present  additional technical challenges. Our analysis also leads to estimators for covariance and regression parameters for time series data under sub-Gaussian assumptions which are of independent interest. 
    \item We also construct semi-parametric efficient confidence region for multivariate parameters with fixed dimension;
    \item Finally we support our theoretical guarantees with a simulation study on bounded noise, which is sub-Gaussian but not Gaussian.
\end{itemize}


\subsection{Related Work}
In the literature on inference for high-dimensional VAR models, most work focuses on the estimation problem. Song and Bickel (\cite{song2011large}) investigate penalized least squares algorithms for different penalties, with some externally imposed assumptions on the temporal dependence. Theoretical guarantees on Dantzig type and Lasso type estimators are studied in \cite{han2015direct} and \cite{basu2015regularized}, but with Gaussian noise. Barigozzi and Brownlees (\cite{barigozzi2018nets}) consider the inference for stationary dependence structure built among variables, other than the parameters in the VAR model. In our work, we control the error bounds of Lasso and Dantzig type estimators for parameter matrices, with sub-Gaussian noise. Then we establish asymptotic distribution of test statistic based on this.

In the high-dimensional hypothesis testing literature, there is some work regarding to testing for high-dimensional mean vector (\cite{srivastava2009test}), covariance matrices (\cite{chen2010tests},\cite{zhang2013tests}) and independence among variables (\cite{schott2005testing}). While for testing on regression parameters, most work assumes i.i.d samples. \cite{lockhart2014significance}, \cite{taylor2014post} and \cite{lee2016exact} proposes methods to test whether a covariate should be selected conditioning on the selection of some other covariates. A penalized score test depending on the tuning parameter $\lambda$ is considered in \cite{voorman2014inference}. Our work follows the a line of work by \cite{zhang2014confidence}, \cite{van2014asymptotically}, \cite{javanmard2014confidence} and \cite{ning2017general}, the de-sparsifying or decorrelated literature. We construct a VAR version of decorrelated score test proposed by \cite{ning2017general}. Chen and Wu (\cite{chen2018testing}) tackles the hypothesis testing problem for time series data as well, but they are testing the trend in a time series, instead of the autoregressive parameter which encodes the influence structure among variables.

As mentioned earlier, our work is most closely related to the prior work of Neykov et al.\cite{neykov2018unified}, which provides a hypothesis testing framework with high-dimensional Gaussian time series as a special case. In our work, we consider the more general and technically challenging case of sub-Gaussian vector auto-regressive models. Throughout this paper, we provide a comparison to results derived in this work for the Gaussian case.

\subsection{Organization of the Paper}
Section \ref{setup} explains the problem set up and proposes our test statistic. Theoretical guarantee is shown in section \ref{theory}. Specifically, section~\ref{NullSec} and \ref{AltSec} present the weak convergence rate of test statistic under the null and alternative hypothesis $\mathcal{H}_0$ and $\mathcal{H}_A$. Section \ref{Estimator} propose some feasible estimators, which satisfy the assumptions required and can be plugged into the test statistic. Section \ref{EstVarSec} considers the case when the variance of noise are unknown, and we construct a confidence region for multivariate parameter vectors in Section \ref{CRSec}. We consider the special case of the AR(1) model with Gaussian noise, a detailed comparison with \cite{neykov2018unified} is provided in section \ref{GaussAR1}. Section \ref{SimSec} provides simulation results and section \ref{main_proof} includes the proofs for the two main theorems. Much of the proof is deferred to Appendices.
\subsection{Notation}
We define the following norms for vectors and matrices: For a vector $u=(u_1,\dots,u_d)^\top\in \mathbb{R}^d$, we define the $p$-norm where $p \geq 1$,$\|u\|_p=\left(\sum_{i=1}^d u_i^p\right)^{\frac{1}{p}}.$ For a matrix $U\in \mathbb{R}^{m\times n}$, the $\ell_p$ norm and Frobenius norm of $U$ is defined as $\|U\|_p=\sup_{v}\frac{\left\|Uv\right\|_p}{\|v\|_p}, \quad \|U\|_{F}=\left(\sum_{i=1}^m\sum_{j=1}^n U_{ij}^2\right)^{\frac{1}{2}}.$ We also use notation $\|U\|_{1,1}$ to denote the $\ell_1$ penalty on $U$, which is $\sum_{i=1}^m\sum_{j=1}^n |U_{i,j}|$.
Furthermore, if $U$ is symmetric the trace norm of $U$ is $\|U\|_{\tr}=\text{tr}(\sqrt{U^2}).$ 

Throughout the paper, we assume that the entries of noise vectors $\{\epsilon_{ti},1\leq i\leq M\}_{t=-\infty}^{\infty}$ are independent sub-Gaussian variables with constant scale factor. A univariate centered random variable $X$ has a sub-Gaussian distribution with scale factor $\tau$ if
$$
M_{X}(t) \triangleq \mathbb{E}\left[\exp(tX)\right]\leq \exp(\tau^2t^2/2).
$$

\section{Problem Setup}\label{setup}

We consider a general vector auto-regressive time series with lag $p$, where $p$ is known and finite and independent of $T$ or other dimensions:
\begin{equation}\label{ARp1}
    X_{t+1}=\sum_{j=1}^{p}A(j)X_{t-j+1}+\epsilon_t,
\end{equation}
where $X_t\in \mathbb{R}^M$, $\epsilon_t\in \mathbb{R}^M$ is zero-mean entry-wise independent sub-Gaussian noise with identity covariance matrix, and $A(j)\in \mathbb{R}^{M\times M}, j=1,\cdots,p$ are parameters of interest. Define the matrix $A^*=(A(1),\cdots,A(p))\in \mathbb{R}^{M\times pM}$ and $\mathcal{X}_t=(X_t^\top,\cdots,X_{t-p+1}^\top)^\top\in \mathbb{R}^{pM}$, then we can also write (\ref{ARp1}) as
\begin{equation}\label{ARp2}
    X_{t+1}=A^*\mathcal{X}_t+\epsilon_t.
\end{equation}
For notational convenience, we assume that time series data $X_t$ has time range $1-p\leq t\leq T$. 

Based on data $(X_t)_{t=1-p}^{T}$, we test the hypothesis of whether a subset of entries in $A^*$ are $0$. Let $A_i^*$ be the $i$th row vector of $A^*$. Without loss of generality, suppose the entries we test are in rows $1,\cdots, k$. Define $D_m\subset \{1,\cdots,pM\}$ as the columns we test in $m$th row with $d_m=|D_m|$, and $D=\{(i,j):1\leq i\leq k, j\in D_i\}$, with $d=|D|=\sum_{m=1}^k d_m$. We test the null hypothesis:
\begin{equation}\label{EqnNullp}
    \mathcal{H}_0:\widetilde{A}_D=0
\end{equation}
where $\widetilde{A}_D=((A_1^*)_{D_1}^\top,\cdots,(A_k^*)_{D_k}^\top)^\top\in \mathbb{R}^d$. We also assume that $d$ is finite and not increasing with $T$. In the work of of Neykov et al.\cite{neykov2018unified}, $d$ is assumed to be $1$.

\subsection{Stationary distribution}

Since we are developing a hypothesis testing framework based on the decorrelated score test, it is important to specify a stationary distribution for $X_t$ 
Using standard notation from auto-regressive time series models, define the polynomial $\mathcal{A}(z)=I_M-\sum_{j=1}^p A(j)z^j$, where $I_M$ is an $M\times M$ identity matrix, and $z$ is a complex number. To guarantee the existence of a stationary solution to (\ref{ARp2}), we assume 
\begin{equation*}
     \det(\mathcal{A}(z))\neq 0,\quad |z|\leq 1.
\end{equation*}
Then we can write
\begin{equation*}
    \left(\mathcal{A}(z)\right)^{-1}=\sum_{j=0}^\infty \Psi_j z^j,
\end{equation*}
where $\Psi_j, j\geq 0$ are all real valued matrices which are polynomial functions of $A(i), 1\leq i\leq p$. Note that in the special case where $p=1$, $\Psi_j = (A^*)^j$.

It can be shown that the unique stationary solution to (\ref{ARp1}) is
\begin{equation*}
    X_t=\sum_{j=0}^\infty \Psi_j\epsilon_{t-j-1},
\end{equation*}
and the covariance matrix $\Sigma$ of $X_t$ satisfies
\begin{equation}\label{Upsilon_def}
    \Sigma=\Cov(X_t)=\sum_{j=0}^\infty \Psi_j\Psi_j^\top.
\end{equation}

\subsection{Decorrelated Score Function}

Using the frameworks developed in \cite{ning2017general} for independent design, we consider the decorrelated score test. First we define the \emph{score function} $S(A^*)\in \mathbb{R}^{M\times M}$, with each entry defined as follows:
	$$
	[S(A^*)]_{jk}=-\frac{1}{T}\sum_{t=0}^{T-1} (X_{t+1,j}-a_j^{*\top}\mathcal{X}_t)\mathcal{X}_{tk}=-\frac{1}{T}\sum_{t=0}^{T-1} \epsilon_{t,j} \mathcal{X}_{tk}.
	$$	
As pointed out in \cite{ning2017general}, the standard score function is infeasible and we need to consider the \emph{decorrelated score function}
\begin{equation*}
    \begin{split}
        S=(S_1^\top,S_2^\top,\cdots,S_k^\top)^\top\in \mathbb{R}^d,
    \end{split}
\end{equation*} 
with each $S_m\in \mathbb{R}^{d_m}$ corresponding to the tested row $(m,D_m)$:
	$$ S_m=-\frac{1}{T}\sum_{t=0}^{T-1} \epsilon_{t,m} (\mathcal{X}_{t,D_m}-w_{m}^{*\top} \mathcal{X}_{t,D_m^c}),
	$$
where $\mathcal{X}_{t,D_m}\in \mathbb{R}^{d_m}$ is composed of the entries of $\mathcal{X}_t$ whose indices are within set $D_m$. $\mathcal{X}_{t,D_m^c}\in \mathbb{R}^{pM-d_m}$ is also defined similarly and $w_{m}^*\in\mathbb{R}^{(pM-d_m)\times d_m}$ is chosen to satisfy
\begin{equation}\label{w_m_def}
    \Cov(X_{t,D_m}-w_{m}^{*\top}\mathcal{X}_{t,D_m^c}, \mathcal{X}_{t,D_m^c})=0.
\end{equation}
Specifically, $w_m^*$ is defined as a function of $\Upsilon=\Cov(\mathcal{X}_t)\in\mathbb{R}^{pM\times pM}$:
\begin{equation}\label{w_m_def_S}
    w_m^*=(\Upsilon_{D_m^c,D_m^c})^{-1}\Upsilon_{D_m^c,D_m}.
\end{equation}

\subsection{Test Statistic}
Based on the decorrelated score function $S_m$, we first define the statistic $V_{T,m}\in \mathbb{R}^{d_m}$: 
$$
V_{T,m}\triangleq \sqrt{T}(\Upsilon^{(m)})^{-\frac{1}{2}}S_m,
$$
with $\Upsilon^{(m)}\in \mathbb{R}^{d_m\times d_m}$ being defined as:
\begin{equation}\label{Cov_m_def}
    \begin{split}
\Upsilon^{(m)}&\triangleq \Cov(\mathcal{X}_{t,D_m}-w_m^{*\top}\mathcal{X}_{t,D_m^c})\\
&=\Cov(\mathcal{X}_{t,D_m}|\mathcal{X}_{t,D_m^c})\\
&=\Upsilon_{D_m,D_m}-\Upsilon_{D_m,D_m^c}(\Upsilon_{D_m^c,D_m^c})^{-1}\Upsilon_{D_m^c,D_m}.
    \end{split}
\end{equation}
Let $V_T$ be the $d$-dimensional vector concatenated by $V_{T,m}$'s:
$$
V_T=(V_{T,1}^\top,\cdots,V_{T,k}^\top)^\top.
$$
One of the main results of the paper is to show that $V_T$ is asymptotically Gaussian. Define $U_T=\|V_T\|_2^2$, then $U_T$ is asymptotically $\chi_d^2$. Since we do not know $\epsilon_t$, $w_m^*$, and $\Upsilon^{(m)}$, we later define estimators for these quantities. Formally, we define our test statistic $\widehat{U}_T$ as
\begin{equation}\label{TestStatistic}
    \widehat{U}_T=T\sum_{m=1}^k\widehat{S}_m^\top \left(\widehat{\Upsilon^{(m)}}\right)^{-1}\widehat{S}_m,
\end{equation}
where $\widehat{\Upsilon^{(m)}}\in \mathbb{R}^{d_m\times d_m}$ is an estimator for $\Upsilon^{(m)}$ and $\widehat{S}_m\in \mathbb{R}^{d_m}$ is defined as
\begin{equation*}
    \widehat{S}_m=-\frac{1}{T}\sum_{t=0}^{T-1}\left(X_{t+1,m}-(\widehat{A}_m)_{D_m^c}^\top \mathcal{X}_{t,D_m^c}\right)(\mathcal{X}_{t,D_m}-\hat{w}_m^\top \mathcal{X}_{t,D_m^c}),
\end{equation*}
with $\widehat{A}_m\in \mathbb{R}^{pM}$ and $\hat{w}_m\in \mathbb{R}^{(pM-d_m)\times d_m}$ estimating $A_m^*$ and $w_m^*$. Here we are not worried about the invertible issue of $\widehat{\Upsilon^{(m)}}$, since $\Upsilon^{(m)}$ is a low dimensional covariance matrix. To guarantee a good estimation of the high-dimensional parameter $A_m^*$ and $w_m^*$, we impose sparsity conditions upon them. Specifically, for each $1\leq m\leq M$, $1\leq i\leq k$ define
\begin{equation}\label{sparsity_def}
    \begin{split}
        \rho_m\triangleq\|A_m^*\|_0,\quad s_i\triangleq\|w_i^*\|_0,
    \end{split}
\end{equation} 
and note that they both depend on $A^*$.

The sparsity of $w_m^*$ can be implied by the sparsity of $\Upsilon^{-1}$, which is a common condition in high-dimensional hypothesis testing literature (e.g. see \cite{van2014asymptotically}). Specifically, the following Lemma shows that when lag $p=1$ and $A^*$ is symmetric, the sparsity of $w_m^*$ is implied by the sparsity of $A^*$:
\begin{Lemma}\label{SparsityBound}
    If $p=1$, $A^*\in \mathbb{R}^{M\times M}$ is symmetric, then $s_m$ defined in (\ref{sparsity_def}) satisfies
    \begin{equation*}
        s_m\leq d_m^2\max_{1\leq i\leq M}\rho_i,\quad \text{for } 1\leq m\leq k.
    \end{equation*}
\end{Lemma}
\noindent The proof for Lemma \ref{SparsityBound} is included in Appendix \ref{RestProof}.

\section{Theoretical guarantee}\label{theory}
In this section, we present uniform convergence results for test statistic $\widehat{U}_T$ under $\mathcal{H}_0$ and $\mathcal{H}_A$, with $A^*$ and estimators satisfying conditions. We also provide feasible estimators, and prove that they satisfy corresponding conditions in Section~\ref{Estimator}. Unknown variance and confidence region construction is discussed in Section~\ref{EstVarSec} and \ref{CRSec}. In Section~\ref{GaussAR1} we provide consequences of our theory under AR(1) model with Gaussian noise and compare our results with Neykov et al.\cite{neykov2018unified}.

Recall that the null hypothesis is
\begin{equation}
\mathcal{H}_0: \widetilde{A}_D=0,
\end{equation}
with $\widetilde{A}_D\in \mathbb{R}^{d}$ being concatenated by $(A_1^*)_{D_1}, \dots, (A_k^*)_{D_k}$. While for the alternative hypothesis, like in \cite{ning2017general}, we consider
\begin{equation}
\label{EqnAlt}
\mathcal{H}_A: \widetilde{A}_D=T^{-\phi}\Delta,
\end{equation}
with some constant $\phi>0$ and constant vector $\Delta\in \mathbb{R}^{d}$. Write
\begin{equation*}
    \Delta=(\Delta_1^\top,\cdots.\Delta_k^\top)^\top,
\end{equation*}
where each $\Delta_m\in \mathbb{R}^{d_m}$. The reason why $T^{-\phi}\Delta$ instead of $\Delta$ is considered in (\ref{EqnAlt}) is that we expect the test to be more sensitive as sample size increases. We will see how the value of $\phi$ influences the convergence of $\widehat{U}_T$ in Theorem \ref{AltThm}.

 We still assume $\epsilon_{ti}$'s are i.i.d. sub-Gaussian random variables, and also consider a special case, where $\epsilon_t\sim \mathcal{N}(0,I)$. We compare our result in the Gaussian case to results in Neykov et al.\cite{neykov2018unified}.

First we define the sets $\Omega_0$ and $\Omega_1$ of feasible parameter matrices $A^*$ under $\mathcal{H}_0$ and $\mathcal{H}_A$ respectively. To control the stability of $\{X_t\}$ in model (\ref{ARp2}), we impose the condition:
\begin{equation}\label{StbARpsG}
\sum_{i=0}^{\infty}\left(\sum_{j=0}^{\infty}\left\|\Psi_{i+j}\right\|_2^2\right)^{\frac{1}{2}}\leq  \beta,
\end{equation}
for some constant $\beta>0$. In the case $p=1$, condition (\ref{StbARpsG}) reduces to
\begin{equation}\label{StbAR1W}
\sum_{i=0}^{\infty}\left(\sum_{j=0}^{\infty}\left\|(A^*)^{i+j}\right\|_2^2\right)^{\frac{1}{2}}\leq  \beta,
\end{equation}
which is implied by $\left\|A^*\right\|_2\leq 1-\epsilon$ for some $0<\epsilon<1$, a typical condition assumed (see e.g.~\cite{neykov2018unified}).
Then define sets $\Omega_0$ and $\Omega_1$ for any $\beta, \rho, s, M, T,\phi>0$, set $D$ of size $d$ and vector $\Delta=(\Delta_1^\top,\cdots, \Delta_k^\top)^\top\in \mathbb{R}^d$:
\begin{equation}\label{Omega0_def}
\begin{split}
\Omega_0=\{A^*\in \mathbb{R}^{M\times pM}: \widetilde{A}_D=0, \sum_{i=0}^{\infty}\left(\sum_{j=0}^{\infty}\left\|\Psi_{i+j}\right\|_2^2\right)^{\frac{1}{2}}\leq  \beta,\\ 
    \max_m\rho_m(A^*) \leq \rho, \max_m s_m(A^*)\leq s\},
    \end{split}
  \end{equation}
  \begin{equation}\label{Omega1_def}
  \begin{split}
      \Omega_1=\{A^*\in \mathbb{R}^{M\times pM}: \widetilde{A}_D=T^{-\phi}\Delta,
      \sum_{i=0}^{\infty}\left(\sum_{j=0}^{\infty}\left\|\Psi_{i+j}\right\|_2^2\right)^{\frac{1}{2}}\leq  \beta, \\
\max_m\rho_m(A^*) \leq \rho, \max_m s_m(A^*)\leq s\}.
\end{split}
\end{equation}
 Note here $\rho_m(A^*)$ and $s_m(A^*)$ are still functions of $A^*$, since $\Upsilon$ is determined by $A^*$. 
Clearly we need reliable estimators for $\widehat{A}_m$, $\hat{w}_m$ and $\widehat{\Sigma^{(m)}}$ with $1\leq m\leq k$, to guarantee the weak convergence of $\widehat{U}_T$. We present the following assumptions for these estimators, which we will verify in section \ref{Estimator}. Note that constants $C$ may depend on $p, d, \beta$ and $\tau$, but do not depend on either $M$ or $T$.

 \begin{Assumption}[Estimation Error for $A_m^*$]\label{Assump_A_bound}For each $A^*\in \Omega_0\cup \Omega_1$, 
\begin{equation}\label{Abnd}
	\begin{split}
     &\left\|\widehat{A}_m-A^*_m\right\|_1\leq C\rho_{m}\sqrt{\frac{\log M}{T}}, \quad \left\|\widehat{A}_m-A^*_m\right\|_2\leq C\sqrt{\frac{\rho_{m}\log M}{T}},\\
    &(\widehat{A}_m-A^*_m)^\top \left(\frac{1}{T}\sum_{t=0}^{T-1}\mathcal{X}_t\mathcal{X}_t^\top\right)(\widehat{A}_m-A^*_m)\leq C\frac{\rho_{m}\log M}{T},
	\end{split}
	\end{equation}
	hold for $1\leq m\leq k$, with probability at least $1-c_1\exp\{-c_2\log M\}$.
\end{Assumption}
These are standard error bounds for Lasso estimator and Dantzig Selector with independent design. In this paper we verify Assumption \ref{Assump_A_bound} in section \ref{Estimator} and the remaining two assumptions when we have dependent sub-Gaussian random variables, as we do for our vector auto-regressive model setting. 

\begin{Assumption}[Estimation Error for $w_m^*$]\label{Assump_w_bound}For each $A^*\in \Omega_0\cup \Omega_1$: 
\begin{equation}\label{wbnd}
\begin{split}
\left\|\hat{w}_m-w_m^*\right\|_1\leq Cs_m\sqrt{\frac{\log M}{T}},&\\
\text{tr}\left[(\hat{w}_m-w^*_{m})^\top \left(\frac{1}{T}\sum_{t=0}^{T-1}\mathcal{X}_{t,D_m^c}\mathcal{X}_{t,D_m^c}^\top\right)(\hat{w}_{m}-w^*_{m})\right]&\leq C\frac{s_{m}\log M}{T},
\end{split}
\end{equation}
hold for $1\leq m\leq k$, with probability at least $1-c_1\exp\{-c_2\log M\}$.
\end{Assumption}
Similar to Assumption \ref{Assump_A_bound}, we will show that both Lasso estimator and Dantzig selector under model \eqref{ARp2} satisfy Assumption \ref{Assump_w_bound}.

\begin{Assumption}[Estimation Error for $\Upsilon^{(m)}$]\label{Assump_Cov_m_bound}For each $A^*\in \Omega_0\cup \Omega_1$,
\begin{equation}\label{Cov_mBnd}
\begin{split}
\left\|\Upsilon^{(m)\frac{1}{2}}\left(\widehat{\Upsilon^{(m)}}\right)^{-1}\Upsilon^{(m)\frac{1}{2}}-I\right\|_\infty\leq C\frac{(s\vee\rho)\log M}{\sqrt{T}},
\end{split}
\end{equation}
hold for $1\leq m\leq k$, with probability at least $1-c_1\exp\{-c_2\log M\}$.
\end{Assumption}
Note that $\Upsilon^{(m)}\in \mathbb{R}^{d_m\times d_m}$ is a low-dimensional matrix, and thus it is computationally feasible to use the sample covariance matrix of $X_{t,D_m}-\hat{w}_m^\top X_{t,D_m^c}$ as an estimator for $\widehat{\Upsilon^{(m)}}$. We show in section \ref{Estimator} that, as long as $\hat{w}_m$ is a reliable estimator for $w_m^*$, $\widehat{\Upsilon^{(m)}}$ would satisfy a tighter bound than \eqref{Cov_mBnd}. This looser bound in Assumption~\ref{Assump_Cov_m_bound} actually allows more choices for estimators for $(\Upsilon^{(m)})^{-1}$, as shown in section \ref{CRSec}.

\subsection{Uniform convergence under null hypothesis}\label{NullSec}
Based on these assumptions, we have the following main theorem. 
\begin{Theorem}\label{NullThm}
	Consider the model (\ref{ARp2}) with i.i.d. sub-Gaussian noise $\epsilon_{ti}$ with sub-Gaussian parameter $\tau$.  If Assumptions \ref{Assump_A_bound}-\ref{Assump_Cov_m_bound} are satisfied, and $(\rho\vee s)\log M=o(\sqrt{T})$, then $\widehat{U}_T$ defined in (\ref{TestStatistic}) satisfies
	\begin{equation}\label{thm1_bound}
    \begin{split}
	&\sup_{x\in \mathbb{R}, A^*\in \Omega_0} \left|\mathbb{P}(\widehat{U}_T\leq x)-F_{d}(x)\right|\\
	\leq &\frac{C_1}{T^{\frac{1}{8}}}+C_2\left(\frac{(s\vee \rho)\log M}{\sqrt{T}}\right)^{\frac{1}{2}}+\frac{C_3}{M^{C_4}},
	\end{split}
\end{equation}
	when $T>C$ for some constant $C$. Here the constants $C_i$'s depend on $p, d,\beta,\tau$.
\end{Theorem}

Theorem \ref{NullThm} proves weak convergence of $\widehat{U}_T$ to $\chi_d^2$. The uniform convergence rate can be understood as follows: the first term is due to the rate obtained by martingale CLT, where we require $T^{-\frac{1}{8}}$ rather than $T^{-\frac{1}{2}}$ due to the dependence; the remaining two terms arise from estimation error, with the second one being the error bounds, and third being the probability that the error bounds do not hold. If we assume Gaussianity, we can improve the first term in the rate of convergence from $T^{-\frac{1}{8}}$ to $T^{-\frac{1}{4}+\alpha}$ for any $\alpha>0$. To the best of our knowledge, ours is the first work that formally attempts to characterize the rates of convergence. 

\begin{Remark}
Compared to the theoretical result for independent design in \cite{ning2017general}, the only additional condition we add is $\sum_{i=0}^{\infty}\left(\sum_{j=0}^{\infty}\left\|\Psi_{i+j}\right\|_2^2\right)^{\frac{1}{2}}\leq \beta$, which is used to control the strength of dependence uniformly. Also, we consider multivariate testing which is more general, and derive the explicit convergence rate.  
\end{Remark}

\begin{Remark}
The test statistic proposed in \cite{van2014asymptotically} and \cite{javanmard2014confidence} for the independent design share similar ideas with our test statistic. Instead of imposing a sparsity assumption upon $w_m^*$, \cite{van2014asymptotically} assumes $\Upsilon^{-1}$ to be row wise sparse. This is actually equivalent to the sparsity assumption on $w_m^*$ in the univariate case. \cite{javanmard2014confidence} does not require the sparsity condition on $\Upsilon^{-1}$, but it is hard to extend their theory to the time series setting, due to a difficulty in applying the martingale CLT.  
\end{Remark}

\begin{Remark}
The theoretical guarantee we obtained here, is more general and stronger than the result achieved in \cite{neykov2018unified}. A more detailed comparison is presented in section \ref{GaussAR1}.
\end{Remark}

\subsection{Uniform convergence under alternative hypothesis}\label{AltSec}
Recall the definition of $\Omega_A$ in (\ref{Omega1_def}). The following theorem establishes the asymptotic behavior of $\widehat{U}_T$ for $A^*\in \Omega_A$, with different values of $\phi$. First define
\begin{equation}\label{DeltaARpDef}
\widetilde{\Delta}=(\widetilde{\Delta}_1^\top,\cdots,\widetilde{\Delta}_k^\top)^\top,\quad \widetilde{\Delta}_m=(\Upsilon^{(m)})^{\frac{1}{2}}\Delta_m,
\end{equation}
where $\Upsilon^{(m)}$ is defined in (\ref{Cov_m_def}).
\begin{Theorem}\label{AltThm}
	Consider the model (\ref{ARp2}) with i.i.d. sub-Gaussian noise $\epsilon_{ti}$ and sub-Gaussian parameter $\tau$. If Assumptions \ref{Assump_A_bound}-\ref{Assump_Cov_m_bound} are satisfied, and $(\rho\vee s)\log M=o(\sqrt{T})$, then when $T>C$ for some constant $C$,
	\begin{itemize}
	    \item[(1)] $\phi=\frac{1}{2}$
	\begin{equation}\label{thm2_bound1}
    \begin{split}
	&\sup_{x\in \mathbb{R}, A^*\in \Omega_1} \left|\mathbb{P}(\widehat{U}_T\leq x)-F_{d,\|\widetilde{\Delta}\|_2^2}(x)\right|\\
	\leq &\frac{C_1}{T^{\frac{1}{8}}}+C_2\left(\frac{(s\vee \rho)\log M}{\sqrt{T}}\right)^{\frac{1}{2}}+\frac{C_3}{M^{C_4}}.
	\end{split}
\end{equation}
	\item[(2)] $0<\phi<\frac{1}{2}$
	\begin{equation}\label{thm2_bound2}
    \begin{split}
	&\sup_{A^*\in \Omega_1} |\mathbb{P}(\widehat{U}_T\leq x)|\\
	\leq &\frac{C_1}{T^{\frac{1}{8}}}+\frac{C_2}{M^{C_3}}+C_4\exp\{-C_5T^{\frac{1}{2}-\phi}+C_6\sqrt{x}\}.
	\end{split}
\end{equation}
	\item[(3)] $\phi>\frac{1}{2}$
	\begin{equation}\label{thm2_bound3}
    \begin{split}
	&\sup_{x\in \mathbb{R}, A^*\in \Omega_1} \left|\mathbb{P}(\widehat{U}_T\leq x)-F_{d}(x)\right|\\
	\leq &\frac{C_1}{T^{\frac{1}{8}}}+C_2\left(\frac{(s\vee \rho)\log M}{\sqrt{T}}\right)^{\frac{1}{2}}+\frac{C_3}{M^{C_4}}+C_3T^{\frac{1-2\phi}{3}}.
	\end{split}
\end{equation}
	\end{itemize}
    Here $C_i$'s are constants depending on $p, d,\beta,\Delta,\tau$. 
\end{Theorem}

Theorem \ref{AltThm} shows the threshold value of $\phi$ for $\mathcal{H}_A$ to be detectable. When $\phi>\frac{1}{2}$, we cannot distinguish $\mathcal{H}_0$ and $\mathcal{H}_A$ since under both cases $\widehat{U}_T$ converges to $\chi_d^2$; When $\phi<\frac{1}{2}$, $\widehat{U}_T$ diverges to $+\infty$ in probability, thus it would be very easy to detect $\mathcal{H}_A$; When $\phi=\frac{1}{2}$, $\widehat{U}_T$ converges to a non-central $\chi_d^2$ with noncentrality parameter determined by constant vector $\Delta$ and $\Upsilon=\text{Cov}(\mathcal{X}_t)$, which implies the power of the test. Note here, \eqref{thm2_bound2} holds also for the trivial case $\phi<0$, since we do not use the fact $\phi>0$ in the proof. 

\begin{Remark}
Theorem \ref{AltThm} is also consistent with the threshold value of $\phi$ given by \cite{ning2017general} for linear regression with i.i.d samples. However, \cite{ning2017general} assumes additional conditions on the scaling of sample size, number of covariates and sparsity of $w_m^*$ for proving asymptotic power. Our conditions are exactly the same as the ones for $\mathcal{H}_0$, due to a more specific model and careful analysis. 
\end{Remark}

\subsection{Feasible Estimators}\label{Estimator}
Both the estimation of $w_m^*$ and $A^*$ can be viewed as high-dimensional sparse regression problems, thus we can use the Lasso or Dantzig selector.  Formally, define
\begin{equation}\label{A_Ldef}
    \widehat{A}^{(L)}=\argmin_{A\in \mathbb{R}^{M\times pM}}\frac{1}{T}\sum_{t=0}^{T-1}\|X_{t+1}-A\mathcal{X}_{t}\|_2^2+\lambda_A \|A\|_{1,1},
\end{equation}
as the Lasso estimator for $A^*$, and
\begin{equation}\label{A_Ddef}
    \widehat{A}^{(D)}=\argmin_{A\in \mathbb{R}^{M\times pM}}\|A\|_{1,1},\quad\text{s.t.}\quad \left\|\frac{1}{T}\sum_{t=0}^{T-1}(X_{t+1}-A\mathcal{X}_{t})\mathcal{X}_t^\top\right\|_\infty\leq \lambda_A,
\end{equation}
as the Dantzig selector estimator for $A^*$.
Similarly, for $1\leq m\leq k,$ define
\begin{equation}\label{w_Ldef}
\hat{w}_m^{(L)}=\argmin_{w\in \mathbb{R}^{(pM-d_m)\times d_m}}\frac{1}{T}\sum_{t=0}^{T-1}\|\mathcal{X}_{t,D_m}-w^\top \mathcal{X}_{t,D_m^c}\|_2^2+\lambda_w\|w\|_{1,1}, 
\end{equation}
and
\begin{equation}\label{w_Ddef}
\hat{w}_m^{(D)}=\argmin_{w\in \mathbb{R}^{(pM-d_m)\times d_m}}\|w\|_{1,1},\quad \text{s.t.}\quad \left\|\frac{1}{T}\sum_{t=0}^{T-1} (\mathcal{X}_{t,D_m}-w^\top \mathcal{X}_{t,D_m^c})\mathcal{X}_{t,D_m^c}^\top\right\|_{\infty}\leq \lambda_w.
\end{equation}
While for estimating $\Upsilon^{(m)}$, since this is a low dimensional covariance matrix for $\mathcal{X}_{t,D_m}-w_m^{*\top}\mathcal{X}_{t,D_m^c}$, we can directly use sample covariance of $\mathcal{X}_{t,D_m}-\hat{w}_m^\top \mathcal{X}_{t,D_m^c}$ as $\widehat{\Upsilon^{(m)}}$:
\begin{equation}\label{S_def}
    \widehat{\Upsilon^{(m)}}=\frac{1}{T}\sum_{t=0}^{T-1} (\mathcal{X}_{t,D_m}-\hat{w}_m^\top \mathcal{X}_{t,D_m^c})(\mathcal{X}_{t,D_m}-\hat{w}_m^\top \mathcal{X}_{t,D_m^c})^\top,
\end{equation}
for $1\leq m\leq k$. Here $\hat{w}_m$ in the definition of (\ref{S_def}) is either $\hat{w}_m^{(L)}$ or $\hat{w}_m^{(D)}$.

As shown in the following, estimators (\ref{A_Ldef}) to (\ref{S_def}) all satisfy Assumptions \ref{Assump_A_bound} to \ref{Assump_Cov_m_bound}, under the model setting stated in (\ref{ARp2}): 
\begin{Lemma}\label{Abound}
	If $\widehat{A}=\widehat{A}^{(L)}$, or $\widehat{A}=\widehat{A}^{(D)}$, which are defined as in (\ref{A_Ldef}) and (\ref{A_Ddef}) with $\lambda_A\asymp\sqrt{\frac{\log M}{T}}$, then $\widehat{A}$ satisfies Assumption \ref{Assump_A_bound} when $T>C\rho \log M$.
\end{Lemma}
\begin{Lemma}\label{wbound}
If $\hat{w}_m=\hat{w}_m^{(L)}$ or $\hat{w}_m=\hat{w}_m^{(D)}$, which are defined as in (\ref{w_Ldef}) and (\ref{w_Ddef}) with $\lambda_w\asymp \sqrt{\frac{\log M}{T}}$, then $\hat{w}_m$'s satisfy Assumption \ref{Assump_w_bound} when $T>Cs \log M$.
\end{Lemma}
\begin{Lemma}\label{Cov_mBound}
	If $\widehat{\Upsilon^{(m)}}$'s are defined as in (\ref{S_def}), where $\hat{w}_m$ satisfies (\ref{wbnd}) with probability at least $1-c_1\exp\{-c_2\log M\}$, then 
	\begin{equation*}
\begin{split}
\left\|\Upsilon^{(m)\frac{1}{2}}\left(\widehat{\Upsilon^{(m)}}\right)^{-1}\Upsilon^{(m)\frac{1}{2}}-I\right\|_\infty\leq C\sqrt{\frac{\log M}{T}},
\end{split}
\end{equation*}
with probability at least $1-c_1\exp\{-c_2\log M\}$, when $T>Cs^2 \log M$.
\end{Lemma}
Note here Lemma \ref{Cov_mBound} is stronger than Assumption \ref{Assump_Cov_m_bound}.
The proof of these Lemmas are deferred to Appendix \ref{lemma_proof}. By these lemmas and Theorem \ref{NullThm}, \ref{AltThm}, we arrive at following Corollary.
\begin{Corollary}\label{Cor}
Under model (\ref{ARp2}) with i.i.d sub-Gaussian noise $\epsilon_{ti}$ with parameter $\tau$, if $\widehat{A}=\widehat{A}^{(L)}$ or $\widehat{A}^{(D)}$, $\hat{w}_m=\hat{w}_m^{(L)}$ or $\hat{w}_m^{(D)}$, and $\widehat{\Upsilon^{(m)}}$'s are defined as in (\ref{S_def}) for $1\leq m\leq k$ with $\lambda_A\asymp\lambda_w\asymp \sqrt{\frac{\log M}{T}}$, then if $(\rho\vee s)\log M=o(\sqrt{T})$ and $T > C$ for some constant $C>0$, bounds (\ref{thm1_bound}) to (\ref{thm2_bound3}) from Theorems \ref{NullThm} and \ref{AltThm} hold.
\end{Corollary}

\subsection{Variance Estimation}\label{EstVarSec}
In this section, we consider the case where $\sigma^{*2}=\text{Var}(\epsilon_{ti})$ is unknown under model \eqref{ARp2}. Actually, if $\sigma^*\neq 1$ is known, it is straightforward to extend Theorem \ref{NullThm} to Theorem \ref{AltThm} for $\widehat{U}_T$ defined as follows:
\begin{equation}\label{U_T_knv}
\widehat{U}_T=T\sum_{m=1}^k \widehat{S}_m^\top (\widehat{\Upsilon^{(m)}})^{-1}\widehat{S}_m/\sigma^{*2}.
\end{equation}
This follows since if we consider $Y_t=X_t/\sigma^*$, time series data $Y_t$ would satisfy the same model but with unit variance noise. 

When $\sigma^{*2}$ is unknown, we apply the estimator
\begin{equation}\label{VarEst}
\hat{\sigma}^2=\frac{1}{MT}\sum_{t=0}^{T-1}\|X_{t+1}-\widehat{A}\mathcal{X}_t\|_2^2,
\end{equation}
and define the test statistic 
\begin{equation}\label{EstVarTS}
\widetilde{U}_T=T\sum_{m=1}^k\widehat{S}_m^\top(\widehat{\Upsilon^{(m)}})^{-1}\widehat{S}_m/\hat{\sigma}^2.
\end{equation}
We show that $\widetilde{U}_T$ has the same convergence results we derive for the unit variance noise case.
\begin{Theorem}\label{EstVarThm}
Consider the model \eqref{ARp2} with i.i.d. sub-Gaussian noise $\epsilon_{ti}$ of variance $\sigma^{*2}=\text{Var}(\epsilon_{ti})\geq\sigma_0^2>0$ and scale factor $\tau\sigma^*$. Then Theorem \ref{NullThm} and \ref{AltThm} hold for $\widetilde{U}_T$ under each corresponding condition, and constants $C_i$'s also depend on $\sigma_0$.
\end{Theorem}

Theorem \ref{EstVarThm} shows that when we have to estimate the unknown $\sigma^{*2}$, test statistic $\widetilde{U}_T$ maintains the same asymptotic behavior as $\widehat{U}_T$ under the known variance case, given that all the assumptions for estimation errors are satisfied and $\sigma^*$ is lower bounded by some constant.
\begin{Remark}
With sub-Gaussian noise $\epsilon_{ti}$, if we still assume the scale factor $\tau\sigma^*$ of $\epsilon_{ti}$ to be bounded by constant, then Lemma \ref{Abound} to \ref{Cov_mBound} would still hold. Thus the assumptions imposed on estimation errors of $\widehat{A}$, $\hat{w}_m$ and $\widehat{\Upsilon^{(m)}}$ are all satisfied. However, if we don't assume $\sigma^*$ to be bounded, then the tuning parameters $\lambda_A$ and $\lambda_w$  have to scale with $\sigma^*$. 
\end{Remark}
\begin{Remark}
\cite{neykov2018unified} proposes another estimator for the variance of $\epsilon_{ti}$, based on the fact that $\Sigma=A\Sigma A^\top+Cov(\epsilon_{t})$. Both these estimators are consistent and lead to convergence in distribution results.
\end{Remark}

\subsection{Semi-parametric Optimal Confidence Region}\label{CRSec}
In this section, we construct a confidence region for $\widetilde{A}_D$, under model \eqref{ARp2} with unknown noise variance $\sigma^{*2}$. Similar to \cite{ning2017general}, we consider the one-step estimator $\hat{a}(m)$ for each $(A_m^*)_{D_m}$, based on the decorrelated score function:
\begin{equation}
\hat{a}(m)=(\widehat{A}_{m})_{D_m}-\left(\widetilde{\Upsilon^{(m)}}\right)^{-1}\widetilde{S}_m,
\end{equation}
where $\widehat{A}_{m}$ is any estimator satisfying the Assumptions \ref{Assump_A_bound} on error bounds for $\widehat{A}_m-A_m^*$, and both the Lasso or Dantzig Estimator for $A_{m}^*$ are suitable. $\widetilde{\Upsilon^{(m)}}$ takes the form:
\begin{equation}
\widetilde{\Upsilon^{(m)}}=\frac{1}{T}\sum_{t=0}^{T-1}\left(\mathcal{X}_{t,D_m}-\hat{w}_m^\top \mathcal{X}_{t,D_m^c}\right)\mathcal{X}_{t,D_m}^\top,
\end{equation}
which is another estimator for $\Upsilon^{(m)}$, and 
\begin{equation*}
\widetilde{S}_m=-\frac{1}{T}\sum_{t=0}^{T-1}\left(X_{t+1,m}-\widehat{A}_m^\top \mathcal{X}_t\right)\left(\mathcal{X}_{t,D_m}-\hat{w}_m^\top \mathcal{X}_{t,D_m^c}\right).
\end{equation*}
We will show that $\hat{a}(m)-(A^*_m)_{D_m}$ is asymptotically Gaussian with covariance matrix $(\Upsilon^{(m)})^{-1}$. Thus we construct the following confidence region for $\widetilde{A}_D$, with asymptotic confidence coefficient $1-\alpha$:
\begin{equation}\label{ConfRegDef}
\begin{split}
CR(\alpha)=\bigg\{&\theta=(\theta_1^\top,\dots,\theta_k^\top)^\top: \theta_m\in \mathbb{R}^{d_m}, \\
&\frac{T}{\hat{\sigma}^2}\sum_{m=1}^{k}(\hat{a}(m)-\theta_m)^\top \widehat{\Upsilon^{(m)}}(\hat{a}(m)-\theta_m)\leq \chi^2_{d}(1-\alpha)\bigg\}.
\end{split}
\end{equation}
This is a $d$ dimensional elliptical ball with center vector $(\hat{a}(1)^\top, \dots \hat{a}(k)^\top)^\top$. The following theorem shows the weak convergence result of  
\begin{equation}\label{CR_TS}
    \widehat{R}_T\triangleq\frac{T}{\hat{\sigma}^2}\sum_{m=1}^{k}(\hat{a}(m)-(A_m^*)_{D_m})^\top \widehat{\Upsilon^{(m)}}(\hat{a}(m)-(A_m^*)_{D_m}).
\end{equation}
\begin{Theorem}\label{ConfRegThm}
Under model \eqref{ARp2} with i.i.d. sub-Gaussian noise $\epsilon_{ti}$ with variance $\sigma^{*2}=\text{Var}(\epsilon_{ti})\geq\sigma_0^2>0$ and sub-Gaussian parameter $\tau\sigma^*$, then Theorem \ref{NullThm} and \ref{AltThm} hold for $\widehat{R}_T$ under each corresponding condition, and the constants $C_i$'s also depend on $\sigma_0$.
\end{Theorem}

\begin{Remark}
In the definition of  one-step estimator $\hat{a}(m)$, we use $\widetilde{\Upsilon^{(m)}}$ instead of $\widehat{\Upsilon^{(m)}}$ for theoretical convenience. Theorem \ref{ConfRegThm} would still hold true if $\hat{a}(m)$ is defined as $(\widehat{A}_{m})_{D_m}-\left(\widehat{\Upsilon^{(m)}}\right)^{-1}\widetilde{S}_m$.
\end{Remark}

\begin{Remark}
We have exactly the same theoretical result for $\widetilde{U}_T$ and $\widehat{R}_T$, and this is due to the close relationship between these two quantities. In particular,
\begin{equation*}
\widehat{R}_T=T\sum_{m=1}^k\widehat{S}_m^\top\left(\widetilde{\Upsilon^{(m)}}^\top\right)^{-1}\widehat{\Upsilon^{(m)}}\left(\widetilde{\Upsilon^{(m)}}\right)^{-1}\widehat{S}_m/\hat{\sigma}^2,
\end{equation*}
compared to $\widetilde{U}_T=T\sum_{m=1}^k\widehat{S}_m^\top(\widehat{\Upsilon^{(m)}})^{-1}\widehat{S}_m/\hat{\sigma}^2.$ We show in the proof of Theorem \ref{ConfRegThm} that $\left(\widetilde{\Upsilon^{(m)}}^\top\right)^{-1}\widehat{\Upsilon^{(m)}}\left(\widetilde{\Upsilon^{(m)}}\right)^{-1}$ also satisfies Assumption \ref{Assump_Cov_m_bound} as an estimator for $\left(\Upsilon^{(m)}\right)^{-1}$.
\end{Remark}

\begin{Remark}
The one-step estimator $\hat{a}(m)$ is asymptotically unbiased, and shares a similar form to the de-biased estimator proposed by \cite{zhang2014confidence}, \cite{van2014asymptotically}. The de-biased estimator in \cite {van2014asymptotically} would take the following form under our setting:
\begin{equation*}
\widehat{b}_m=(\widehat{A}_m)_{D_m}+\widehat{\Theta}_{D_m,\cdot}\frac{1}{T}\sum_{t=0}^{T-1}
\mathcal{X}_t\left(X_{t+1,m}-\mathcal{X}_t^\top \widehat{A}_m\right),
\end{equation*}
where $\widehat{\Theta}$ is computed by node-wise regression, as an estimator for $\Upsilon^{-1}$. When $d_m=|D_m|=1$, this is essentially the same as our estimator $\hat{a}(m)$, but would be slightly different in the multivariate case. Note that the asymptotic covariance matrix for $\hat{a}(m)$ equals to the partial information matrix $I^*(A_{m,D_m}|A_{m,D_m^c}$), and thus is semi-parametric efficient, while $\hat{b}_m$ is only efficient when it is a scalar.
\end{Remark}

\begin{Remark}
$\widehat{R}_T$ is also very similar to the test statistic proposed by \cite{neykov2018unified} for VAR model with lag 1. The only difference lies in the estimation of Var$(\epsilon_{ti})$, and they only consider Dantzig selector for estimating $A^*$ and $w_m^*$. We will provide a detailed comparison between their theoretical result with ours in section \ref{GaussAR1}.
\end{Remark}

\subsection{Special case: AR(1) with Gaussian noise}\label{GaussAR1}
Our theoretical guarantee covers VAR models with lag $p$ and sub-Gaussian noise, of which AR(1) model and Gaussian noise are special cases. Here we explain the consequences of our result under this special case and provide comparison with \cite{neykov2018unified}. 

When we consider lag $p=1$, the constraint for $A^*$ becomes
\begin{equation*}
\sum_{i=0}^{\infty}\left(\sum_{j=0}^{\infty}\left\|(A^*)^{i+j}\right\|_2^2\right)^{\frac{1}{2}}\leq  \beta, \max_m\rho_m(A^*) \leq \rho, \max_m s_m(A^*)\leq s,
\end{equation*}
with $(\rho\vee s)\log M=o(\sqrt{T})$.
The two sparsity conditions and sample size requirement are included in the conditions \cite{neykov2018unified} proposes. In addition, they assume the following:

\begin{equation*}
    \|A^*\|_1\leq C, \|A^*\|_2\leq 1-\varepsilon, \left\|\Sigma^{-1}\right\|_1\leq C.
\end{equation*}
for some $0<\varepsilon<1$. Note that we don't require these conditions, among which the first and third are quite strong, and the second one $\|A^*\|_2\leq 1-\varepsilon$ is sufficient for our condition $\sum_{i=0}^{\infty}\left(\sum_{j=0}^{\infty}\left\|(A^*)^{i+j}\right\|_2^2\right)^{\frac{1}{2}}\leq  \beta$. This follows since if $\|A^*\|_2\leq 1-\varepsilon$,
\begin{equation*}
    \sum_{i=0}^{\infty}\left(\sum_{j=0}^{\infty}\left\|(A^*)^{i+j}\right\|_2^2\right)^{\frac{1}{2}}\leq \sum_{i=0}^{\infty}\left(\sum_{j=0}^{\infty}\left\|A^*\right\|_2^{2(i+j)}\right)^{\frac{1}{2}}\leq \frac{\sum_{i=0}^{\infty}(1-\varepsilon)^i}{\sqrt{1-(1-\varepsilon)^2}} \leq \left(2\varepsilon-\varepsilon^2\right)^{-\frac{1}{2}}.
\end{equation*}

Until now the discussion focuses on the case where $\epsilon_{ti}$ are i.i.d. sub-Gaussian noise of scale factor $C\sigma^*$, with $(\sigma^*)^2$ being the variance of $\epsilon_{ti}$ and lower bounded by some constant. Thus our setting covers the case where $\epsilon_t\sim \mathcal{N}(0,(\sigma^*)^2I)$ with $\sigma^*\geq c$. If $\epsilon_t\sim \mathcal{N}(0,\Psi)$ with $\Psi_{ii}\geq c$ as assumed in \cite{neykov2018unified}, we can still prove the same theoretical guarantee, under even weaker condition based on spectral density, due to established concentration bounds in \cite{basu2015regularized}.  

\section{Numerical Experiments}\label{SimSec}
 
In this section, we provide a simulation study to validate our theoretical results. For simplicity, our simulation is based on the AR(1) model:
\begin{align}
X_{t+1}=A^*X_t+\epsilon_{t},\quad t=0,\dots,T,
\end{align} 
where $A^*\in \mathbb{R}^{M\times M}$ is set to be row-wise sparse. Symmetricity is not required in our theory, but in order to ensure the sparsity of $w_m^*$, we focus on symmetric matrices under $\mathcal{H}_0$, and slightly asymmetric ones under $\mathcal{H}_A$. The eigenvalues of $A^*$ all fall in the unit circle of the complex plane, which ensures the existence of stationary solution to this model. White noise $\epsilon_{ti}$ is simulated as independent $\mbox{Uniform}(-1,1)$ in order to satisfy the sub-Gaussianity condition. Other distributions were also used but not reported since the results were very similar.

To consider multi-variate test sets, throughout the simulation we test the index set $D$ with $d=|D|=6$, which involves three different rows and two columns in each row: $$D=\{(1,3),(1,5),(3,3),(3,4),(5,4),(5,8)\}.$$
The null hypothesis takes the form $\mathcal{H}_0: \widetilde{A}_D=\mu$ with some $d$-dimensional vector $\mu$. Correspondingly, we consider alternative hypothesis $\mathcal{H}_A:\widetilde{A}_D=\mu+T^{-\phi}\Delta$, with $\Delta$ randomly selected from $d$-dimensional Gaussian distribution, and $\phi$ ranges from $0.25$ to $1.2$. 

Under $\mathcal{H}_0$, we generate $A^*$ with different row-wise sparsity levels and structures, and for each $A^*$, vector $\mu$ may differ depending on the corresponding $\widetilde{A}_D$. Under $\mathcal{H}_A$, $A^*$ are still the same matrices as under $\mathcal{H}_0$, but only adding the tested indices $\widetilde{A}_D$ by $T^{-\phi}\Delta$. The experiments are repeated under different settings of $A^*$, $\Delta$, $M, T$ and $\phi$. 

We use Lasso estimators defined in \eqref{A_Ldef}, \eqref{w_Ldef} for the estimation of $A^*$ and $w_m^*$, $1\leq m\leq k$, and tuning parameters $\lambda_A$, $\lambda_w$ are selected using cross validation. In cross validation, the training sets are composed of consecutive time series data, with the remaining 10\% of the original data set being testing sets. Under $\mathcal{H}_0$, 1000 simulations are carried out under each parameter setting, while under $\mathcal{H}_A$, we have 100 simulations. In the following sections, we look into false positive rates (FPR) and true positive rates (TPR) of test statistics $\widetilde{U}_T$ and $\widehat{R}_T$ as defined in \eqref{EstVarTS} and \eqref{CR_TS}, when we set the level of test as $\alpha=0.05$.

\subsection{Under the Null Hypothesis}\label{H0sim}
\begin{itemize}[leftmargin=6mm]
	\item[(1)] Varying sparsity\\
	\begin{figure}
		\centering
		\includegraphics[width=0.6\textwidth]{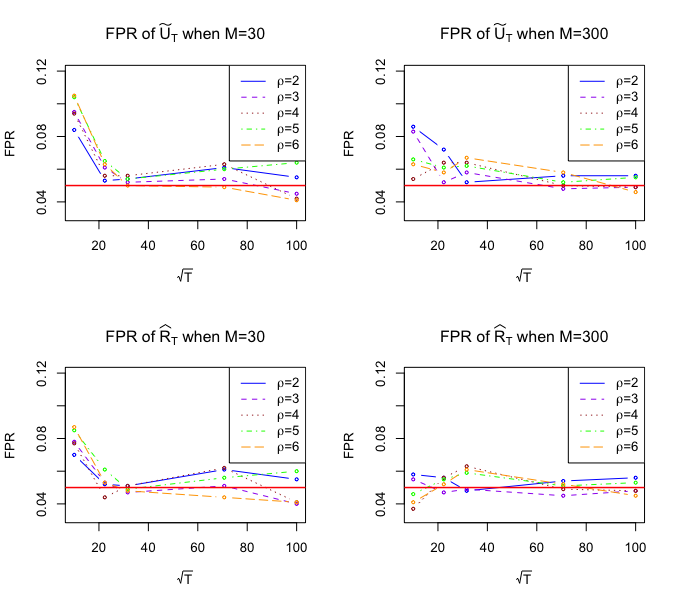}
		\caption{False positive rate (FPR) of $\widetilde{U}_T$ and $\widehat{R}_T$ v.s.$\sqrt{T}$, with various dimension $M$ and sparsity level $\rho$. The red line is the significance level $\alpha=0.05$.}
		\label{H0vsN01}
	\end{figure}
	Here we summarize the experiments with randomly generated $A^*$, that are symmetric and row-wise sparse, with different sparsity levels $\rho$ defined in \eqref{sparsity_def}. Figure \ref{H0vsN01} shows how FPR of $\widetilde{U}_T$ and $\widehat{R}_T$ averaged over 1000 experiments vary with $\sqrt{T}$. We can see that when $T$ increases to about 500, the FPR becomes stable and close to $\alpha=0.05$ regardless of $\rho, M$, choice between $\widetilde{U}_T$ and $\widehat{R}_T$. 

	When the sample size $T$ is small, the test tends to be conservative, which is the consequence of estimating variance $\sigma^{*2}$ and covariances $\Upsilon^{(m)}$'s. In the simulation we use naive estimators for these two quantities, as defined in \eqref{VarEst} and \eqref{S_def} which tend to be smaller than the true parameters. This is because we usually fit noise in the regression, as noticed by \cite{fan2012variance}. As shown in these two figures, $\widehat{R}_T$ is less conservative than $\widetilde{U}_T$ when $T$ is small, since the magnitude of $\widetilde{\Upsilon}^{(m)}$ is larger than $\widehat{\Upsilon}^{(m)}$, which makes $\left(\widetilde{\Upsilon}^{(m)^\top}\right)^{-1}\widehat{\Upsilon}^{(m)}\left(\widetilde{\Upsilon}^{(m)}\right)^{-1}$ probably a better estimator for $\Upsilon^{(m)}$. 
		\begin{figure}
			\centering
			\includegraphics[width=0.6\textwidth]{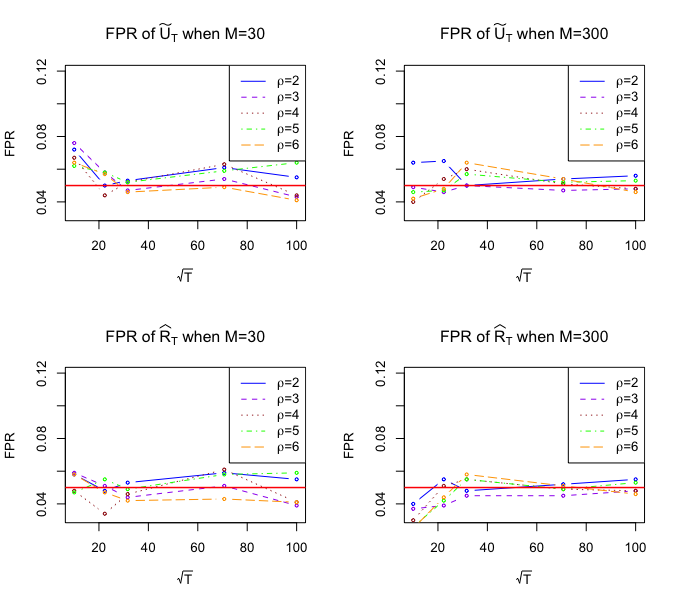}
			\caption{FPR of $\widetilde{U}_T$ and $\widehat{R}_T$ when residual variance is known.}
			\label{H0vsN12}
		\end{figure}
	We also summarize the FPR when the variance $\sigma^{*2}$ of $\epsilon_{ti}$ is known in Figure \ref{H0vsN12}. We can see from these figures that $\widehat{U}_T$ is still a little conservative when $T$ is small, while $\widehat{R}_T$ with $\hat{\sigma}^2$ substituted by $\sigma^{*2}$ is not conservative. 

	\item[(2)] Different Graph Structures\\
	If we consider the $M$ actors in the time series as nodes in a network, and a nonzero $A_{ij}^*$ represents an directed edge from $j$ to $i$, then each matrix $A^*$ corresponds to a $M$-dimensional directed graph. We experiment with different structures of $A^*$, which also correspond to different graph structure, including block graph or chain graph. Specifically, we consider matrices with $\ell_2$ norm equal to 0.75:
	\begin{equation*}
	A^{(1)}=\begin{pmatrix}
	1/4 &1/2&0&0&\cdots&0&0\\
	1/2 & 1/4&0&0&\cdots&0&0\\
	0&0&1/4 &1/2&\cdots&\vdots&\vdots\\
	0&0&1/2 & 1/4&\cdots&\vdots&\vdots\\
	\vdots&\vdots&\ddots&\ddots&\ddots&\vdots&\vdots\\
	0&0&\cdots&\cdots&\cdots&1/4&1/2\\
	0&0&\cdots&\cdots&\cdots&1/2&1/4
	\end{pmatrix},
	\end{equation*}
	which is a block graph;
	\begin{equation*}
	A^{(2)}=\begin{pmatrix}
	c & c & 0 &\cdots &\cdots &0\\
	c & 0 & c &\cdots &\cdots &0\\
	0 &c &0 &c &\cdots &0\\
	\vdots &\ddots& \ddots &\ddots &\ddots &\vdots\\
	0&\cdots&\cdots&c&0&c\\
	0&\cdots&\cdots&\cdots&c&0
	\end{pmatrix},
	\end{equation*}
	with constant $c$ chosen to ensure $\left\|A^{(2)}\right\|_2=0.75$, which is a chain graph; and $A^{(3)}$ being randomly generated symmetric matrix of sparsity level $\rho=2$, and largest eigenvalue equal to 0.75. Figure \ref{H0vsst01} shows the difference among these three different structures. We can see that block graph is less accurate than the other two, which is due to a larger variance for each $X_{t,D_m}-w_m^{*\top}X_{t,D_m^c}$. Investigating the question of how graph structure theoretically influences testing performance remains an open and interesting direction.
	\begin{figure}
		\centering
		\includegraphics[width=0.7\textwidth]{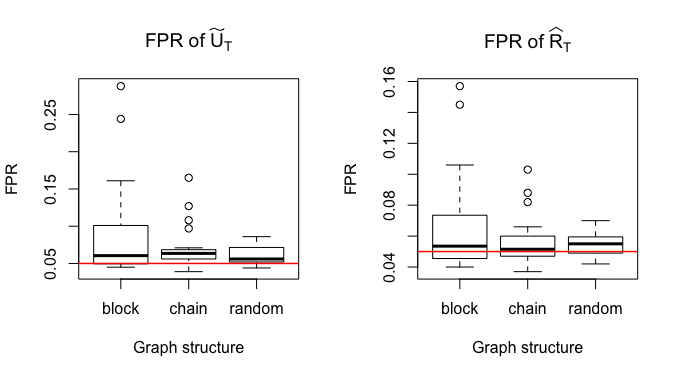}
		\caption{FPR under different graph structure. Block refers to $A^{(1)}$, chain refers to $A^{(2)}$ and random refers to $A^{(3)}$.}
		\label{H0vsst01}
	\end{figure}
\end{itemize}

\subsection{Alternative Hypothesis}\label{HAsim}
    First we look into how the true positive rate (TPR) varies with $\|T^{-\phi}\Delta\|_2$, since we set $\mathcal{H}_A$ as $\widetilde{A}_D=\mu+T^{-\phi}\Delta$ and $\|T^{-\phi}\Delta\|_2$ may be viewed as a measure of distance from the null hypothesis. Fig.~\ref{H1vsdelta11} only presents the simulation results when $A^*=A^{(1)}$ and $M=300$, while the other choices of $A^*$ and $M$ generate very similar results. We can see from these two figures that as $\|T^{-\phi}\Delta\|_2$ increases, TPR approaches 1. The slope increases when sample size $T$ gets larger, or when the test statistic changes from $\widehat{R}_T$ to $\widetilde{U}_T$. This aligns with intuition, since when $T$ increases, we are supposed to distinguish between $\mathcal{H}_0$ and $\mathcal{H}_A$ better, and $\widetilde{U}_T$ is more conservative than $\widehat{R}_T$ as we show in subsection \ref{H0sim}.
    \begin{figure}
        \centering
        \includegraphics[width=0.7\textwidth]{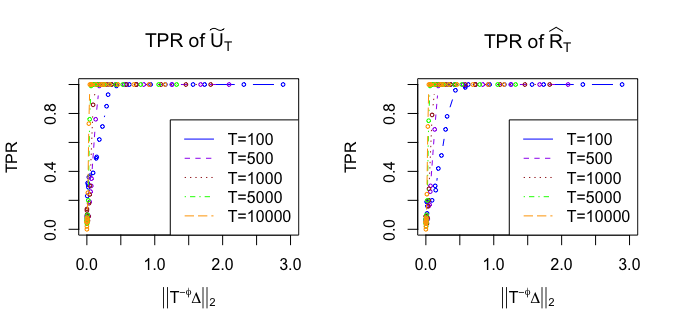}
        \caption{True positive rate of $\widetilde{U}_T$ and $\widehat{R}_T$, when $A^*=A^{(1)}$ and $M=300$}
        \label{H1vsdelta11}
    \end{figure}

We also check the influence of $\phi$. Figure \ref{H1vsphi01} reveals how TPR changes when $T$ increases, if we set $\left\|\widetilde{\Delta}\right\|_2$ and $\phi$ fixed. If $\phi<0.5$, TPR converges to 1 very quickly, while if $\phi>0.5$, TPR converges to 0.05, but the convergence is slower when $\phi$ or $\left\|\widetilde{\Delta}\right\|_2$ increases. When $\phi=0.5$, Theorem \ref{EstVarThm} and \ref{ConfRegThm} states that $\widetilde{U}_T$ and $\widehat{R}_T$ would converge to $\chi_{d,\left\|\widetilde{\Delta}\right\|_2^2}$, thus the TPR should converge to some value between 0.05 and 1, depending on $d$ and $\left\|\widetilde{\Delta}\right\|_2^2$. The black lines in figure \ref{H1vsphi01} indicate this convergence value, but since the test tends to be conservative when $T$ is not large enough, TPR when $\phi=0.5$ is usually above the black line. The conservative issue is more severe under $\mathcal{H}_A$ since the deviation $\widetilde{\Delta}$ is also multiplied by the estimated variances, which exaggerates the conservative tendency. However, this may not be a big concern under $\mathcal{H}_A$, since we always want the TPR to be large.
\begin{figure}
    \centering
    \includegraphics[width=0.7\textwidth]{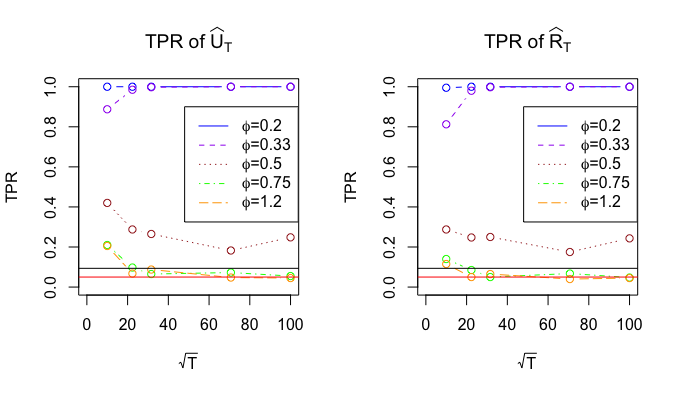}
    \caption{TPR of $\widetilde{U}_T$ and $\widehat{R}_T$ when $\left\|\tilde{\Delta}\right\|_2=1$, $A^*=A^{(1)}$. Results for different graph size $M$ from 30 to 300 are combined together and average TPR is taken. Red line is significance level $\alpha$, the value that TPR should converge to when $\phi<0.5$; while the black line is the convergence point specified in Theorem \ref{AltThm} when $\phi=0.5$.}
    \label{H1vsphi01}
\end{figure}

\section{Proof Overview}\label{main_proof}

One of the main contributions of this work is the proof technique, which addresses a number of technical challenges and develops novel concentration bounds for dependent sub-Gaussian random vectors. In this section, we present and discuss key lemmas for the proof and provide the main steps for proving Theorems \ref{NullThm} and \ref{AltThm}, deferring the more technically intensive steps to the supplement. 
\subsection{Key Lemmas}\label{KeyLemma}
The major technical challenge lies in proving the following two concentration bounds for dependent sub-Gaussian random vectors.
\begin{Lemma}[Deviation Bound for $A^*$]\label{DevBnd}
Under model \eqref{ARp2}, when $\epsilon_{ti}$ are sub-Gaussian noise with scale factor $\tau$, and $A^*\in \Omega_0\cup \Omega_1$,
\begin{equation*}
    \mathbb{P}\left(\left\|\frac{1}{T}\sum_{t=0}^{T-1}\epsilon_{t}\mathcal{X}_t^\top\right\|_\infty>C\sqrt{\frac{\log M}{T}}\right)\leq c_1\exp\{-c_2\log M\},
\end{equation*}
When $T\geq C\log M$.
\end{Lemma}
Lemma \ref{DevBnd} is a standard deviation bound for proving estimation error bound of Lasso type or Dantzig selector type estimators. We apply this lemma both in the proof of Theorem \ref{NullThm}, \ref{AltThm} and Lemma \ref{Abound}.
\begin{Lemma}\label{QuaBound}
	    Under model \eqref{ARp2}, when $\epsilon_{ti}$ are sub-Gaussian noise with constant scale factor $\tau$, and $A^*\in \Omega_0\cup \Omega_1$, if $B\in \mathbb{R}^{pM\times pM}$ is a symmetric matrix, we have
   	\begin{align*}
   	\mathbb{P}\left(\left|\frac{1}{T}\sum_{t=0}^{T-1}\mathcal{X}_t^\top B\mathcal{X}_t-\text{tr}(B\Upsilon)\right|>\delta\right)\leq c_1\exp\left\{-c_2T\min\left\{\frac{\delta}{\|B\|_2},\frac{\delta^2}{\|B\|_{\tr}\|B\|_2}\right\}\right\}.
   	\end{align*}
\end{Lemma}
Lemma \ref{QuaBound} provides concentration bound for the sample average of general quadratic form $\mathcal{X}_t^\top B\mathcal{X}_t$, and is very helpful in proving martingale CLT under our setting, REC, Lemma \ref{Cov_mBound}, etc.

In the Gaussian case, both these lemmas follow from prior work in~\cite{basu2015regularized} which relies on the fact that dependent Gaussian vectors can be rotated to be independent. Since dependent sub-Gaussian random variables cannot be rotated to be independent (only uncorrelated), we exploit the independence of $\epsilon_t$ by representing each $\mathcal{X}_t$ by linear function of the infinite series $\{\epsilon_i\}_{i=-\infty}^{i=t}$ and then use a careful truncation argument. We analyze sufficiently many terms in the summation, and control the infinite residues. 

\subsection{Proof of Theorem \ref{NullThm}}
\begin{proof}
Suppose $A^*\in \Omega_0$. We will use $C_i,c_i$ to refer to constants that only depend on $p,d,\beta,\tau$ (not $M$ or $T$), and different constants might share the same notation.

The proof can be divided into two major parts: showing the convergence of $U_T$ to $\chi_d^2$, and bounding the estimation error $\left|\widehat{U}_T-U_T\right|$. Formally, for any $\varepsilon>0$,
\begin{equation*}
\begin{split}
&\mathbb{P}(\widehat{U}_T\leq x)-F_{d}(x)\\
\leq &\mathbb{P}(U_T\leq x+\varepsilon)+\mathbb{P}(\left|\widehat{U}_T-U_T\right|>\varepsilon)-F_{d}(x)\\
	\leq &\left|\mathbb{P}(U_T\leq x+\varepsilon)-F_{d}(x+\varepsilon)\right|+F_{d}(x+\varepsilon)-F_{d}(x)+\mathbb{P}\left(\left|\widehat{U}_T-U_T\right|>\varepsilon\right),
\end{split}
\end{equation*}
	and
	\begin{equation*}
	\begin{split}
	    &F_{d}(x)-\mathbb{P}(\widehat{U}_T\leq x)\\
	    =&\mathbb{P}(\widehat{U}_T> x)-(1-F_{d}(x))\\
	\leq &\mathbb{P}(U_T> x-\varepsilon)+\mathbb{P}(\left|\widehat{U}_T-U_T\right|>\varepsilon)-1+F_{d}(x)\\
	\leq &\left|F_{d}(x-\varepsilon)-\mathbb{P}\left(U_T\leq x-\varepsilon\right)\right|+F_{d}(x)-F_{d}(x-\varepsilon)+\mathbb{P}\left(\left|\widehat{U}_T-U_T\right|>\varepsilon\right),
	\end{split}
	\end{equation*}
	which implies
	\begin{equation}\label{main_prob_diff}
	\begin{split}
	    &\left|\mathbb{P}(\widehat{U}_T\leq x)-F_{d}(x)\right|\\
	\leq &\sup_{y\in \mathbb{R}}\left|\mathbb{P}(U_T\leq y)-F_{d}(y)\right|+F_d(x+\varepsilon)-F_d(x-\varepsilon)+\mathbb{P}\left(\left|\widehat{U}_T-U_T\right|>\varepsilon\right).
	\end{split}
	\end{equation}
	
	In the following, we provide bounds on each of the three terms. The following lemma shows the uniform weak convergence rate of $\left\|V_T+\mu\right\|_2^2$ to $\chi^2_{d,\|\mu\|_2^2}$, of which the convergence of $U_T=\|V_T\|_2^2$ to $\chi_d^2$ is a special case.
	\begin{Lemma}[Convergence Rate of $\left\|V_T+\mu\right\|_2^2$]\label{CLT}
	Under model (\ref{ARp2}) with $\epsilon_{ti}$ being sub-Gaussian noise of scale factor $\tau$, then for any $A^*\in \Omega_0$, $\forall \mu \in \mathbb{R}^d$, 
	\begin{equation}\label{CLT_boundsG}
	\sup_{x\in \mathbb{R}}\left|\mathbb{P}(\|V_T+\mu\|_2^2\leq x)-F_{d,\|\mu\|_2^2}(x)\right|\leq C(\|\mu\|_2)T^{-\frac{1}{8}},
	\end{equation}
	when $T>C$ for some absolute constant $C$, where $C(\|\mu\|_2)$ is a constant depending on and is non-decreasing with respect to $\|\mu\|_2$. 
\end{Lemma} 
This Lemma is proved in section \ref{lemma_proof2}, by applying a uniform martingale central limit theorem result. 
Thus, by Lemma \ref{CLT}, if $T>C$ for some constant $C$,
\begin{equation*}
    \sup_{y\in \mathbb{R}}\left|\mathbb{P}(U_T\leq y)-F_{d}(y)\right|\leq CT^{-\frac{1}{8}}.
\end{equation*}
Meanwhile,
	$$F_d(x+\varepsilon)-F_d(x-\varepsilon)\leq C_2\varepsilon
	$$ since $\chi_d^2$ has bounded density. 
	
	Now we only need to choose a proper $\varepsilon$ and bound $\mathbb{P}\left(\left|\widehat{U}_T-U_T\right|>\varepsilon\right)$. 
	\begin{equation}\label{EstErr}
	    \begin{split}
	        \left|\widehat{U}_T-U_T\right|\leq &\sum_{m=1}^k\left|T\widehat{S}_m^\top(\widehat{\Upsilon^{(m)}})^{-1}\widehat{S}_m-\left\|V_{T,m}\right\|_2^2\right|\\
	\leq &\sum_{m=1}^k\Bigg|T\widehat{S}_m^\top\left((\widehat{\Upsilon^{(m)}})^{-1}-(\Upsilon^{(m)})^{-1}\right)\widehat{S}_m+\left\|\sqrt{T}(\Upsilon^{(m)})^{-\frac{1}{2}}\widehat{S}_m\right\|_2^2-\left\|V_{T,m}\right\|_2^2\Bigg|\\
	\leq &\sum_{m=1}^k \left\|\Upsilon^{(m)\frac{1}{2}}\left(\widehat{\Upsilon^{(m)}}\right)^{-1}\Upsilon^{(m)\frac{1}{2}}-I\right\|_\infty\left\|\sqrt{T}(\Upsilon^{(m)})^{-\frac{1}{2}}\widehat{S}_m\right\|_1^2\\
	&+\left\|\sqrt{T}(\Upsilon^{(m)})^{-\frac{1}{2}}(\widehat{S}_m-S_m)\right\|_2^2+2\left\|V_{T,m}\right\|_2\left\|\sqrt{T}(\Upsilon^{(m)})^{-\frac{1}{2}}(\widehat{S}_m-S_m)\right\|_2.
	    \end{split}
	\end{equation}
	Define $E_m =\sqrt{T}(\Upsilon^{(m)})^{-\frac{1}{2}}\left(\widehat{S}_m-S_m\right)$, then \eqref{EstErr} turns into
	\begin{equation}\label{U_T_dev}
	    \begin{split}
	        \left|\widehat{U}_T-U_T\right|\leq &\sum_{m=1}^k \left\|E_m\right\|_2^2+2\left\|V_{T,m}\right\|_2\left\|E_m\right\|_2\\
	&+\left\|\Upsilon^{(m)\frac{1}{2}}\left(\widehat{\Upsilon^{(m)}}\right)^{-1}\Upsilon^{(m)\frac{1}{2}}-I\right\|_\infty\left(\left\|V_{T,m}\right\|_2+\left\|E_m\right\|_2\right)^2.
	    \end{split}
	\end{equation}
	
	We can bound $\|V_{T,m}\|_2$ using Lemma \ref{CLT} and $\left\|\Upsilon^{(m)\frac{1}{2}}\left(\widehat{\Upsilon^{(m)}}\right)^{-1}\Upsilon^{(m)\frac{1}{2}}-I\right\|_\infty$ using Lemma \ref{Cov_mBnd}, while for bounding the estimation induced error $\|E_m\|_2$, we first apply the following lemma to bound the eigenvalues of $\Upsilon^{(m)}$.

\begin{Lemma}\label{CovEigenBnd}
Consider the model (\ref{ARp1}) with independent noise $\epsilon_{ti}$ of unit variance, $A^*$ satisfies \eqref{StbARpsG}, then the eigenvalues of $\Upsilon$ can be bounded as follows:
\begin{equation*}
    0<C_1(\beta)\leq \Lambda_{\min}\left(\Upsilon\right)\leq \Lambda_{\max}\left(\Upsilon\right)\leq C_2(\beta).
\end{equation*}
\end{Lemma}

Lemma \ref{CovEigenBnd} is proved based on established results in \cite{basu2015regularized}. Note that we assumed unit variance in Theorem \ref{NullThm} and \ref{AltThm}, so we can apply Lemma \ref{CovEigenBnd} here. Since $\left(\Upsilon^{(m)}\right)^{-1}=\left(\Upsilon^{-1}\right)_{D_m,D_m}$, applying Lemma \ref{CovEigenBnd} would lead us to the following:
\begin{equation}\label{Cov_mEigenBnd}
    \begin{split}
        \Lambda_{\min}\left((\Upsilon^{(m)})^{-1}\right)\geq \Lambda_{\min}(\Upsilon^{-1})=\Lambda_{\max}(\Upsilon)^{-1}\geq C,\\
        \Lambda_{\max}\left((\Upsilon^{(m)})^{-1}\right)\leq \Lambda_{\max}(\Upsilon^{-1})=\Lambda_{\min}(\Upsilon)^{-1}\leq C.
    \end{split}
\end{equation}
Thus we have
\begin{equation*}
        \left\|E_m\right\|_2\leq C\sqrt{T}\left\|\widehat{S}_m-S_m\right\|_2,
\end{equation*}
with 
\begin{equation}\label{NullSmErr}
\begin{split}
    \widehat{S}_m-S_m=&(\hat{w}_m-w_m^*)^\top\frac{1}{T}\sum_{t=0}^{T-1} \mathcal{X}_{t,D_m^c}\epsilon_{t,m}\\
    &+\frac{1}{T}\sum_{t=0}^{T-1}(\mathcal{X}_{t,D_m}-w_m^{*\top} \mathcal{X}_{t,D_m^c})\mathcal{X}_{t,D_m^c}^\top\left((\widehat{A}_m)_{D_m^c}-(A_m^*)_{D_m^c}\right)\\
    &-(\hat{w}_m-w_m^*)^\top \left(\frac{1}{T}\sum_{t=0}^{T-1}\mathcal{X}_{t,D_m^c}\mathcal{X}_{t,D_m^c}^\top\right)\left((\widehat{A}_m)_{D_m^c}-(A^*_m)_{D_m^c}\right).
\end{split}
\end{equation}
    The following two lemmas provide bounds for $\left\|\frac{1}{T}\sum_{t=0}^{T-1} \mathcal{X}_{t,D_m^c}\epsilon_{t,m}\right\|_\infty$, and 
    $$\left\|\frac{1}{T}\sum_{t=0}^{T-1}(\mathcal{X}_{t,D_m}-w_m^{*\top}\mathcal{X}_{t,D_m^c})\mathcal{X}_{t,D_m^c}^\top\right\|_{\infty}.$$
\begin{Lemma}
When $T\geq C\log M$,
	 \begin{align*}
	\mathbb{P}\left(\left\|\frac{1}{T}\sum_{t=0}^{T-1}\epsilon_t\mathcal{X}_t^\top\right\|_\infty>C\sqrt{\frac{\log M}{T}}\right)\leq c_1\exp\{-c_2\log M\}.
	\end{align*}
\end{Lemma}
Lemma \ref{DevBnd} is a common condition in high-dimensional regression problems, and is usually referred to as deviation bound. We will prove it in Section \ref{lemma_proof2}.
\begin{Lemma}[Deviation Bound for $w_m^*$]\label{true_w_dev}
With probability at least $1-c_1\exp\{-c_2\log M\}$, for all $1\leq m\leq k$,
    \begin{align*}
	\left\|\frac{1}{T}\sum_{t=0}^{T-1}(\mathcal{X}_{t,D_m}-w_m^{*\top}\mathcal{X}_{t,D_m^c})\mathcal{X}_{t,D_m^c}^\top\right\|_{\infty}\leq C\sqrt{\frac{\log M}{T}}.
	\end{align*}
\end{Lemma}
Lemma \ref{true_w_dev} can also be viewed as a deviation bound, if we consider a regression problem with $\mathcal{X}_{t,D_m}$ as response and $\mathcal{X}_{t,D_m^c}$ as covariates. This is also proved in Section \ref{lemma_proof2}. Applying Assumptions \ref{Assump_A_bound} and \ref{Assump_w_bound}, with probability at least $1-c_1\exp\{-c_2\log M\}$,
\begin{equation*}
    \|E_m\|_2\leq C\frac{(s_m\vee \rho_m)\log M}{\sqrt{T}}+\sqrt{T}Q_1^{\frac{1}{2}}Q_2^{\frac{1}{2}}\leq C\frac{(s_m\vee \rho_m)\log M}{\sqrt{T}},
\end{equation*}
where 
\begin{equation*}
    \begin{split}
        Q_1=&\left(\left(\widehat{A}_m\right)_{D_m^c}-\left(A^*_m\right)_{D_m^c}\right)^\top \left(\frac{1}{T}\sum_{t=0}^{T-1}\mathcal{X}_{t,D_m^c}\mathcal{X}_{t,D_m^c}^\top\right)\left(\left(\widehat{A}_m\right)_{D_m^c}-(A^*_m)_{D_m^c}\right)\\
        Q_2=&\text{tr}\left[(\hat{w}_m-w^*_{m})^\top \left(\frac{1}{T}\sum_{t=0}^{T-1}\mathcal{X}_{t,D_m^c}\mathcal{X}_{t,D_m^c}^\top\right)(\hat{w}_{m}-w^*_{m})\right],
    \end{split}
\end{equation*}
and Assumption \ref{Assump_A_bound} and \ref{Assump_w_bound} implies $Q_1\leq C\frac{\rho_{m}\log M}{T}$ and $Q_2\leq C\frac{s_{m}\log M}{T}$. The former is not straightforward: to see why it holds true, let $\hat{h}_m=\widehat{A}_m-A_m^*$ and $H=\frac{1}{T}\sum_{t=0}^{T-1}\mathcal{X}_t\mathcal{X}_t^\top$, then we have
    \begin{equation}\label{qua_error}
        \begin{split}
            Q_1=&\frac{1}{T}\sum_{t=0}^{T-1}\left[\mathcal{X}_{t,D_m^c}^\top \left(\hat{h}_m\right)_{D_m^c}\right]^2\\
            =&\frac{1}{T}\sum_{t=0}^{T-1}\left[\mathcal{X}_t^\top\hat{h}_m - \mathcal{X}_{t,D_m}^\top \left(\hat{h}_m\right)_{D_m}\right]^2\\
            \leq &\frac{2}{T}\sum_{t=0}^{T-1}\left[\left(\mathcal{X}_t^\top\hat{h}_m\right)^2+\left(\mathcal{X}_{t,D_m}^\top \left(\hat{h}_m\right)_{D_m}\right)^2\right]\\
            =&2\hat{h}_m^\top H \hat{h}_m+2\left(\hat{h}_m\right)_{D_m}^\top H_{D_m,D_m}\left(\hat{h}_m\right)_{D_m}\\
            \leq &C\frac{\rho_m\log M}{T}.
        \end{split}
    \end{equation}
    Here we apply Assumption \ref{Assump_A_bound}, and the fact that
    \begin{equation*}
    \begin{split}
        &\left(\hat{h}_m\right)_{D_m}^\top H_{D_m,D_m}\left(\hat{h}_m\right)_{D_m}\\
        \leq &d_m\|H\|_\infty\|\hat{h}_m\|_2^2\\
        \leq &d_m\left(\left\|H-\Upsilon\right\|_{\infty}+\Lambda_{\max}(\Upsilon)\right)\frac{\rho_m\log M}{T}\\
        \leq &C\frac{\rho_m\log M}{T}.
    \end{split}
    \end{equation*}
    The last inequality is due to Lemma \ref{CovEigenBnd} and the following lemma:
    \begin{Lemma}\label{CovBound}
With probability at least $1-c_1\exp\{-c_2\log M\}$,
    \begin{align*}
	\left\|\frac{1}{T}\sum_{t=0}^{T-1}\mathcal{X}_t \mathcal{X}_t^\top-\Upsilon\right\|_\infty\leq C\sqrt{\frac{\log M}{T}}.
	\end{align*}
\end{Lemma}

Therefore, by taking a union bound, we show that
\begin{equation*}
    \|E_m\|_2\leq  C\frac{(s_m\vee \rho_m)\log M}{\sqrt{T}},
\end{equation*}
for any $1\leq m\leq k$, with probability at least $1-c_1\exp\{-c_2\log M\}$.
    
    Meanwhile, by applying Lemma \ref{CLT}, one can show that for $y>\sqrt{5d}$,
    \begin{equation}\label{V_T_TB}
      \begin{split}
         \mathbb{P}\left(\left\|V_{T,m}\right\|_2>y\right)\leq &CT^{-\frac{1}{8}}+1-F_d(y^2)\\
        \leq &CT^{-\frac{1}{8}}+\exp\{-(y^2-d)/4\}\\
        \leq &CT^{-\frac{1}{8}}+Cy^{-2}, 
      \end{split}        
    \end{equation}
    where the second inequality is due to a $\chi_d^2$ tail bound established in \cite{laurent2000adaptive} (see Lemma 1 in \cite{laurent2000adaptive}), and the third inequality comes from the fact that, $\forall$ constant $C_1>0$, $\exists$ constant $C_2$ such that
    \begin{equation*}
        \sup_{y\geq 0}y^2e^{-C_1y^2}\leq C_2.
    \end{equation*} 
    Let $y=\left(\frac{(s\vee\rho)\log M}{\sqrt{T}}\right)^{-\frac{1}{4}}$ and plug it into \eqref{U_T_dev}, then with Assumption \ref{Assump_Cov_m_bound}, we can show that with probability at least 
    \begin{align*}
    1-c_1\exp\{-c_2\log M\}-c_3T^{-\frac{1}{8}}-c_4\left(\frac{(s\vee\rho)\log M}{\sqrt{T}}\right)^{\frac{1}{2}},
    \end{align*}
    the following holds:
    \begin{align*}
        \left|\widehat{U}_T-U_T\right|\leq &C_1\frac{(s\vee\rho)\log M}{\sqrt{T}}\left(\frac{(s\vee\rho)\log M}{\sqrt{T}}\right)^{-\frac{1}{2}}+C_2\left(\frac{(s\vee\rho)\log M}{\sqrt{T}}\right)^{\frac{3}{4}}\\
        \leq &C\left(\frac{(s\vee \rho)\log M}{\sqrt{T}}\right)^{\frac{1}{2}},
    \end{align*}
    if $(s\vee \rho)\log M=o(\sqrt{T})$ and $T>C$ for some constant $C$.
	Therefore, applying (\ref{main_prob_diff}) with $\varepsilon=C\left(\frac{(s\vee \rho)\log M}{\sqrt{T}}\right)^{\frac{1}{2}}$,
	\begin{align*}
	\left|\mathbb{P}(\widehat{U}_T\leq x)-F_{d}(x)\right|\leq C_1T^{-\frac{1}{8}}+C_2\left(\frac{(s\vee \rho)\log M}{\sqrt{T}}\right)^{\frac{1}{2}}+C_3\exp\{-c\log M\}.
	\end{align*}
	Since constants $C_i$ only depend on $d,\beta$ and $\tau$, this bound also holds for supremum over $A^*\in\Omega_0$ and $x\in \mathbb{R}$. Note that for a clear presentation, we are not showing the sharpest bound, which can be obtained by choosing a different $y$.
\end{proof}

\subsection{Proof of Theorem \ref{AltThm}}
\begin{proof}[proof of Theorem \ref{AltThm}]
We prove this case by case. We will use $C_i,c_i$ to refer to constants that only depend on $d,\beta,\Delta, \phi$, and different constants might share the same notation. 

Similar from the proof of Theorem \ref{NullThm}, the major part of the proof is devoted to bounding $\left|\widehat{U}_T-\left\|V_T+\mu\right\|_2^2\right|$ with high probability for some vector $\mu\in \mathbb{R}^d$.
\begin{enumerate}
\item[(1)] $\phi=\frac{1}{2}$\\
Suppose $A^*\in \Omega_1$. Using similar deduction as in the proof of Theorem \ref{NullThm}, for any $\varepsilon>0$,
	\begin{equation}\label{main_prob_diff_p1}
	\begin{split}
	    &\left|\mathbb{P}(\widehat{U}_T\leq x)-F_{d,\|\widetilde{\Delta}\|_2^2}(x)\right|\\
	    \leq &\sup_{y\in \mathbb{R}}\left|\mathbb{P}\left(\|V_T-\widetilde{\Delta}\|_2^2\leq y\right)-F_{d,\|\widetilde{\Delta}\|_2^2}(y)\right|\\
	&+F_{d,\|\widetilde{\Delta}\|_2^2}(x+\varepsilon)-F_{d,\|\widetilde{\Delta}\|_2^2}(x-\varepsilon)+\mathbb{P}\left(\left|\widehat{U}_T-\left\|V_T-\widetilde{\Delta}\right\|_2^2\right|>\varepsilon\right).
	\end{split}
	\end{equation}
	
\begin{enumerate}[leftmargin = 0cm]
    \item[(a)] Bounding the first two terms\\
    The first term is the convergence rate of $\|V_T-\widetilde{\Delta}\|_2^2$ to $\chi^2_{d,\|\widetilde{\Delta}\|_2^2}$. By Lemma \ref{CLT}, 
	$$
	\sup_{y\in \mathbb{R}}\left|\mathbb{P}\left(\|V_T-\widetilde{\Delta}\|_2^2\leq y\right)-F_{d,\|\widetilde{\Delta}\|_2^2}(y)\right|\leq C(\|\widetilde{\Delta}\|_2)T^{-\frac{1}{8}}\leq C\|\Delta\|_2T^{-\frac{1}{8}}.
	$$
	The last inequality is due to
	\begin{equation*}
\|\widetilde{\Delta}\|_2^2=\sum_{m=1}^k \|\widetilde{\Delta}_m\|_2^2\leq \sum_{m=1}^k\Lambda_{\max}\left(\Upsilon^{(m)}\right)\|\Delta\|_2^2,
\end{equation*}
and an upper bound for $\Lambda_{\max}\left(\Upsilon^{(m)}\right)$ in \eqref{Cov_mEigenBnd}.

    Bounding the second term in \eqref{main_prob_diff_p1} is not straightforward as bounding $F_d(x+\varepsilon)-F_d(x-\varepsilon)$ in the proof of Theorem \ref{NullThm}, since $\widetilde{\Delta}$ is not a constant vector when $A^*$ takes different values in $\Omega_1^*$. We only have a uniform bound of $\left\|\widetilde{\Delta}\right\|_2$ as shown above. One can show that
	\begin{equation*}
                  \begin{split}
	    &F_{d,\|\widetilde{\Delta}\|_2^2}(x+\varepsilon)-F_{d,\|\widetilde{\Delta}\|_2^2}(x-\varepsilon)= \mathbb{P}\left(\left\|Z+\widetilde{\Delta}\right\|_2^2\in (x-\varepsilon,x+\varepsilon]\right)\\
                  \leq &\begin{cases}
                  C(d)\left((x+\varepsilon)^{\frac{d}{2}}-(x-\varepsilon)^{\frac{d}{2}}\right)e^{-(\sqrt{x-\varepsilon}-\|\widetilde{\Delta}\|_2)^2/2}, & \sqrt{x-\varepsilon}\geq 2\|\widetilde{\Delta}\|_2\\
                   C(d)\left((x+\varepsilon)^{\frac{d}{2}}-(x-\varepsilon)^{\frac{d}{2}}\right), & \sqrt{x-\varepsilon}<2\|\widetilde{\Delta}\|_2
                  \end{cases},
                  \end{split}
	\end{equation*}
	where $Z$ is a $d$-dimensional standard Gaussian random vector with density $\phi(z)=C(d)\exp\{-\|z\|_2^2/2\}$. The last inequality holds because that, for any set $\mathcal{C}\subset \mathbb{R}^d$,
	\begin{equation*}
	    \mathbb{P}\left(Z\in \mathcal{C}\right)\leq \sup_{z\in \mathcal{C}}\phi(z)\int_{z\in \mathcal{C}} \mathrm{d} z.
	\end{equation*}
	Suppose $0<\varepsilon\leq 1$, then if $\sqrt{x-\varepsilon}\geq 2\|\widetilde{\Delta}\|_2$,
	\begin{equation*}
    \begin{split}
            &\left((x+\varepsilon)^{\frac{d}{2}}-(x-\varepsilon)^{\frac{d}{2}}\right)\exp\left\{-(\sqrt{x-\varepsilon}-\|\widetilde{\Delta}\|_2)^2/2\right\}\\
            \leq &d\varepsilon(x+\varepsilon)^{\frac{d}{2}-1}\exp\{-(x-\varepsilon)/8\}\\
            \leq &d\varepsilon e^{\frac{\varepsilon}{4}}\sup_{y\geq 0}y^{\frac{d}{2}-1}\exp\{-y/8\}\leq C(d)\varepsilon,
            \end{split}
    \end{equation*}
     otherwise,
    \begin{equation*}
\left((x+\varepsilon)^{\frac{d}{2}}-(x-\varepsilon)^{\frac{d}{2}}\right)\leq d\varepsilon(x+\varepsilon)^{\frac{d}{2}-1}\leq C(d)\varepsilon.
\end{equation*}
Thus, 
\begin{equation*}
    F_{d,\|\widetilde{\Delta}\|_2^2}(x+\varepsilon)-F_{d,\|\widetilde{\Delta}\|_2^2}(x-\varepsilon)\leq C(d)\varepsilon.
\end{equation*}
\item[(b)] Bounding $\left|\widehat{U}_T-\left\|V_T-\widetilde{\Delta}\right\|_2^2\right|$\\
Similar from \eqref{U_T_dev} in the proof of Theorem \ref{NullThm}, it is straightforward to show that
\begin{equation}\label{Alt1U_T_dev}
\begin{split}
	&\left|\widehat{U}_T-\left\|V_T-\widetilde{\Delta}\right\|_2^2\right|\\
	\leq &\sum_{m=1}^k \left\|E_m\right\|_2^2+2\left\|V_{T,m}-\widetilde{\Delta}_m\right\|_2\left\|E_m\right\|_2\\
	&+\left\|\Upsilon^{(m)\frac{1}{2}}\left(\widehat{\Upsilon^{(m)}}\right)^{-1}\Upsilon^{(m)\frac{1}{2}}-I\right\|_\infty\left(\|V_{T,m}-\widetilde{\Delta}_m\|_2+\|E_m\|_2\right)^2,
	\end{split}
	\end{equation}
	where $E_m=\sqrt{T}(\Upsilon^{(m)})^{-\frac{1}{2}}\widehat{S}_m-V_{T,m}+\widetilde{\Delta}_m$. To bound $\|E_m\|_2$, note that
	$$
	V_{T,m}-\widetilde{\Delta}_m=\sqrt{T}(\Upsilon^{(m)})^{-\frac{1}{2}}S_m-\widetilde{\Delta}_m=\sqrt{T}(\Upsilon^{(m)})^{-\frac{1}{2}}(S_m-\Upsilon^{(m)}(A_m^*)_{D_m}),
    $$
    and 
    \begin{align*}
       S_m-\Upsilon^{(m)}(A_m^*)_{D_m}=&\left[\frac{1}{T}\sum_t(\mathcal{X}_{t,D_m}-w_m^{*\top}\mathcal{X}_{t,D_m^c})\mathcal{X}_{t,D_m}^\top-\Upsilon^{(m)}\right](A_m^*)_{D_m}\\
       -&\frac{1}{T}\sum_{t=0}^{T-1}\left(\mathcal{X}_{t+1,m}-(A_m^*)_{D_m^c}^{\top}\mathcal{X}_{t,D_m^c}\right)\left(\mathcal{X}_{t,D_m}-w_m^{*\top}\mathcal{X}_{t,D_m^c}\right)\\
       =&\widetilde{S}_m+W_m^*\left(\frac{1}{T}\sum_t \mathcal{X}_t \mathcal{X}_{t,D_m}^\top-\Upsilon_{\cdot,D_m}\right)(A_m^*)_{D_m},
    \end{align*}
    with $\widetilde{S}_m\in \mathbb{R}^{d_m}$ and $W_m^*\in \mathbb{R}^{d_m\times M}$ defined as follows:
    \begin{equation*}
    \begin{split}
        \widetilde{S}_m=&-\frac{1}{T}\sum_{t=0}^{T-1}(\mathcal{X}_{t+1,m}-(A_m^*)_{D_m^c}^{\top}\mathcal{X}_{t,D_m^c})(\mathcal{X}_{t,D_m}-w_m^{*\top}\mathcal{X}_{t,D_m^c}),
    \end{split}
    \end{equation*}
    \begin{equation}\label{W_def}
        (W_m^*)_{\cdot,D_m}=I_{d_m\times d_m},\quad (W_m^*)_{\cdot,D_m^c}=w_m^{*\top}.
    \end{equation}
	Therefore,
	\begin{align*}
	    \|E_m\|_2\leq &\left\|\sqrt{T}(\Upsilon^{(m)})^{-\frac{1}{2}}(\widehat{S}_m-\widetilde{S}_m)\right\|_2\\
	    &+\left\|(\Upsilon^{(m)})^{-\frac{1}{2}}W_m^*\left(\frac{1}{T}\sum_t \mathcal{X}_t \mathcal{X}_{t,D_m}^\top-\Upsilon_{\cdot,D_m}\right)\Delta_m\right\|_2\\
	    \leq &C\sqrt{T}\left\|\widehat{S}_m-\widetilde{S}_m\right\|_2+C\sqrt{d_m}\max_i\|(W_m^*)_{i\cdot}\|_1\left\|\frac{1}{T}\sum_t \mathcal{X}_t \mathcal{X}_t^\top-\Upsilon\right\|_\infty.
	\end{align*}
	The last inequality applies \eqref{Cov_mEigenBnd}.
    Meanwhile, 
	\begin{equation}\label{Wl1Bnd}
	\begin{split}
	    \max_i\left\|(W_m^*)_{i\cdot}\right\|_1=&1+\max_i\left\|(w_m^*)_{\cdot i}\right\|_1\\
	    \leq &1+\max_i\sqrt{s_m}\left\|(w_m^*)_{\cdot i}\right\|_2\\
	    \leq &1+\sqrt{s_m}\Lambda_{\min}(\Upsilon_{D_m^c,D_m^c})^{-1}\max_i\|\Upsilon_{\cdot i}\|_2\\
	    \leq &1+C\sqrt{s_m}\max_i\sqrt{(\Upsilon^2)_{ii}}\\
	    \leq &1+C\sqrt{s_m}\Lambda_{\max}(\Upsilon)\leq C\sqrt{s_m}.
	\end{split}
    \end{equation}
    The first equality and second inequality come from the definition of $W_m^*$ and $w_m^*$; the third inequality is because that $\left\|\Upsilon_{\cdot i}\right\|_2^2=\left(\Upsilon^2\right)_{ii}$; the fourth inequality is due to that $\left(\Upsilon^2\right)_{ii}=e_i^\top \Upsilon^2 e_i\leq \Lambda_{\max}(\Upsilon)^2$; and the last inequality is obtained from Lemma \ref{CovEigenBnd}.
    Applying Lemma \ref{CovBound} leads us to  
    \begin{equation*}
        \|E_m\|_2\leq C\sqrt{\frac{s_m\log M}{T}}+C\sqrt{T}\|\widehat{S}_m-\widetilde{S}_m\|_2.
    \end{equation*}
	    We can write $\widehat{S}_m-\widetilde{S}_m$ as
	    \begin{equation*}
	\begin{split}
	        \widehat{S}_m-\widetilde{S}_m=&(\hat{w}_m-w_m^*)^\top \frac{1}{T}\sum_{t=0}^{T-1}\mathcal{X}_{t,D_m^c}\left(\epsilon_{t,m}+T^{-\frac{1}{2}}\Delta_m^{\top}\mathcal{X}_{t,D_m}\right)\\
	        &+\left(\left(\widehat{A}_m\right)_{D_m^c}-\left(A_m^*\right)_{D_m^c}\right)^\top \frac{1}{T}\sum_{t=0}^{T-1}\mathcal{X}_{t, D_m^c}\left(\mathcal{X}_{t,D_m}-w_m^{*\top} \mathcal{X}_{t, D_m^c}\right)^\top\\
	        &-\left(\left(\widehat{A}_m\right)_{D_m^c}-\left(A_m^*\right)_{D_m^c}\right)^\top \frac{1}{T}\sum_{t=0}^{T-1}\mathcal{X}_{t, D_m^c}\mathcal{X}_{t, D_m^c}^\top(\hat{w}_m-w_m^{*\top}).
	   \end{split}
    \end{equation*}
    Note that 
    \begin{equation}\label{CovInfNorm}
    \begin{split}
        \left\|\frac{1}{T}\sum_{t=0}^{T-1}\mathcal{X}_{t,D_m^c}\mathcal{X}_{t,D_m}^\top\right\|_{\infty}\leq &\left\|\frac{1}{T}\sum_t \mathcal{X}_t \mathcal{X}_t^\top-\Upsilon\right\|_\infty+\left\|\Upsilon\right\|_{\infty}\\
        \leq &C\sqrt{\frac{\log M}{T}}+\left\|\Upsilon\right\|_2\leq C,
    \end{split}
    \end{equation}
    due to Lemma \ref{CovEigenBnd} and \ref{CovBound}, which further implies 
    \begin{equation*}
        \left\|(\hat{w}_m-w_m^*)^\top \frac{1}{T}\sum_{t=0}^{T-1}\mathcal{X}_{t,D_m^c}\mathcal{X}_{t,D_m}^\top \Delta_m\right\|_2\leq C\|\hat{w}_m-w_m^*\|_1.
    \end{equation*}
    Applying Assumption \ref{Assump_A_bound} to \ref{Assump_Cov_m_bound}, Lemma \ref{DevBnd}, \ref{true_w_dev}, one can show that with probability at least $1-c_1\exp\{-c_2\log M\}$,
    \begin{equation}\label{Alt1E_m}
        \|E_m\|_2\leq C\frac{(s_m\vee \rho_m)\log M}{\sqrt{T}},
    \end{equation}
    with the same arguments as bounding $\|\widehat{S}_m-S_m\|_2$ under $\mathcal{H}_0$.
	    
    While for $\left\|V_{T,m}-(\Upsilon^{(m)})^{\frac{1}{2}}\Delta_m\right\|_2$, applying Lemma \ref{CLT} leads us to 
    \begin{equation*}
	\begin{split}
        &\mathbb{P}\left(\left\|V_{T,m}-(\Upsilon^{(m)})^{\frac{1}{2}}\Delta_m\right\|_2>y\right)\\
        \leq &C_1T^{-\frac{1}{8}}+1-F_{d,\|\widetilde{\Delta}\|_2^2}(y^2)\\
        = &C_1T^{-\frac{1}{8}}+\mathbb{P}\left(\|Z+\widetilde{\Delta}\|_2^2>y^2\right)\\
        \leq &C_1T^{-\frac{1}{8}}+\mathbb{P}\left(\|Z\|_2^2>\left(y-C\|\Delta\|_2\right)^2\right),
    \end{split}
    \end{equation*}
    for any $y\geq 0$, where $Z\sim \mathcal{N}(0,I_d)$. We apply the tail bound for $\chi^2_d$ (Lemma 1 in \cite{laurent2000adaptive}) as in \eqref{V_T_TB}, and obtain
    \begin{equation*}
	\begin{split}
        \mathbb{P}\left(\|Z\|_2^2>\left(y-C\|\Delta\|_2\right)^2\right)\leq C\left(y-C\|\Delta\|_2\right)^{-2}\leq Cy^{-2},
    \end{split}
    \end{equation*}
    when $y>C$ for some constant $C$.
    Let $y=\left(\frac{(s\vee\rho)\log M}{\sqrt{T}}\right)^{-\frac{1}{4}}$, and plug $\left\|V_{T,m}-(\Upsilon^{(m)})^{\frac{1}{2}}\Delta_m\right\|_2\leq y$, \eqref{Alt1E_m} and \eqref{Cov_mBnd} into \eqref{Alt1U_T_dev}, one can show that 
    \begin{equation*}
\begin{split}
        &\left|\widehat{U}_T-\left\|V_T-\widetilde{\Delta}\right\|_2^2\right|\\
        \leq &C_1\left(\frac{(s\vee\rho)\log M}{\sqrt{T}}\right)^{\frac{3}{4}}+C_2\frac{(s\vee \rho)\log M}{\sqrt{T}}\left(\frac{(s\vee\rho)\log M}{\sqrt{T}}\right)^{-\frac{1}{2}}\\
        \leq &C\left(\frac{(s\vee \rho)\log M}{\sqrt{T}}\right)^{\frac{1}{2}},
 \end{split}
    \end{equation*}
    with probability at least 
    \begin{align*}
    1-c_1\exp\{-c_2\log M\}-c_3T^{-\frac{1}{8}}-c_4\left(\frac{(s\vee\rho)\log M}{\sqrt{T}}\right)^{\frac{1}{2}},
    \end{align*}
if $(s\vee \rho)\log M=o(T)$ and $T>C$.
\end{enumerate}
	
Therefore, applying (\ref{main_prob_diff_p1}) with $\varepsilon=C\left(\frac{(s\vee \rho)\log M}{\sqrt{T}}\right)^{\frac{1}{2}}$ leads to
\begin{equation*}
\begin{split}
	&\left|\mathbb{P}(\widehat{U}_T\leq x)-F_{d}(x)\right|\\
	\leq &C_1T^{-\frac{1}{8}}+C_2\left(\frac{(s\vee \rho)\log M}{\sqrt{T}}\right)^{\frac{1}{2}}+C_3\exp\{-C_4\log M\}.
	\end{split}
    \end{equation*}
	Since constants $C_i$ only depend on $d,\beta,\Delta,\tau$, this bound also holds for supremum over $A^*\in\Omega_1$ and $x\in \mathbb{R}$.

   \item[(2)] $0<\phi<\frac{1}{2}$\\
   First we provide a lower bound for $\widehat{U}_T$ with high probability. Since bounds in Assumption \ref{Assump_A_bound} to \ref{Assump_Cov_m_bound},  Lemma \ref{DevBnd} to \ref{CovBound} hold with probability at least $1-c_1\exp\{-c_2\log M\}$, we apply these bounds directly in following deduction. Meanwhile, we always assume $(\rho\vee s)\log M=o(\sqrt{T})$ and $T>C$ for desired constant $C$. With these conditions, one can show that
   \begin{equation}\label{Alt2U_T_dev}
      \begin{split}
          \widehat{U}_T=&\sum_{m=1}^k T\widehat{S}_m^\top (\widehat{\Upsilon^{(m)}})^{-1}\widehat{S}_m\\
          \geq& \sum_{m=1}^k T\|\Upsilon^{(m)-\frac{1}{2}}\widehat{S}_m\|_2^2 \left(1-d_m\left\|\Upsilon^{(m)\frac{1}{2}}(\widehat{\Upsilon^{(m)}})^{-1}\Upsilon^{(m)\frac{1}{2}}-I\right\|_{\infty}\right)\\
          \geq &CT\sum_{m=1}^k \left\|(\Upsilon^{(m)})^{-\frac{1}{2}}\widehat{S}_m\right\|_2^2\\
          \geq &C\left(\left(T\sum_{m=1}^k\left\|(\Upsilon^{(m)})^{-\frac{1}{2}}(\widehat{S}_m-S_m)\right\|_2^2\right)^{\frac{1}{2}}-\|V_T\|_2\right)^2.
      \end{split}
       \end{equation}
      The third line is due to Assumption \ref{Assump_Cov_m_bound}, which implies $\left\|\Upsilon^{(m)\frac{1}{2}}(\widehat{\Upsilon^{(m)}})^{-1}\Upsilon^{(m)\frac{1}{2}}-I\right\|_{\infty}$ converges to 0 under our scaling $(\rho\vee s)\log M=o(\sqrt{T})$.
      
      We provide a lower bound for $\left\|(\Upsilon^{(m)})^{-\frac{1}{2}}(\widehat{S}_m-S_m)\right\|_2^2$ in the following. First write $\widehat{S}_m-S_m$ as
      \begin{equation*}
      \begin{split}
          \widehat{S}_m-S_m=&(\hat{w}_m-w_m^*)^\top\left(\frac{1}{T}\sum_{t=0}^{T-1}\epsilon_{t,m}\mathcal{X}_{t,D_m^c}\right)\\
          &-\frac{1}{T}\sum_{t=0}^{T-1}(\mathcal{X}_{t,D_m}-\hat{w}_m^\top \mathcal{X}_{t,D_m^c})\mathcal{X}_{t,D_m}^\top (A_m^*)_{D_m}\\
          &+\frac{1}{T}\sum_{t=0}^{T-1}(\mathcal{X}_{t,D_m}-\hat{w}_m^\top \mathcal{X}_{t,D_m^c})\mathcal{X}_{t,D_m^c}^{\top}((\widehat{A}_m)_{D_m^c}-(A_m^*)_{D_m^c})\\
          \triangleq&E_m^{(1)}+E_m^{(2)}+E_m^{(3)},
      \end{split}
      \end{equation*}
      we find the upper bounds for $\left\|E_m^{(1)}\right\|_2, \left\|E_m^{(3)}\right\|_2$ and lower bound for $\left\|E_m^{(2)}\right\|_2$ in the following.
      Applying Assumption \ref{Assump_w_bound} and Lemma \ref{DevBnd} provides an upper bound for $\left\|E_m^{(1)}\right\|_2$:
      \begin{align*}
          \|E_m^{(1)}\|_2\leq  \|\hat{w}_m-w_m^*\|_1\left\|\frac{1}{T}\sum_{t=0}^{T-1}\mathcal{X}_t\epsilon_t^\top\right\|_\infty\leq \frac{s_m\log M}{T}.
      \end{align*}
      Since
      \begin{equation*}
      \begin{split}
          \left\|E_m^{(3)}\right\|_2\leq &\left\|(\hat{w}_m-w_m^*)^\top \frac{1}{T}\sum_{t=0}^{T-1}\mathcal{X}_{t,D_m^c}\mathcal{X}_{t,D_m^c}^\top((\widehat{A}_m)_{D_m^c}-(A_m^*)_{D_m^c})\right\|_2\\
          &+\sqrt{d_m}\left\|\frac{1}{T}\sum_{t=0}^{T-1}(\mathcal{X}_{t,D_m}-w_m^{*\top} \mathcal{X}_{t,D_m^c})\mathcal{X}_{t,D_m^c}^{\top}\right\|_\infty\left\|(\widehat{A}_m)_{D_m^c}-(A_m^*)_{D_m^c})\right\|_1,
      \end{split}
      \end{equation*}
      then using the same argument as bounding $\|\widehat{S}_m-S_m\|_2$ when proving Theorem \ref{NullThm}, we have
      \begin{equation*}
          \left\|E_m^{(3)}\right\|_2\leq C\frac{(s_m\vee \rho_m)\log M}{T}.
      \end{equation*}
      To lower bound $\|E_m^{(2)}\|_2$, first note that
      \begin{equation}\label{Upsilon_mBnd}
          \begin{split}
              &\left\|\frac{1}{T}\sum_{t=0}^{T-1}\left(\mathcal{X}_{t,D_m}-\hat{w}_m^{\top} \mathcal{X}_{t,D_m}\right)\mathcal{X}_{t,D_m^c}^{\top}-\Upsilon^{(m)}\right\|_\infty\\
              \leq& \max_i\left\|(W_m^*)_{i\cdot}\right\|_1\left\|\frac{1}{T}\sum_{t=0}^{T-1}\mathcal{X}_t\mathcal{X}_t^\top-\Upsilon\right\|_{\infty}+\|\hat{w}_m-w_m^*\|_1\left\|\frac{1}{T}\sum_{t=0}^{T-1}\mathcal{X}_t\mathcal{X}_t^\top\right\|_\infty\\
              \leq &Cs_m\sqrt{\frac{\log M}{T}},
          \end{split}
      \end{equation}
      where we apply \eqref{Wl1Bnd}, Lemma \ref{CovBound}, Assumption \ref{Assump_w_bound}, and bound $\left\|\frac{1}{T}\sum_{t=0}^{T-1}\mathcal{X}_t\mathcal{X}_t^\top\right\|_\infty$ using the same argument as in \eqref{CovInfNorm}. Thus,
      \begin{equation*}
          \|E_m^{(2)}\|_2\geq T^{-\phi}\left\|\Upsilon^{(m)}\Delta_m\right\|_2-Cs_m\sqrt{\frac{\log M}{T}}T^{-\phi}\geq CT^{-\phi},
      \end{equation*}
      since $\Delta_m$ is a constant vector, and $\Lambda_{\min}(\Upsilon^{(m)}$ is lower bounded by constant as in \eqref{Cov_mEigenBnd}.
      
      Applying these bounds for $\|E_m^{(i)}\|_2, 1\leq i\leq 3$, one can show that,
      
      \begin{equation*}
      \begin{split}
          T\sum_{m=1}^k\left\|(\Upsilon^{(m)})^{-\frac{1}{2}}(\widehat{S}_m-S_m)\right\|_2^2\geq \sum_{m=1}^k\left(C_1T^{\frac{1}{2}-\phi}-C_2\frac{(s\vee\rho)\log M}{\sqrt{T}}\right)^2
      \geq CT^{1-2\phi}.
      \end{split}
      \end{equation*}
      Plug this into \eqref{Alt2U_T_dev} and apply Lemma \ref{CLT}, we have
      \begin{equation*}
      \begin{split}
          &\mathbb{P}(\widehat{U}_T\leq x)\leq C\exp\{-c\log M\}+\mathbb{P}\left(\|V_T\|_2\geq C_1T^{\frac{1}{2}-\phi}-C_2\sqrt{x}\right)\\
          \leq &C_1\exp\{-c\log M\}+C_2T^{-\frac{1}{8}}+1-F_d((C_3T^{\frac{1}{2}-\phi}-C_4\sqrt{x})^2)\\
          \leq &C_1\exp\{-c\log M\}+C_2T^{-\frac{1}{8}}+C_3\exp\{-(C_3T^{\frac{1}{2}-\phi}-C_4\sqrt{x})^2\},
    \end{split}
      \end{equation*}
     where in the last line we apply the $\chi_d^2$ tail bound as in \eqref{V_T_TB}. Since the constants here only depend on $d,\beta,\Delta,\tau$, this bound holds when taking supremum over $A^*\in \Omega_1$ and $x\in \mathbb{R}$. 
     \item[(3)] $\phi>\frac{1}{2}$\\
     The proof of this case is similar to that of Theorem \ref{NullThm}. The only thing different lies in the choice of $\varepsilon$ and bounding $\mathbb{P}\left(\left|\widehat{U}_T-U_T\right|>\varepsilon\right)$. The bound (\ref{U_T_dev}) for $\left|\widehat{U}_T-U_T\right|$ still holds here, with $E_m=\sqrt{T}(\Upsilon^{(m)})^{-\frac{1}{2}}(\widehat{S}_m-S_m)$. We directly apply the bounds in Assumptions \ref{Assump_A_bound} to \ref{Assump_Cov_m_bound}, and Lemma \ref{DevBnd} to Lemma \ref{CovBound} in the following. First we write
     \begin{equation*}
         \begin{split}
             \widehat{S}_m-S_m= &(\hat{w}_m-w_m^*)^\top\frac{1}{T}\sum_{t=0}^{T-1} \mathcal{X}_{t,D_m^c}\epsilon_{t,m}\\
         &+\frac{1}{T}\sum_{t=0}^{T-1}\left(\mathcal{X}_{t,D_m}-w_m^{*\top} \mathcal{X}_{t,D_m^c}\right)\mathcal{X}_{t,D_m^c}^\top\left((\widehat{A}_m)_{D_m^c}-(A_m^*)_{D_m^c}\right)\\
        &-(\hat{w}_m-w_m^*)^\top \left(\frac{1}{T}\sum_{t=0}^{T-1}\mathcal{X}_{t,D_m^c}\mathcal{X}_{t,D_m^c}^\top\right)\left((\widehat{A}_m)_{D_m^c}-(A^*_m)_{D_m^c}\right)\\
         &-T^{-(1+\phi)}\sum_{t=0}^{T-1}(\mathcal{X}_{t,D_m}-\hat{w}_m^\top \mathcal{X}_{t,D_m^c})\mathcal{X}_{t,D_m}^\top\Delta_m.
         \end{split}
     \end{equation*}
         Note here that the first three terms are exactly the same as in \eqref{NullSmErr}, and thus can be bounded as in the proof of Theorem \ref{NullThm}. We only have to tackle the last term. By \eqref{Upsilon_mBnd}, one can show that,
         \begin{equation*}
             \begin{split}
                 \left\|\frac{1}{T}\sum_{t=0}^{T-1}(\mathcal{X}_{t,D_m}-\hat{w}_m^\top \mathcal{X}_{t,D_m^c})\mathcal{X}_{t,D_m}^\top\Delta_m\right\|_2
         \leq \left\|\Upsilon^{(m)}\Delta_m\right\|_2+Cs_m\sqrt{\frac{\log M}{T}}\leq C,
             \end{split}
         \end{equation*}
     Thus, going through the same arguments as bounding $\left\|\widehat{S}_m-S_m\right\|_2$ under $\mathcal{H}_0$, we have
     \begin{equation*}
             \|E_m\|_2\leq C_1\frac{(s\vee \rho)\log M}{\sqrt{T}}+C_2T^{\frac{1}{2}-\phi},
     \end{equation*}
     with probability at least $1-C\exp\{-c\log M\}$.
     Recall that in \eqref{V_T_TB}, when $y>C$ for some constant $C$,
     $$
     \mathbb{P}(\|V_{T,m}\|_2\geq y)\leq C_1T^{-\frac{1}{8}}+C_2y^{-2}.
     $$
     Let $y=\left(\frac{(s\vee \rho)\log M}{\sqrt{T}}\right)^{-\frac{1}{4}}\wedge T^{\frac{2\phi-1}{6}}$, then by (\ref{U_T_dev}) one can show that
    \begin{align*}
        &\left|\widehat{U}_T-U_T\right|\\
        \leq &C_1\frac{(s\vee \rho)\log M}{\sqrt{T}}\left(\frac{(s\vee\rho)\log M}{\sqrt{T}}\right)^{-\frac{1}{2}}+C_2\left(\frac{(s\vee\rho)\log M}{\sqrt{T}}\right)^{\frac{3}{4}}+C_3T^{\frac{1-2\phi}{3}}\\
        \leq &C_1\left(\frac{(s\vee \rho)\log M}{\sqrt{T}}\right)^{\frac{1}{2}}+C_2T^{\frac{1-2\phi}{3}},
    \end{align*}
    with probability at least 
    \begin{align*}
    1-c_1\exp\{-c_2\log M\}-c_3T^{-\frac{1}{8}}-c_4\left(\frac{(s\vee\rho)\log M}{\sqrt{T}}\right)^{\frac{1}{2}}-c_5T^{\frac{1-2\phi}{3}},
    \end{align*}
    if $(s\vee \rho)\log M=o(\sqrt{T})$ and $T>C$ for some constant $C$.
	Therefore, applying (\ref{main_prob_diff}) with $\varepsilon=C_1\left(\frac{(s\vee \rho)\log M}{\sqrt{T}}\right)^{\frac{1}{2}}+C_2T^{\frac{1-2\phi}{3}}$,
    \begin{align*}
	&\left|\mathbb{P}(\widehat{U}_T\leq x)-F_{d}(x)\right|\\
	\leq &C_1T^{-\frac{1}{8}}+C_2\left(\frac{(s\vee \rho)\log M}{\sqrt{T}}\right)^{\frac{1}{2}}+C_3T^{\frac{1-2\phi}{3}}+C_4\exp\{-C_5\log M\}.
	\end{align*}
	Since constants $C_i$ only depend on $d,\beta,\tau,\Delta$, this bound also holds for supremum over $A^*\in\Omega_1$ and $x\in \mathbb{R}$. 
  \end{enumerate}
\end{proof}

\section{Conclusion}
In this paper, we have provided theoretical guarantees for hypothesis tests for sparse high-dimensional auto-regressive models with sub-Gaussian innovations. Specific upper bounds for the convergence rates of test statistics are given. Importantly, our results go beyond the Gaussian assumption and do not rely on mixing assumptions. As a consequence of our theory, we also develop novel concentration bounds for quadratic forms of dependent sub-Gaussian random variables using a careful truncation argument. 

It would be of interest to consider other variance estimation method, e.g., scaled Lasso \cite{sun2012scaled}, or cross-validation based method \cite{fan2012variance}, and establish corresponding theoretical guarantee. There also remain a number of open questions/challenges including extensions to generalized linear models, heavy-tailed innovations and incorporating hidden variables under time series setting.

\section*{Acknowledgements}
We would like to thank both Sumanta Basu and Yiming Sun for useful discussions and comments. LZ and GR were supported by ARO W911NF-17-1-0357 and NGA HM0476-17-1-2003. GR was also supported by NSF DMS-1811767.

\bibliographystyle{abbrvnat} 
\bibliography{reference.bib}



\appendix
\section{Proof of Lemmas in Section \ref{Estimator}}\label{lemma_proof}
\begin{proof}[Proof of Lemma \ref{Abound}]
We prove the error bounds for each $\widehat{A}_m$ and then take a union bound. Without loss of generality, we consider the estimation of $A_1^*\in \mathbb{R}^M$. With a little abuse of notation, let $S=\text{supp}(A_1^*)$, $\hat{h}=\widehat{A}_1-A_1^*$, $S=\text{supp}(A_1^*)$, and $H=\frac{1}{T}\sum_{t=0}^{T-1}\mathcal{X}_t\mathcal{X}_t^\top$ ($S$ is not the decorrelated score function we defined in section \ref{TestStatistic}). We would like to bound $\|\hat{h}\|_1$, $\|\hat{h}\|_2$ and $\hat{h}^\top H\hat{h}$ under two cases separately:
    \begin{itemize}
        \item[(1)] $\widehat{A}=\widehat{A}^{(L)}$.\\
        Here we adopt the standard proof framework for Lasso. By (\ref{A_Ldef}) we know that $\widehat{A}_1\in \mathbb{R}^M$ satisfies
	$$
	\widehat{A}_1=\argmin_{\beta\in \mathbb{R}^M}\frac{1}{T}\sum_{t=0}^{T-1}(X_{t+1,1}-\mathcal{X}_t^\top \beta)^2+\lambda_A\|\beta\|_1,
	$$
	which implies
	\begin{equation*}
	\frac{1}{T}\sum_{t=0}^{T-1}(X_{t+1,1}-\mathcal{X}_t^\top \widehat{A}_1)^2+\lambda_A\|\widehat{A}_1\|_1\leq \frac{1}{T}\sum_{t=0}^{T-1}(X_{t+1,{1}}-\mathcal{X}_t^\top A_1^*)^2+\lambda_A\|A_1^*\|_1.
	\end{equation*}
	Rearranging the terms, we have
	\begin{equation*}
	    \begin{split}
	      \hat{h}^\top H\hat{h}&\leq 2\hat{h}^\top\left(\frac{1}{T}\sum_{t=0}^{T-1}\epsilon_{t,{1}}\mathcal{X}_t\right)+\lambda_A\|A_1^*\|_1-\lambda_A\|\widehat{A}_1\|_1\\
	&\leq 2\left\|\frac{1}{T}\sum_{t=0}^{T-1}\epsilon_t\mathcal{X}_t^\top\right\|_\infty \|\hat{h}\|_1+\lambda_A\|\hat{h}_{S}\|_1-\lambda_A\|\hat{h}_{S^c}\|_1.
	    \end{split}
	\end{equation*}
	The last line is due to that 
	\begin{align*}
	 \|A_1^*\|_1-\|\widehat{A}_1\|_1&=\|(A_1^*)_{S}\|_1-\|(\widehat{A}_1)_{S}\|_1-\|(\widehat{A}_1)_{S^c}\|_1\\
	 &=\|(A_1^*)_{S}\|_1-\|(\widehat{A}_1)_{S}\|_1-\|\hat{h}_{S^c}\|_1\\
	 &\leq \|\hat{h}_{S}\|_1-\|\hat{h}_{S^c}\|_1.
	\end{align*}
	By Lemma \ref{DevBnd}, with probability at least $1-c_1\exp\{-c_2\log M\}$, 
	$$
	\left\|\frac{1}{T}\sum_{t=0}^{T-1}\epsilon_{t}\mathcal{X}_t^\top\right\|_\infty\leq \frac{1}{4}\lambda_A=C\sqrt{\frac{\log M}{T}}.
	$$ 
	Meanwhile, since $H$ is positive semi-definite, 
	\begin{equation*}
	\begin{split}
	0\leq \hat{h}^\top H\hat{h}&\leq \frac{3\lambda_A}{2}\|\hat{h}_{S}\|_1-\frac{\lambda_A}{2}\|\hat{h}_{S^c}\|_1,\\
	\|\hat{h}_{S^c}\|_1&\leq 3\|\hat{h}_{S}\|_1.
	\end{split}
	\end{equation*}
	We have the following restricted eigenvalue condition for $H$.
	\begin{Lemma}\label{REC}
    Under the model specified in \eqref{ARp2} with independent sub-Gaussian noise $\epsilon_{ti}$ of constant scale factor, and $A^*\in \Omega_0\cup \Omega_1$, for any set $J\subset \{1,2,\cdots,pM\}$, positive integer $\kappa>0$, $H$ satisfies the following REC:
	$$
	\inf\{v^\top H v:v\in \mathcal{C}(J,\kappa),\|v\|_2\leq 1\}\geq C_1>0,
	$$
	with probability at least $1-2\exp\left\{-cT\right\}$, when $|J|\log pM\leq C_2T$.
 Here $\mathcal{C}(J,\kappa)=\{v:\|v_{J^c}\|_1\leq \kappa\|v_{J}\|_1\}$, constant $C_1$ depends on $\beta$, $c$ and $C_2$ depend on $\kappa$ and $\beta$.
\end{Lemma}
Here $\hat{h}\in \mathcal{C}(S,3)$, $|S|=\rho_{1}$, by Lemma \ref{REC}, when $T\geq C\rho \log M$,
$$
\hat{h}^\top H\hat{h}\geq C\|\hat{h}\|_2^2,
$$
with probability at least $1-2\exp\{-cT\}$, when $T>C\rho \log M$. 
	Thus
	\begin{equation}\label{h_l2_bound}
	  \|\hat{h}\|_2^2\leq C\hat{h}^\top H\hat{h}\leq C\lambda_A\|\hat{h}_{S}\|_1\leq C\sqrt{\frac{\rho_{1}\log M}{T}}\|\hat{h}\|_2,
	\end{equation}
	which implies
	\begin{equation*}
	\begin{split}
	    \|\hat{h}\|_2\leq &C\sqrt{\frac{\rho_{1}\log M}{T}},\quad \hat{h}^\top H\hat{h}\leq C\frac{\rho_1\log M}{T},\\
	    \|\hat{h}\|_1\leq &4\|\hat{h}_{S}\|_1\leq 4\sqrt{\rho_{1}}\|\hat{h}\|_2\leq C\rho_{1}\sqrt{\frac{\log M}{T}},
	\end{split}
	\end{equation*}
	with probability at least $1-c_1\exp\{-c_2\log M\}$.
\item[(2)] $\widehat{A}=\widehat{A}^{(D)}$.\\
        Here we adopt the standard proof framework for Dantzig selector. By (\ref{A_Ddef}), 
        \begin{equation}\label{A_Dcond}
            \widehat{A}_1=\argmin_{\beta\in \mathbb{R}^M}\|\beta\|_1,\quad\text{s.t.}\quad \left\|\frac{1}{T}\sum_{t=0}^{T-1}(X_{t+1,1}-\mathcal{X}_t^\top \beta)\mathcal{X}_t\right\|_\infty\leq\lambda_A.
        \end{equation}
	     By Lemma \ref{DevBnd}, when $T\geq C\log M$, with probability at least $1-c_1\exp\{-c_2\log M\}$, 
	     $$
	     \left\|\frac{1}{T}\sum_{t=0}^{T-1}(X_{t+1,1}-\mathcal{X}_t^\top A_1^*)\mathcal{X}_t\right\|_\infty=\left\|\frac{1}{T}\sum_{t=0}^{T-1}\epsilon_{t,1}\mathcal{X}_t\right\|_\infty\leq \lambda_A,
	     $$
	     which implies
	     $$
	     \|H\hat{h}\|_\infty\leq C\sqrt{\frac{\log M}{T}}.
	     $$
	     Meanwhile, by (\ref{A_Dcond}),
	     \begin{equation*}
	         \begin{split}
	             \|\widehat{A}_1\|_1\leq \|A_1^*\|_1,\quad \|\hat{h}_{S^c}\|_1\leq \|\hat{h}_S\|_1.
	         \end{split}
	     \end{equation*}
Here $\hat{h}\in \mathcal{C}(S,1)$, $|S|=\rho_{1}$, by Lemma \ref{REC}, when $T\geq C\rho \log M$,
$$
\hat{h}^\top H\hat{h}\geq C\|\hat{h}\|_2^2,
$$
with probability at least $1-2\exp\{-cT\}$, when $T>C\rho \log M$. Thus
	\begin{equation}\label{h_l2_bound}
	  \|\hat{h}\|_2^2\leq C\hat{h}^\top H\hat{h}\leq \|H\hat{h}\|_\infty\|\hat{h}\|_1\leq C\sqrt{\frac{\log M}{T}}\|\hat{h}\|_1\leq C\sqrt{\frac{\rho_1\log M}{T}}\|\hat{h}\|_2,
	\end{equation}
	which implies
	\begin{equation*}
	\begin{split}
	    \|\hat{h}\|_2\leq &C\sqrt{\frac{\rho_{1}\log M}{T}},\quad \hat{h}^\top H\hat{h}\leq C\frac{\rho_{1}\log M}{T},\\
	    \|\hat{h}\|_1\leq &4\|\hat{h}_{S}\|_1\leq 4\sqrt{\rho_{1}}\|\hat{h}\|_2\leq C\rho_{1}\sqrt{\frac{\log M}{T}},
	\end{split}
	\end{equation*}
	with probability at least $1-c_1\exp\{-c_2\log M\}$.
    \end{itemize}
Therefore, after taking a union bound over $m=1,\cdots,k$, proof complete.
\end{proof}

\begin{proof}[Proof of Lemma \ref{wbound}]
Without loss of generality, we consider the estimation of $(w_1^*)_{\cdot,1}$ and then take a union bound. Let $v^*=(w_1^*)_{\cdot,1}$, $\hat{v}=(\hat{w}_1)_{\cdot,1}$, $\hat{h}=\hat{v}-v^*\in\mathbb{R}^{M-d_1}$ and $S=\text{supp}(v^*)$. Then we prove upper bounds for $\|\hat{h}\|_1$ and $\hat{h}^\top H_{D_1^c,D_1^c}\hat{h}$ with high probability under two cases. 
\begin{itemize}
    \item[(1)] $\hat{w}_m=\hat{w}_m^{(L)}$.\\
    Looking into the definition (\ref{w_Ldef}) of $\hat{w}_1$, it is clear that the optimization can be viewed as $d_1$ separate optimization problems, in terms of each column of $\hat{w}_1$. Thus
    \begin{equation*}
        \hat{v}=\arg\min_{v\in \mathbb{R}^{M-d_1}}\frac{1}{T}\sum_{t=0}^{T-1}\left(\left(\mathcal{X}_{t,D_1}\right)_1-\mathcal{X}_{t,D_1^c}^\top v\right)^2+\lambda_w\|v\|_1.
    \end{equation*}
    The following proof is almost identical to the proof in Lemma \ref{Abound} under $\widehat{A}=\widehat{A}^{(L)}$, except some difference in notation and application of Lemmas. 
    One can show that,
    \begin{equation*}
    \begin{split}
        &\frac{1}{T}\sum_{t=0}^{T-1}\left(\left(\mathcal{X}_{t,D_1}\right)_1-\mathcal{X}_{t,D_1^c}^\top \hat{v}\right)^2+\lambda_w\|\hat{v}\|_1\\
        \leq &\frac{1}{T}\sum_{t=0}^{T-1}\left(\left(\mathcal{X}_{t,D_1}\right)_1-\mathcal{X}_{t,D_1^c}^\top v^*\right)^2+\lambda_w\|v^*\|_1,
        \end{split}
        \end{equation*}
        Rearranging the inequality gives us
        \begin{equation*}
            \begin{split}
                \hat{h}^\top H_{D_1^c,D_1^c}\hat{h}\leq & 2\hat{h}^\top\left(\frac{1}{T}\sum_{t=0}^{T-1}\left(\left(\mathcal{X}_{t,D_1}\right)_1-\mathcal{X}_{t,D_1^c}v^*\right)\mathcal{X}_{D_1^c}\right)+\lambda_w\|v^*\|_1-\lambda_w\|\hat{v}\|_1\\
	\leq &2\left\|\frac{1}{T}\sum_{t=0}^{T-1}(\mathcal{X}_{t,D_1}-w_1^{*\top}\mathcal{X}_{t,D_1^c})\mathcal{X}_{D_1^c}^\top\right\|_\infty \|\hat{h}\|_1+\lambda_w\|\hat{h}_{S}\|_1-\lambda_A\|\hat{h}_{S^c}\|_1.
            \end{split}
        \end{equation*}

	By Lemma \ref{true_w_dev}, with probability at least $1-c_1\exp\{-c_2\log M\}$, 
	$$
	\left\|\frac{1}{T}\sum_{t=0}^{T-1}(\mathcal{X}_{t,D_1}-w_1^{*\top}\mathcal{X}_{t,D_1^c})\mathcal{X}_{D_1^c}^\top\right\|_\infty\leq \frac{1}{4}\lambda_w=C\sqrt{\frac{\log M}{T}},
	$$ 
	which implies, 
	\begin{equation*}
	\begin{split}
	0\leq \hat{h}^\top H_{D_1^c,D_1^c}\hat{h}&\leq \frac{3\lambda_w}{2}\|\hat{h}_{S}\|_1-\frac{\lambda_w}{2}\|\hat{h}_{S^c}\|_1,\\
	\|\hat{h}_{S^c}\|_1&\leq 3\|\hat{h}_{S}\|_1.
	\end{split}
	\end{equation*}
	Let $\tilde{h}\in\mathbb{R}^{M}$ be defined as the following:
\begin{equation}\label{h_def}
    \tilde{h}_{D_1}=0,\quad \tilde{h}_{D_1^c}=\hat{h},
\end{equation}
	By Lemma \ref{REC}, when $T\geq Cs\log M$, with probability at least $1-2\exp\{-cT\}$, 
	\begin{equation*}
	\begin{split}
	    \|\hat{h}\|_2^2=\|\tilde{h}\|_2^2\leq &C\tilde{h}^\top H\tilde{h}=2\hat{h}^\top H_{D_1^c,D_1^c}\hat{h}\leq C\lambda_w\|\hat{h}_{S}\|_1\leq C\sqrt{\frac{s_{1}\log M}{T}}\|\hat{h}\|_2,
	\end{split}
	\end{equation*}
	which implies
	\begin{equation*}
	     \hat{h}^\top H\hat{h}\leq C\frac{s_1\log M}{T},
	\end{equation*}
	and
	\begin{equation*}
	    \|\hat{h}\|_1\leq 4\|\hat{h}_{S}\|_1\leq 4\sqrt{s_{1}}\|\hat{h}\|_2\leq Cs_{1}\sqrt{\frac{\log M}{T}},
	\end{equation*}
	with probability at least $1-c_1\exp\{-c_2\log M\}$.
	\item[(2)]$\hat{w}_m=\hat{w}_m^{(D)}$.\\
	By (\ref{w_Ddef}),
   \begin{equation}\label{v_def}
    \hat{v}=\argmin_{v\in \mathbb{R}^{M-d_1}}{\|v\|_1},\quad \text{s.t. }\left\|\frac{1}{T}\sum_{t=0}^{T-1}\left((\mathcal{X}_{t,D_1})_1-v^\top \mathcal{X}_{t,D_1^c}\right)\mathcal{X}_{t,D_1^c}\right\|_\infty\leq \lambda_w.
    \end{equation}
    This proof is also pretty similar to the proof of Lemma \ref{Abound} under the case where $\widehat{A}=\widehat{A}^{(D)}$.
By Lemma \ref{true_w_dev}, 
$$
\left\|\frac{1}{T}\sum_{t=0}^{T-1}\left(\left(\mathcal{X}_{t,D_1}\right)_1-v^{*\top}\mathcal{X}_{t,D_1^c}\right)_1\mathcal{X}_{t,D_1^c}\right\|_\infty\leq \lambda_w=C\sqrt{\frac{\log M}{T}},
$$
with probability at least $1-c_1\exp\{-c_2\log M\}$. Thus,
\begin{equation*}
    \left\|H_{D_1^c,D_1^c}^\top\hat{h}\right\|_\infty\leq C\sqrt{\frac{\log M}{T}}.
\end{equation*}
Meanwhile, by (\ref{v_def}),
\begin{equation*}
    \left\|\hat{v}\right\|_1=\left\|\hat{v}_S\right\|_1+\left\|\hat{v}_{S^c}\right\|_1\leq \left\|v^*\right\|_1=\left\|v^*_S\right\|_1,
\end{equation*} 
which further implies
\begin{equation}\label{w_cone}
    \begin{split}
        \left\|\hat{h}_{S^c}\right\|_1&\leq \left\|\hat{h}_{S}\right\|_1.
    \end{split}
\end{equation}
Recall the definition of $\tilde{h}$ in (\ref{h_def}),then by Lemma \ref{REC}, (\ref{w_cone}) and (\ref{h_def}), when $T\geq Cs\log M$,
\begin{equation*}
\begin{split}
    \|\hat{h}\|_2^2=\|\tilde{h}\|_2^2\leq &C\tilde{h}^\top H\tilde{h}\\
    =&C\hat{h}^\top H_{D_1^c,D_1^c}\hat{h}\\
    \leq &C\left\|\hat{h}\right\|_1\left\|H_{D_1^c,D_1^c}^\top\hat{h}\right\|_\infty\\
    \leq &C\sqrt{\frac{\log M}{T}}\|\hat{h}_S\|_1\\
    \leq &C\sqrt{\frac{s_1\log M}{T}}\|\hat{h}\|_2,
\end{split}
\end{equation*}
which implies
\begin{equation*}
    \hat{h}^\top H_{D_1^c,D_1^c}\hat{h}\leq C\frac{s_1\log M}{T},
\end{equation*}
and 
\begin{equation*}
    \|\hat{h}\|_1\leq C\sqrt{s_1}\|\hat{h}\|_2\leq Cs_1\sqrt{\frac{\log M}{T}},
\end{equation*}
with probability at least $1-c_1\exp\{-c_2\log M\}$.
\end{itemize}
Since 
$$
\|(\hat{w}_1)-(w_1^*)\|_1=\sum_{j=1}^{d_1}\|(\hat{w}_1)_{\cdot,j}-(w_1^*)_{\cdot,j}\|_1,
$$
and
\begin{equation*}
\begin{split}
     &\text{tr}\left\{(\hat{w}_1-w^*_1)^\top \left(\frac{1}{T}\sum_{t=0}^{T-1}\mathcal{X}_{t,D_m^c}\mathcal{X}_{t,D_m^c}^\top\right)(\hat{w}_1-w^*_1)\right\}\\
     =&\sum_{j=1}^{d_1}\left((\hat{w}_1)_{\cdot,j}-(w_1^*)_{\cdot,j}\right)^\top \left(\frac{1}{T}\sum_{t=0}^{T-1}\mathcal{X}_{t,D_1^c}\mathcal{X}_{t,D_1^c}^\top\right)\left((\hat{w}_1)_{\cdot,j}-(w_1^*)_{\cdot,j}\right),
\end{split}
\end{equation*}
taking a union bound over $\{\hat{w}_m:m=1,\cdots,k\}$ and all columns of $\hat{w}_m$, proof is complete.
\end{proof}

\begin{proof}[Proof of Lemma \ref{Cov_mBound}]
The following established result can be applied here: 
\begin{Lemma}\label{InvBnd}
For any invertible matrix $B$, if $B+\Delta$ is also invertible, then
\begin{equation}
\|(B+\Delta)^{-1}-B^{-1}\|_2\leq \frac{\|B^{-1}\|_2^2\|\Delta\|_2}{1-\|B^{-1}\|_2\|\Delta\|_2}.
\end{equation}
\end{Lemma}
Since $\left\|I\right\|_2=1$, one can show that for $1\leq m\leq k$,
\begin{equation*}
    \begin{split}
        \left\|\Upsilon^{(m)\frac{1}{2}}\widehat{\Upsilon^{(m)}}^{-1}\Upsilon^{(m)\frac{1}{2}}-I\right\|_\infty\leq\left\|\Upsilon^{(m)\frac{1}{2}}\widehat{\Upsilon^{(m)}})^{-1}\Upsilon^{(m)\frac{1}{2}}-I\right\|_2\leq  \frac{\|\Delta\|_2}{1-\|\Delta\|_2},
    \end{split}
\end{equation*}
where $\Delta=\Upsilon^{(m)-\frac{1}{2}}\widehat{\Upsilon^{(m)}}\Upsilon^{(m)-\frac{1}{2}}-I$. Due to \eqref{Cov_mEigenBnd}, 
\begin{equation*}
\begin{split}
    &\left\|\Delta\right\|_2\leq \left(\Lambda_{\min}\Upsilon^{(m)}\right)^{-1}\left\|\widehat{\Upsilon^{(m)}}-\Upsilon^{(m)}\right\|_2\\
    \leq &C\left\|\widehat{\Upsilon^{(m)}}-\Upsilon^{(m)}\right\|_{F}\leq d_m \left\|\widehat{\Upsilon^{(m)}}-\Upsilon^{(m)}\right\|_\infty.
    \end{split}
\end{equation*}
In the following we bound $\left\|\widehat{\Upsilon^{(m)}}-\Upsilon^{(m)}\right\|_\infty$. Write $\widehat{\Upsilon^{(m)}}-\Upsilon^{(m)}$ as
\begin{align*}
\widehat{\Upsilon^{(m)}}-\Upsilon^{(m)}=&W_m^*\left(\frac{1}{T}\sum_{t=0}^{T-1}\mathcal{X}_t\mathcal{X}_t^\top-\Upsilon\right)W_m^{*\top}\\
&-(\hat{w}_m-w_m^*)^\top \frac{1}{T}\sum_{t=0}^{T-1}\mathcal{X}_{t,D_m^c}(\mathcal{X}_{t,D_m}-w_m^{*\top}\mathcal{X}_{t,D_m^c})^\top\\
&-\frac{1}{T}\sum_{t=0}^{T-1}(\mathcal{X}_{t,D_m}-w_m^{*\top}\mathcal{X}_{t,D_m^c}) \mathcal{X}_{t,D_m^c}^\top(\hat{w}_m-w_m^*)\\
&+(\hat{w}_m-w_m^*)^\top\left(\frac{1}{T}\sum_{t=0}^{T-1}\mathcal{X}_{t,D_m^c}\mathcal{X}_{t,D_m^c}^\top\right)(\hat{w}_m-w_m^*)\\
\triangleq &E_1^{(m)}-E_2^{(m)}-\left(E_2^{(m)}\right)^\top+E_3^{(m)},
\end{align*}
where $W_m^*$ is defined as in \eqref{W_def}.
Actually,
\begin{align*}
\|E_1^{(m)}\|_\infty&=\max_{i,j}\left|W_{m,i\cdot}^*\left(\frac{1}{T}\sum_{t=0}^{T-1}\mathcal{X}_t\mathcal{X}_t^\top-\Upsilon\right)W_{m,j\cdot}^{*\top}\right|\\
&=\max_{i,j}\left|\frac{1}{T}\sum_{t=0}^{T-1}\mathcal{X}_t^\top W_{m,i\cdot}^{*\top}W_{m,j\cdot}^*\mathcal{X}_t-\text{tr}(W_{m,i\cdot}^{*\top}W_{m,j\cdot}^*\Upsilon)\right|,
\end{align*}
which is the maximum over deviations of some quadratic forms from their expectation.
The following lemma provides a bound for quadratic form $\frac{1}{T}\sum_{t=0}^{T-1}\mathcal{X}_t^\top B\mathcal{X}_t$, with $B\in \mathbb{R}^{M\times M}$ being any symmetric matrix.

By Lemma \ref{QuaBound}, we only need to bound the trace norm and operator norm of 
$$
\frac{1}{2}\left((W_m^*)_{i\cdot}^\top (W_m^*)_{j\cdot}+(W_m^*)_{j\cdot}^\top (W_m^*)_{i\cdot}\right).
$$ The following lemma establishes the relationship between $\|\cdot\|_{\tr}$ and $\|\cdot\|_2$ for symmetric matrices.
\begin{Lemma}\label{trl2norm}
For any symmetric matrix $U$ of rank $r$, $\|U\|_{\tr}\leq r\|U\|_2$.
\end{Lemma}
Since $\frac{1}{2}\left((W_m^*)_{i\cdot}^\top(W_m^*)_{j\cdot}+(W_m^*)_{j\cdot}^\top(W_m^*)_{i\cdot}\right)$ is of rank 2, 
\begin{equation}\label{WtrBnd}
\begin{split}
    &\left\|\frac{1}{2}\left((W_{m}^*)_{i\cdot}^\top (W_{m}^*)_{j\cdot}+(W_{m}^*)_{j\cdot}^{\top}(W_m^*)_{i\cdot}\right)\right\|_{\tr}\\
    \leq &2\left\|\frac{1}{2}\left((W_{m}^*)_{i\cdot})^{\top}(W_{m}^*)_{j\cdot}+(W_{m}^*)_{j\cdot}^{\top}(W_{m}^*)_{i\cdot}\right)\right\|_2\\
    \leq &2\left\|(W_{m}^*)_{i\cdot}^{\top}(W_{m}^*)_{j\cdot}\right\|_2= 2\|(W_{m}^*)_{i\cdot}\|_2\|(W_{m}^*)_{j\cdot}\|_2.
\end{split}
\end{equation}
Meanwhile, similar from \eqref{Wl1Bnd}, we bound $\max_i \|(W_m^*)_{i\cdot}\|_2^2$ by
\begin{equation}\label{Wl2Bnd}
    \begin{split}
        \|(W_m^*)_{i\cdot}\|_2^2=&1+\|(w_m^*)_{\cdot,i}\|_2^2\\
        \leq &1+\Lambda_{\max}(\Upsilon_{D_m^c,D_m^c}^{-1})^2\left\|\Upsilon_{\cdot,i}\right\|_2^2\\
        \leq &1+\Lambda_{\min}(\Upsilon)^{-2}\Lambda_{\max}(\Upsilon)^2\leq C,
    \end{split}
    \end{equation}
    where the second inequality is due to that $\left\|\Upsilon_{\cdot,i}\right\|_2^2=(\Upsilon^2)_{ii}\leq \Lambda_{\max}(\Upsilon^2)\leq \Lambda_{\max}(\Upsilon)^2$.
    Thus, both the trace norm and $\ell_2$ norm of $\frac{1}{2}\left(W_{m,i\cdot}^{*\top}W_{m,j\cdot}^*+W_{m,j\cdot}^{*\top}W_{m,i\cdot}^*\right)$ can be bounded by constant, and applying Lemma \ref{QuaBound} gives us
\begin{equation*}
    \mathbb{P}\left(\|E_1^{(m)}\|_\infty>C\sqrt{\frac{\log M}{T}}\right)\leq c_1\exp\{-c_2\log M\}.
\end{equation*}
 Meanwhile, by Lemma \ref{true_w_dev} and Assumption \ref{Assump_w_bound}, with probability at least $1-c_1\exp\{-c_2\log M\}$,
\begin{align*}
\|E_2^{(m)}\|_\infty&\leq \left\|\frac{1}{T}\sum_{t=0}^{T-1}(\mathcal{X}_{t,D_m}-w_m^{*\top}\mathcal{X}_{t,D_m^c})\mathcal{X}_{t,D_m^c}\right\|_\infty\|\hat{w}_m-w_m^*\|_1\\
&\leq C\frac{s_m\log M}{T},
\end{align*}
and 
\begin{equation*}
    \begin{split}
        \left\|E_3^{(m)}\right\|_\infty=&\max_{i,j}\left|(E_3^{(m)})_{ij}\right|=\max_{i,j}(\hat{w}_m-w_m^*)_{\cdot,i}^\top H_{D_m^c,D_m^c}(\hat{w}_m-w_m^*)_{\cdot,j}\\
        \leq &\max_l(\hat{w}_m-w_m^*)_{\cdot,l}^\top H_{D_m^c,D_m^c}(\hat{w}_m-w_m^*)_{\cdot,l}\\
        \leq &\text{tr}\left\{(\hat{w}_m-w_m^*)^\top H_{D_m^c,D_m^c}(\hat{w}_m-w_m^*)\right\}\\
        \leq &C \frac{s_m\log M}{T}.
    \end{split}
\end{equation*}
Here the second line is because that $H_{D_m^c,D_m^c}=\frac{1}{T}\sum_{t=0}^{T-1}\mathcal{X}_{t,D_m^c}\mathcal{X}_{t,D_m^c}^\top$ is symmetric and positive semi-definite, thus we can apply Cauchey-Schwartz inequality. When $T\geq Cs^2\log M$.
$$
\frac{s_m\log M}{T}\leq \sqrt{\frac{\log M}{T}},
$$
which implies
\begin{align*}
\left\|\widehat{\Upsilon^{(m)}}-\Upsilon^{(m)}\right\|_2&\leq \|\widehat{\Upsilon^{(m)}}-\Upsilon^{(m)}\|_{F}\leq d_m\|\widehat{\Upsilon^{(m)}}-\Upsilon^{(m)}\|_\infty\leq C\sqrt{\frac{\log M}{T}}.
\end{align*}
Therefore, take a union bound over $1\leq m\leq k$, with probability at least $1-c_1\exp\{-c_2\log M\}$, 
\begin{equation*}
\left\|\Upsilon^{(m)\frac{1}{2}}\widehat{\Upsilon^{(m)}}^{-1}\Upsilon^{(m)\frac{1}{2}}-I\right\|_\infty\leq C\sqrt{\frac{\log M}{T}}.
\end{equation*}
when $T\geq Cs^2\log M$.
\end{proof}

\section{Proof of Theorem \ref{EstVarThm} and Theorem \ref{ConfRegThm}}
\begin{proof}[Proof of Theorem \ref{EstVarThm}]
Now we consider model \eqref{ARp2}, with unknown $\sigma^{*2}=\text{Var}(\epsilon_{ti})\geq \sigma_0^2$. Under this model, we use the notation $\widehat{U}_T$ for the quantity defined in the following:
\begin{equation*}
    \widehat{U}_T=T\sum_{m=1}^k \widehat{S}_m^\top (\widehat{\Upsilon^{(m)}})^{-1}\widehat{S}_m/\sigma^{*2}.
\end{equation*}
As explained in Section \ref{EstVarSec}, $\widehat{U}_T$ satisfies Theorem \ref{NullThm} and \ref{AltThm} under each corresponding condition. We show in the following that we only need to control the estimation error of $\hat{\sigma}^2$. 
Note that for any $0<\delta<1$,
\begin{equation*}
    \mathbb{P}\left(\widetilde{U}_T\leq x\right)\leq \mathbb{P}\left(\widehat{U}_T\leq \frac{x}{1-\delta}\right)+\mathbb{P}\left(\frac{\sigma^{*2}}{\hat{\sigma}^2}<1-\delta\right),
\end{equation*}
and
\begin{equation*}
    \mathbb{P}\left(\widetilde{U}_T> x\right)\leq \mathbb{P}\left(\widehat{U}_T> \frac{x}{1+\delta}\right)+\mathbb{P}\left(\frac{\sigma^{*2}}{\hat{\sigma}^2}> 1+\delta\right).
\end{equation*}
For any distribution function $F(x)$,
\begin{equation*}
\begin{split}
    \left|\mathbb{P}\left(\widetilde{U}_T\leq x\right)-F(x)\right|\leq &\sup_y \left|\mathbb{P}\left(\widehat{U}_T\leq y\right)-F(y)\right|+\sup_y \left|F(y)-F(y(1-\delta))\right|\\
    &+\mathbb{P}\left(\hat{\sigma}^2< \frac{\sigma^{*2}}{1+\delta}\right)+\mathbb{P}\left(\hat{\sigma}^2> \frac{\sigma^{*2}}{1-\delta}\right).
\end{split}
\end{equation*}
Recall that Theorem \ref{NullThm} and \ref{AltThm} establish bounds for $\mathbb{P}\left(\widehat{U}_T\leq x\right)-F_d(x)$ under $\mathcal{H}_0$, or under $\mathcal{H}_A$ with $\phi>\frac{1}{2}$, for $\mathbb{P}\left(\widehat{U}_T\leq x\right)-F_{d,\|\widetilde{\Delta}\|_2^2}(x)$ when $\phi=\frac{1}{2}$, and for $\mathbb{P}\left(\widehat{U}_T\leq x\right)$ when $0<\phi<\frac{1}{2}$. Thus we only need to bound $\mathbb{P}\left(\hat{\sigma}^2< \frac{\sigma^{*2}}{1+\delta}\right)$, $\mathbb{P}\left(\hat{\sigma}^2> \frac{\sigma^{*2}}{1-\delta}\right)$ and $\sup_y \left|F(y)-F(y(1-\delta))\right|$ with $F(x)=F_d(x)$ or $F(x)=F_{d,\|\widetilde{\Delta}\|_2^2}(x)$. Since $0<\delta<1$,
\begin{equation*}
    \mathbb{P}\left(\hat{\sigma}^2< \frac{\sigma^{*2}}{1+\delta}\right)+\mathbb{P}\left(\hat{\sigma}^2> \frac{\sigma^{*2}}{1-\delta}\right)\leq \mathbb{P}\left(|\hat{\sigma}^2-\sigma^{*2}|>\frac{\delta\sigma^{*2}}{2}\right)\leq \mathbb{P}\left(|\hat{\sigma}^2-\sigma^{*2}|>\frac{\delta\sigma_0^2}{2}\right).
\end{equation*}
Meanwhile,
\begin{equation*}
    \begin{split}
        \hat{\sigma}^2-\sigma^{*2}=&\frac{1}{MT}\sum_{t=0}^{T-1}\left\|X_{t+1}-\widehat{A}\mathcal{X}_t\right\|_2^2-\sigma^{*2}\\
        =&\frac{1}{MT}\sum_{t=0}^{T-1}\|\epsilon_t\|_2^2-\sigma^{*2}+\frac{1}{MT}\sum_{t=0}^{T-1}\left\|(\widehat{A}-A^*)\mathcal{X}_t\right\|_2^2\\
        &+\frac{2}{MT}\sum_{t=0}^{T-1}\left|\epsilon_t^\top (\widehat{A}-A^*)\mathcal{X}_t\right|\\
        =&\frac{1}{MT}\sum_{t=0}^{T-1}\|\epsilon_t\|_2^2-\sigma^{*2}+\frac{1}{M}\sum_{i=1}^M(\widehat{A}_i-A^*_i)^\top H(\widehat{A}_i-A_i^*)\\
        &+\frac{2}{M}\sum_{i=1}^M (\widehat{A}_i-A_i^*)^\top \left(\frac{1}{T}\sum_{t=0}^{T-1}\epsilon_{ti}\mathcal{X}_t\right).
    \end{split}
\end{equation*}
By Assumption \ref{Assump_A_bound} and Lemma \ref{DevBnd}, with probability at least $1-c_1\exp\{-c_2\log M\}$, 
\begin{equation*}
\begin{split}
    \frac{1}{M}\sum_{i=1}^M(\widehat{A}_i-A^*_i)^\top H(\widehat{A}_i-A_i^*)\leq C\frac{\rho \log M}{T}\leq C\sqrt{\frac{\rho \log M}{T}},
\end{split}
\end{equation*}
and
\begin{equation*}
\begin{split}
    \frac{2}{M}\sum_{i=1}^M(\widehat{A}_i-A^*_i)^\top H\left(\frac{1}{T}\sum_{t=0}^{T-1}\epsilon_{ti} \mathcal{X}_t\right)\leq &2\max_i\left\|\widehat{A}_i-A_i^*\right\|_1\left(\frac{1}{T}\sum_{t=0}^{T-1}\epsilon_t\mathcal{X}_t^\top\right)\\
    \leq &C\frac{\rho \log M}{T}\leq C\sqrt{\frac{\rho \log M}{T}}.
\end{split}
\end{equation*}
Also, since $\epsilon_{ti}$ are independent sub-Gaussian random variables with scale factor $C\sigma^*$, the first term can be bounded by Bernstein type inequality of sub-exponential random variables(see proposition 5.16 in \cite{vershynin2010introduction}):
\begin{equation*}
    \mathbb{P}\left(\left|\frac{1}{MT}\sum_{t=0}^{T-1}\|\epsilon_t\|_2^2-\sigma^{*2}\right|>\frac{\delta\sigma^{*2}}{2} \right)\leq 2\exp\left\{-cMT\min\{\delta^2,\delta\}\right\}.
\end{equation*}
Let $\delta=C\sqrt{\frac{\rho \log M}{T}}$, then 
\begin{equation*}
    \begin{split}
        &\mathbb{P}\left(\hat{\sigma}^2< \frac{\sigma^{*2}}{1+\delta}\right)+\mathbb{P}\left(\hat{\sigma}^2> \frac{\sigma^{*2}}{1-\delta}\right)\\
        \leq &2\exp\left\{-c_1\rho M\log M\right\}+c_2\exp\{-c_3\log M\}.
    \end{split}
\end{equation*}
While for $\sup_x F_{d,\|\mu\|_2^2}(x)-F_{d,\|\mu\|_2^2}\left(x(1-\delta)\right)$ with any $\mu\in \mathbb{R}^d$ satisfying $\|\mu\|_2\leq C$, if $\delta<\frac{1}{2}$,
\begin{equation*}
    \begin{split}
        &F_{d,\|\mu\|_2^2}(x)-F_{d,\|\mu\|_2^2}\left(x(1-\delta)\right)\\
        =&\mathbb{P}\left(\left\|Z+\mu\right\|_2^2\in \left(x(1-\delta),x\right]\right)\\
        \leq&C(d)\left(x^{\frac{d}{2}}-(x(1-\delta))^{\frac{d}{2}}\right)\sup_{\|z+\mu\|_2^2\in (x(1-\delta),x]}e^{-\|z\|_2^2/2}\\
        \leq& C(d)\delta x^{\frac{d}{2}}\exp\left\{-\frac{1}{2}\left(\sqrt{x(1-\delta)}-\|\mu\|_2\right)^21(\sqrt{x(1-\delta)}\geq \|\mu\|_2)\right\}.
    \end{split}
\end{equation*}
Here $Z\in \mathbb{R}^d$ is a standard Gaussian random vector, the third line is due to that the density of $Z$ is $(2\pi)^{-\frac{d}{2}}e^{-\|z\|_2^2/2}$, and the fourth line applies the fact that when $0< \delta<\frac{1}{2}$,
\begin{equation*}
    \left[1-(1-\delta)^{\frac{d}{2}}\right]\leq \frac{d}{2}\sup_{\xi\in (1-\delta,1)}\xi^{\frac{d}{2}-1}\delta=\frac{d}{2}(1-\delta)^{(\frac{d}{2}-1)1(d\leq 2)}\delta\leq C(d)\delta.
\end{equation*}
Meanwhile, when $\sqrt{x(1-\delta)}<\|\mu\|_2$, 
\begin{equation*}
    x^{\frac{d}{2}}\leq \frac{\|\mu\|_2^d}{(1-\delta)^{\frac{d}{2}}}\leq C(d),
\end{equation*} 
and when $\sqrt{x(1-\delta)}\geq\|\mu\|_2$, 
\begin{equation*}
\begin{split}
    &x^{\frac{d}{2}}\exp\left\{-\frac{1}{2}\left(\sqrt{x(1-\delta)}-\|\mu\|_2\right)^21(\sqrt{x(1-\delta)}\geq \|\mu\|_2)\right\}\\
    \leq &\sup_{y\geq 0}(y+C)^de^{-y^2/2}\leq C(d),
\end{split}
\end{equation*}
which implies 
\begin{equation*}
    F_{d,\|\mu\|_2^2}(x)-F_{d,\|\mu\|_2^2}\left(x(1-\delta)\right)\leq C(d)\delta.
\end{equation*}
To see why all the bounds for $\widehat{U}_T$ still hold for $\widetilde{U}_T$, note that we only need to add $C\sqrt{\frac{\rho \log M}{T}}+2\exp\left\{-c_1\rho M\log M\right\}+c_2\exp\{-c_3\log M\}$ to the bounds under $\mathcal{H}_0$, and under $\mathcal{H}_A$ when $\phi\geq \frac{1}{2}$, which only changes the constant factors of the previous bounds. For the bound under $\mathcal{H}_A$ when $0<\phi<\frac{1}{2}$, we substitute $x$ by $\frac{x}{1-\delta}$ with $\delta=C\sqrt{\frac{\log M}{T}}$, and add $2\exp\left\{-c_1\rho M\log M\right\}+c_2\exp\{-c_3\log M\}$, which only changes the constant factors as well. Therefore, all the conclusions for $\widehat{U}_T$ in Theorem \ref{NullThm} and \ref{AltThm} still hold for $\widetilde{U}_T$ under each corresponding condition.
\end{proof}
\begin{proof}[Proof of Theorem \ref{ConfRegThm}]
First we show the connection between $R_T$ and $\widetilde{U}_T$. Note that
\begin{equation*}
    \begin{split}
        \widetilde{S}_m=&-\frac{1}{T}\sum_{t=0}^{T-1}\left(\mathcal{X}_{t,D_m}-\hat{w}_m^\top \mathcal{X}_{t,D_m^c}\right)\left(X_{t+1,m}-\widehat{A}_m^\top \mathcal{X}_t\right)\\
        &=\widehat{S}_m+\left[\frac{1}{T}\sum_{t=0}^{T-1}\left(\mathcal{X}_{t,D_m}-\hat{w}_m^\top \mathcal{X}_{t,D_m^c}\right)\mathcal{X}_{t,D_m}^\top \right]\left(\left(\widehat{A}_m\right)_{D_m}-\left(A_m^*\right)_{D_m}\right)\\
        &=\widehat{S}_m+\widetilde{\Upsilon^{(m)}}\left(\left(\widehat{A}_m\right)_{D_m}-\left(A_m^*\right)_{D_m}\right),
    \end{split}
\end{equation*}
which implies
\begin{equation*}
    \begin{split}
        \hat{a}(m)-(A_m^*)_{D_m}=&(\widehat{A_m})_{D_m}-(A_m^*)_{D_m}-\left(\widetilde{\Upsilon^{(m)}}\right)^{-1}\widetilde{S}_m=-\left(\widetilde{\Upsilon^{(m)}}\right)^{-1}\widehat{S}_m.
    \end{split}
\end{equation*}
Thus
\begin{equation*}
\begin{split}
    R_T=&\frac{T}{\hat{\sigma}^2}\sum_{m=1}^k \left(\hat{a}(m)-(A_m^*)_{D_m}\right)^{\top} \widehat{\Upsilon^{(m)}}\left(\hat{a}(m)-(A_m^*)_{D_m}\right)\\
    =&\frac{T}{\hat{\sigma}^2}\sum_{m=1}^k \widehat{S}_m^\top \left(\widetilde{\Upsilon^{(m)}}^\top\right)^{-1}\widehat{\Upsilon^{(m)}}\left(\widetilde{\Upsilon^{(m)}}\right)^{-1}\widehat{S}_m,
\end{split}
\end{equation*}
and the only difference between $R_T$ and $\widetilde{U}_T$ is that we substitute $\left(\widehat{\Upsilon^{(m)}}\right)^{-1}$ by $\left(\widetilde{\Upsilon^{(m)}}^\top\right)^{-1}\widehat{\Upsilon^{(m)}}\left(\widetilde{\Upsilon^{(m)}}\right)^{-1}$. We only need to prove that $\left(\widetilde{\Upsilon^{(m)}}^\top\right)^{-1}\widehat{\Upsilon^{(m)}}\left(\widetilde{\Upsilon^{(m)}}\right)^{-1}$ satisfies Assumption \ref{Assump_Cov_m_bound}. The argument is very similar to the proof of Lemma \ref{Cov_mBound}, but we need to bound $\left\|\widetilde{\Upsilon^{(m)}}\left(\widehat{\Upsilon^{(m)}}\right)^{-1}\widetilde{\Upsilon^{(m)}}^\top-\Upsilon^{(m)}\right\|_{\infty}$ instead of $\left\|\widehat{\Upsilon^{(m)}}-\Upsilon^{(m)}\right\|_{\infty}$ here.

Let $E=\widetilde{\Upsilon^{(m)}}-\widehat{\Upsilon^{(m)}}$, then
\begin{equation*}
    \begin{split}
        &\widetilde{\Upsilon^{(m)}}\left(\widehat{\Upsilon^{(m)}}\right)^{-1}\widetilde{\Upsilon^{(m)}}^\top\\
        =&\left(\widehat{\Upsilon^{(m)}}+E\right)\left(\widehat{\Upsilon^{(m)}}\right)^{-1}\left(\widehat{\Upsilon^{(m)}}+E^\top\right)\\
        =&\widehat{\Upsilon^{(m)}}+E+E^\top+E\left(\widehat{\Upsilon^{(m)}}\right)^{-1}E^\top.
    \end{split}
\end{equation*}
Recall that when proving Lemma \ref{Cov_mBound}, we already upper bound $\left\|\widehat{\Upsilon^{(m)}}-\Upsilon^{(m)}\right\|_{\infty}$ by $C\sqrt{\frac{\log M}{T}}$ with probability at least $1-c_1\exp\{-c_2\log M\}$. Thus for any vector $u\in \mathbb{R}^{d_m}$ s.t $\|u\|_2=1$,
\begin{equation*}
\begin{split}
     u^\top\widehat{\Upsilon^{(m)}}u=&u^\top \Upsilon^{(m)}u+u^\top \left(\widehat{\Upsilon^{(m)}}-\Upsilon^{(m)}\right)u\\
     \geq &\Lambda_{\min}\left(\Upsilon^{(m)}\right)-d_m\left\|\widehat{\Upsilon^{(m)}}-\Upsilon^{(m)}\right\|_{\infty}\geq C,
\end{split}
\end{equation*}
which implies $\Lambda_{\max}\left(\left(\widehat{\Upsilon^{(m)}}\right)^{-1}\right)\leq C$, and $\left\|E\left(\widehat{\Upsilon^{(m)}}\right)^{-1}E^\top\right\|_{\infty}\leq Cd_m\|E\|_{\infty}$. We bound $\|E\|_{\infty}$ in the following. One can show that
\begin{equation*}
\begin{split}
    \|E\|_{\infty}=&\left\|\frac{1}{T}\sum_{t=0}^{T-1}\left(\mathcal{X}_{t,D_m}-\hat{w}_m^\top \mathcal{X}_{t,D_m^c}\right)\mathcal{X}_{t,D_m^c}^\top \hat{w}_m\right\|_{\infty}\\
    \leq &\left\|\frac{1}{T}\sum_{t=0}^{T-1}\left(\mathcal{X}_{t,D_m}-w_m^{*\top} \mathcal{X}_{t,D_m^c}\right)\mathcal{X}_{t,D_m^c}^\top\right\|_{\infty}\left(\|w_m^*\|_1+\|\hat{w}_m-w_m^*\|_1\right)\\
    &+\max_{i,j}\left|\left((\hat{w}_m-w_m^*)\right)_{\cdot i}^\top \frac{1}{T}\sum_{t=0}^{T-1}\left(\mathcal{X}_{t,D_m^c}\mathcal{X}_{t,D_m^c}^\top\right)\left((\hat{w}_m-w_m^*)\right)_{\cdot j}\right|\\
    &+\left\|\frac{1}{T}\sum_{t=0}^{T-1}\mathcal{X}_{t,D_m^c}\mathcal{X}_{t,D_m^c}^\top w_m^*\right\|_{\infty}\|\hat{w}_m-w_m^*\|_1.
\end{split}
\end{equation*}
Applying \eqref{Cov_mEigenBnd}, \eqref{Wl2Bnd}, Lemma \ref{CovBound}, we have
\begin{equation}\label{UpsilonWBnd}
\begin{split}
    &\left\|\frac{1}{T}\sum_{t=0}^{T-1}\mathcal{X}_{t,D_m^c}\mathcal{X}_{t,D_m^c}^\top w_m^*\right\|_{\infty}\\
    \leq &\left\|\Upsilon_{D_m^c,D_m^c}w_m^*\right\|_{\infty}+\left\|\frac{1}{T}\sum_{t=0}^{T-1}\mathcal{X}_{t,D_m^c}\mathcal{X}_{t,D_m^c}^\top-\Upsilon_{D_m^c,D_m^c}\right\|_{\infty}\|w_m^*\|_1\\
    \leq &\Lambda_{\max}(\Upsilon)\max_i\|(w_m^*)_{\cdot,i}\|_2+C\frac{s_m\log M}{T}\leq C.
\end{split}
\end{equation}
Thus, with Lemma \ref{true_w_dev}, Assumption \ref{Assump_w_bound}, and \eqref{UpsilonWBnd}, we show that with probability at least $1-c_1\exp\{-c_2\log M\}$,
\begin{equation*}
    \|E\|_{\infty}\leq C\sqrt{\frac{\log M}{T}}+C\frac{s_m\log M}{T}+Cs_m\sqrt{\frac{\log M}{T}}\leq Cs_m\sqrt{\frac{\log M}{T}}.
\end{equation*}
Therefore, using the same arguments as in the proof of Lemma \ref{Cov_mBound},
\begin{equation*}
\begin{split}
    &\left\|\Upsilon^{(m)\frac{1}{2}}\widetilde{\Upsilon^{(m)}}\left(\widehat{\Upsilon^{(m)}}\right)^{-1}\widetilde{\Upsilon^{(m)}}^\top\Upsilon^{(m)\frac{1}{2}}-I\right\|_2\\
    \leq&C\left\|\widetilde{\Upsilon^{(m)}}\left(\widehat{\Upsilon^{(m)}}\right)^{-1}\widetilde{\Upsilon^{(m)}}^\top-\Upsilon^{(m)}\right\|_2\\
    \leq& Cd_m\left\|\widetilde{\Upsilon^{(m)}}\left(\widehat{\Upsilon^{(m)}}\right)^{-1}\widetilde{\Upsilon^{(m)}}^\top-\Upsilon^{(m)}\right\|_{\infty}\\
    \leq &C\left\|\widehat{\Upsilon^{(m)}}-\Upsilon^{(m)}\right\|_{\infty}+C\left\|E\right\|_{\infty}\\
    \leq &Cs_m\sqrt{\frac{\log M}{T}}.
\end{split}
\end{equation*}
By Lemma \ref{InvBnd}, 
\begin{equation*}
    \left\|\Upsilon^{(m)-\frac{1}{2}}\left(\widetilde{\Upsilon^{(m)}}^\top\right)^{-1}\widehat{\Upsilon^{(m)}}\left(\widetilde{\Upsilon^{(m)}}\right)^{-1}\Upsilon^{(m)-\frac{1}{2}}-I\right\|_{\infty}\leq Cs_m\sqrt{\frac{\log M}{T}}.
\end{equation*}

\end{proof}
\section{Proof of Lemmas in Section \ref{main_proof}}\label{lemma_proof2}
\begin{proof}[Proof of Lemma \ref{CLT}]
	Let
	$$
	\xi_{T,t}=-\frac{1}{\sqrt{T}}\begin{pmatrix}
	\epsilon_{t,1}(\Upsilon^{(1)})^{-\frac{1}{2}}W_1^*\mathcal{X}_t\\
	\vdots\\
	\epsilon_{t,k}(\Upsilon^{(k)})^{-\frac{1}{2}}W_k^*\mathcal{X}_t
	\end{pmatrix}.
	$$
	Define filtration $\mathcal{F}_{T,t}=\sigma(X_{-p+1},X_{-p+2},\cdots, X_{t+1})$, then $(\xi_{Tt},\mathcal{F}_{Tt})_{0\leq t\leq T-1}$ is a martingale difference sequence, and $V_T=\sum_{t=0}^{T-1}\xi_{T,t}$. To bound the convergence rate, we are going to use a modified version of Lemma 4 in \textsl{Grama and Haeusler} (2006).
	\begin{Lemma}\label{convergence.rate}
		Let $(\xi_{ni}, \mathcal{F}_{ni})_{0\leq i\leq n}$ be a martingale difference sequence taking values in $\mathbb{R}^d$. Let $X_k^n=\sum_{i=1}^k \xi_{ni}$, and $\left\langle X^n\right\rangle_k=\sum_{i=1}^k a_{ni}\triangleq \sum_{i=1}^k \mathbb{E}(\xi_{ni}\xi_{ni}^\top|\mathcal{F}_{n,i-1})$. Define $R_{\delta}^{n,d}=L_{\delta}^{n,d}+N_{\delta}^{n,d}$, $$L_\delta^{n,d}=\sum_{i=1}^n \mathbb{E}\|\xi_{ni}\|_2^{2+2\delta},
		N_\delta^{n,d}=\mathbb{E}\|\left\langle X^n\right\rangle _n-I\|_{\tr}^{1+\delta}.$$
		Then $\forall \mu\in \mathbb{R}^d, r\geq 0, 0<\delta\leq\frac{1}{2}$, when $R_{\delta}^{n,d}\leq 1$,
		$$
		\mathbb{P}(\|X_n^n+\mu\|_2\geq r)-\mathbb{P}(\|Z+\mu\|_2\geq r)\leq C(\|\mu\|_2,d,\delta)\left(R_\delta^{n,d}\right)^{\frac{1}{3+2\delta}},
		$$
		where $Z_{d\times 1}\sim \mathcal{N}(0,I)$, $C(\|\mu\|_2,d,\delta)$ is non-decreasing as $\|\mu\|_2$ increases. 
	\end{Lemma}
By Lemma \ref{convergence.rate}, to bound $\sup_{x>0,}\left|\mathbb{P}(\|V_T+\mu\|_2^2\leq x)-F_{d,\|\mu\|_2^2}(x)\right|$, we only need to bound $R_{\delta}^{T,d}=L_{\delta}^{T,d}+N_{\delta}^{T,d}$. 
	\begin{align*}
	L_{\delta}^{T,d}&=\sum_{t=0}^{T-1}\mathbb{E}\left(\|\xi_{T,t}\|_2^{2+2\delta}\right)\\
	&\leq CT^{-(1+\delta)}\sum_{t=1}^T \mathbb{E}\left(\sum_{m=1}^k \|W_m^*\mathcal{X}_t\|_2^2\epsilon_{t,m}^2\right)^{1+\delta}\\
	&\leq CT^{-(1+\delta)}\sum_{t=0}^{T-1}k^{\delta}\sum_{m=1}^k \mathbb{E}\left(|\epsilon_{t,m}|^{2+2\delta}\|W_m^*\mathcal{X}_t\|_2^{2+2\delta}\right)\\
	&=T^{-\delta}k^{\delta}C(\delta)\sum_{m=1}^k \mathbb{E}\left(\|W_m^*\mathcal{X}_0\|_2^{2+2\delta}\right)
	\end{align*}
	Here the second line is due to $\Lambda_{\min}(\Upsilon^{(m)})\geq 1$, and the third line is due to $f(x)=x^{1+\delta}$ is a convex function. More specifically,
	\begin{align*}
	    &\left(\sum_{m=1}^k \|W_m^*\mathcal{X}_t\|_2^2\epsilon_{t,K}^2\right)^{1+\delta}\leq \frac{1}{k}\sum_{m=1}^k\left(k\|W_m^*\mathcal{X}_t\|_2^2\epsilon_{t,K}^2\right)^{1+\delta}=k^{\delta}\sum_{m=1}^k\left(\|W_m^*\mathcal{X}_t\|_2^{2+2\delta}\epsilon_{t,K}^{2+2\delta}\right).
	\end{align*}
	While for the last line, since $\epsilon_{t,m}$ is sub-Gaussian with parameter $\tau$, $\mathbb{E}|\epsilon_{t,m}|^{2+2\delta}\leq C(\delta)$. Note that $d,\beta,\tau$ are all viewed as constants here. Due to the sub-Gaussianity of $\epsilon_{t,i}$'s, we have the following lemma.
	\begin{Lemma}\label{subGaussnorm}
	    $$
	    \mathbb{E}\left(\|W_m^*\mathcal{X}_t\|_2^q\right)^{\frac{1}{q}}\leq Cq\quad \text{for all }q\geq 1.
	    $$ 
	\end{Lemma}
	Therefore, 
	$$
	\mathbb{E}\left(\|W_m^*\mathcal{X}_0\|_2^{2+2\delta}\right)\leq C(\delta),
	$$
	which implies
	$$
	L_{\delta}^{T,d}\leq C(\delta)T^{-\delta}.
	$$
	While for $N_{\delta}^{T,d}$, since 
	\begin{align*}
	&\sum_{t=0}^{T-1}\mathbb{E}\left(\xi_{T,t}\xi_{T,t}^\top|\mathcal{F}_{T,t-1}\right)-I\\
	=&\begin{pmatrix}
	(\Upsilon^{(1)})^{-\frac{1}{2}}B_1(\Upsilon^{(1)})^{-\frac{1}{2}}&\cdots&\cdots&0\\
	0&(\Upsilon^{(2)})^{-\frac{1}{2}}B_2(\Upsilon^{(2)})^{-\frac{1}{2}}&\cdots&0\\
	\vdots&\ddots&\ddots&\vdots\\
	0&\cdots&\cdots&(\Upsilon^{(k)})^{-\frac{1}{2}}B_k(\Upsilon^{(k)})^{-\frac{1}{2}}
	\end{pmatrix},
	\end{align*}
	where $B_m=W_m^*\left(\frac{1}{T}\sum_{t=0}^{T-1}\mathcal{X}_t\mathcal{X}_t^\top-\Upsilon\right)W_m^{*\top}$,
	\begin{align*}
	N_{\delta}^{T,d}=
	&\mathbb{E}\left(\left(\sum_{m=1}^k\left\|(\Upsilon^{(m)})^{-\frac{1}{2}}B_m(\Upsilon^{(m)})^{-\frac{1}{2}}\right\|_{\tr}\right)^{1+\delta}\right)\\
	\leq&\mathbb{E}\left(\left(\sum_{m=1}^kd_m\left\|(\Upsilon^{(m)})^{-\frac{1}{2}}B_m(\Upsilon^{(m)})^{-\frac{1}{2}}\right\|_2\right)^{1+\delta}\right)\\	
	\leq&\mathbb{E}\left(\left(\sum_{m=1}^kd_m^2\|B_m\|_{\infty}\right)^{1+\delta}\right),
	\end{align*}
	where the second line is because that $(\Upsilon^{(m)})^{-\frac{1}{2}}B_m(\Upsilon^{(m)})^{-\frac{1}{2}}$ is of rank at most $d_m$, and we can apply Lemma \ref{trl2norm}; the last line is due to 
	\begin{align*}
	    \|B_m\|_2=\sup_{\|u\|_2=1}\|B_mu\|_2\leq \sup_{\|u\|_2=1}\sqrt{d_m}\|B_mu\|_{\infty}\leq \sup_{\|u\|_2=1}\sqrt{d_m}\|B_m\|_{\infty}\|u\|_1=d_m\|B_m\|_\infty.
	\end{align*}
	Since 
	$$
	(B_m)_{ij}=\frac{1}{T}\sum_{t=0}^{T-1}\mathcal{X}_t^\top (W_m^*)_{i\cdot}^\top(W_m^*)_{j\cdot}\mathcal{X}_t-\text{tr}\left((W_m^*)_{i\cdot}^\top(W_m^*)_{j\cdot}\Upsilon\right),
	$$ 
	by Lemma \ref{QuaBound}, we only need to bound the operator norm and trace norm of 
	$$
	\frac{1}{2}\left((W_m^*)_{i\cdot}^\top(W_m^*)_{j\cdot}+(W_m^*)_{j\cdot}^\top(W_m^*)_{i\cdot}\right).
	$$
	By \eqref{WtrBnd} and \eqref{Wl2Bnd}, we have the following:
	\begin{align*}
	    &\left\|\frac{1}{2}\left((W_m^*)_{i\cdot}^\top(W_m^*)_{j\cdot}+(W_m^*)_{j\cdot}^\top(W_m^*)_{i\cdot}\right)\right\|_{\tr}\\
	    \leq& 2\left\|\frac{1}{2}\left((W_m^*)_{i\cdot}^\top(W_m^*)_{j\cdot}+(W_m^*)_{j\cdot}^\top(W_m^*)_{i\cdot}\right)\right\|_2\leq C.
	\end{align*}
	Therefore, applying Lemma \ref{QuaBound} leads us to 
	\begin{align*}
	    &\mathbb{P}\left(\left(\sum_{m=1}^kd_m^2\|B_m\|_{\infty}\right)^{1+\delta}>x\right)\\
	    \leq &\sum_{m=1}^k\mathbb{P}\left(\|B_m\|_{\infty}>\frac{x^{\frac{1}{1+\delta}}}{d^2}\right)\\
	    \leq &c_1\exp\left\{-c_2T\min\left\{x^{\frac{1}{1+\delta}},x^{\frac{2}{1+\delta}}\right\}\right\},
	\end{align*}
	which implies
	\begin{align*}
	N_{\delta}^{T,d}\leq &\int_0^\infty \mathbb{P}\left(\left(\sum_{m=1}^k d_m^2\|B_m\|_{\infty}\right)^{1+\delta}>x\right) dx\\
	\leq & \int_0^\infty c_1\exp\left\{-c_2T\min\left\{x^{\frac{2}{1+\delta}},x^{\frac{1}{1+\delta}}\right\}\right\} dx\\
	\leq & C(\delta)\left(\int_0^1 u^{\delta}\exp\{-cTu^2\}du+\int_1^{\infty} u^{\delta}\exp\{-cTu\} du\right)\\
	\leq & C(\delta)\left(T^{-\frac{1+\delta}{2}}\Gamma\left(\frac{1+\delta}{2}\right)+T^{-1-\delta}\Gamma(1+\delta)\right)\\
	\leq &C(\delta)T^{-\frac{1+\delta}{2}}.
	\end{align*}
	Thus,
	$$
	R_\delta^{T,d}=N_{\delta}^{T,d}+L_{\delta}^{T,d}\leq C(\delta)\left(T^{-\delta}+T^{-\frac{1+\delta}{2}}\right).
	$$
	By Lemma \ref{convergence.rate}, for any $x\geq 0$, $\mu\in \mathbb{R}^d$, and $0\leq\delta\leq \frac{1}{2}$, when $T>C(\delta)$,
	\begin{align*}
	\left|\mathbb{P}\left(\|V_T+\mu\|_2^2\leq x\right)-F_{d,\|\mu\|_2^2}(x)\right| \leq C(\|\mu\|_2,\delta)\left(R_\delta^{T,d}\right)^{\frac{1}{3+2\delta}}.
	\end{align*}
	The best rate is achieved when $\delta=\frac{1}{2}$, and thus when $T>C$,
	\begin{align*}
	\sup_{x\geq 0}\left|\mathbb{P}\left(\|V_T+\mu\|_2^2\leq x\right)-F_{d,\|\mu\|_2^2}(x)\right| \leq C(\|\mu\|_2)T^{-\frac{1}{8}},
	\end{align*}
	\end{proof}
	
	\begin{proof}[Proof of Lemma \ref{Cov_mEigenBnd}]
	We prove the lower and upper bounds for eigenvalues of $\Upsilon$, by establishing a connection between our stability condition \eqref{StbARpsG} and another spectral density based condition proposed in \cite{basu2015regularized}. First we introduce the following lemma, which is a direct result of proposition 2.3 and (2.6) in \cite{basu2015regularized} under our setting.
	\begin{Lemma}\label{CovEigenBndGen}
	Under the model specified in \eqref{ARp2} with independent noise $\epsilon_{ti}$ of unit variance, the eigenvalues of $\Upsilon$ can be bounded as follows:
	\begin{equation*}
	    \left(\mu_{\max}(\mathcal{A})\right)^{-1}\leq \Lambda_{\min}(\Upsilon)\leq \Lambda_{\max}(\Upsilon)\leq \left(\mu_{\min}(\mathcal{A})\right)^{-1},
	\end{equation*}
	where $\mu_{\min}(\mathcal{A})=\min_{|z|=1}\Lambda_{\min}\left(\mathcal{A}^*(z)\mathcal{A}(z)\right)$, and $\mu_{\max}(\mathcal{A})=\max_{|z|=1}\Lambda_{\max}\left(\mathcal{A}^*(z)\mathcal{A}(z)\right)$.
	\end{Lemma}
	By Lemma \ref{CovEigenBndGen}, we only need to prove that condition \eqref{StbARpsG} implies a lower bound for $\mu_{\min}(\mathcal{A})$ and upper bound for $\mu_{\max}(\mathcal{A})$. First note that
\begin{equation*}
\begin{split}
        \mu_{\min}(\mathcal{A})=&\min_{|z|=1}\Lambda_{\min}\left(\mathcal{A}(z)\mathcal{A}^*(z)\right)\\
        =&\min_{|z|=1}\inf_u\frac{\left\|\mathcal{A}^*(z)u\right\|_2^2}{\|u\|_2^2}\\
        =&\min_{|z|=1}\inf_v\frac{\|v\|_2^2}{\left\|\left(\mathcal{A}^*(z)\right)^{-1}v\right\|_2^2}\\
        =&\min_{|z|=1}\left(\left\|\left(\mathcal{A}^*(z)\right)^{-1}\right\|_2\right)^{-2},
\end{split}
\end{equation*}
where the last equality is due to that $\left\|\left(\mathcal{A}^*(z)\right)^{-1}\right\|_2=\sup_v\frac{\left\|\left(\mathcal{A}^*(z)\right)^{-1}v\right\|_2}{\|v\|_2}$. Meanwhile, for any $|z|=1$,
\begin{equation*}
\begin{split}
    \left\|\left(\mathcal{A}^*(z)\right)^{-1}\right\|_2=\left\|\mathcal{A}^{-1}(z)\right\|_2=& \left\|\sum_{j=0}^{\infty}\Psi_jz^j\right\|_2\leq \sum_{j=0}^{\infty}\|\Psi_j\|_2\leq \beta,
\end{split}
\end{equation*}
where we apply condition \eqref{StbARpsG} in the last inequality. Thus $ \mu_{\min}(\mathcal{A})\geq \beta^{-2}$.

While for bounding $\mu_{\max}(\mathcal{A})$, we start by bounding $\|A_n\|_2$ for $0\leq n\leq p$. Here we define $A_0=I_{M\times M}$, and $A_n=0$ for all $n>p$. Since 
\begin{equation*}
    I=\mathcal{A}^{-1}(z)\mathcal{A}(z)=\left(\sum_{j=0}^{\infty}\Psi_j z^j\right)\left(\sum_{i=0}^pA_i z^i\right)=\sum_{n=0}^{\infty}\left(\sum_{i=0}^{\infty}\Psi_iA_{n-i}\right)z^n,
\end{equation*}
one can show that $\Psi_0=I$, and $\sum_{i=0}^{n} \Psi_iA_{n-i}=0$ for $n\geq 1$. Thus 
$$
A_n=-\sum_{i=1}^n \Psi_iA_{n-i} \text{ for } n\geq 1,
$$ and $\|A_n\|_2\leq \sum_{i=1}^n \|\Psi_i\|_2\|A_{n-i}\|_2$. We have the following claim:
\begin{equation}\label{ARpAopnorm}
    \text{For } 0\leq n\leq p, \quad \|A_n\|_2\leq \beta^n \vee 1.
\end{equation}
This can be proved by induction. It is clear that $\|A_0\|_2=\|I\|_2=\beta^0$, and if \eqref{ARpAopnorm} holds for $0\leq n=k\leq p$, 
\begin{equation*}
    \|A_{k+1}\|_2\leq \sum_{i=1}^n \|\Psi_i\|_2(\beta^{n-i}\vee 1)\leq \beta\max_i(\beta^{n-i}\vee 1)\leq \beta^n\vee 1.
\end{equation*}
Therefore, $\mu_{\max}(\mathcal{A})$ can be bounded in the following:
\begin{equation*}
\begin{split}
    \mu_{\max}(\mathcal{A})=&\max_{|z|=1}\Lambda_{\max}\left(\mathcal{A}(z)\mathcal{A}^*(z)\right)\\
    =&\max_{|z|=1}\left\|\mathcal{A}^*(z)\right\|_2^2\\
        \leq &\left(\sum_{i=0}^p \|A_i\|_2\right)^2\\
        \leq &\left(\frac{\beta^{p+1}-1}{\beta-1}\right)^21(\beta>1)+(p+1)^21(0\leq \beta\leq 1).
\end{split}
\end{equation*}
With Lemma \ref{CovEigenBndGen}, we conclude that
\begin{equation*}
    C_1(\beta)\leq \Lambda_{\min}(\Upsilon)\leq \Lambda_{\max}(\Upsilon)\leq C_2(\beta),
\end{equation*}
where $C_1(\beta)=\left(\frac{1-\beta}{1-\beta^{p+1}}\right)^{2}1(\beta>1)+(p+1)^{-2}1(0\leq \beta\leq 1)$, and $C_2(\beta)=\beta^2$.
\end{proof}

	\begin{proof}[Proof of Lemma \ref{DevBnd}]
    Recall that $X_t=\sum_{j=0}^{\infty}\Psi_j\epsilon_{t-j-1}$. Define $\Psi_j^{(p)}\in \mathbb{R}^{pM\times M}$ as the following:
    \begin{equation}\label{PsipDef}
        \Psi_{j}^{(p)}=\begin{pmatrix}
        \Psi_j1(j\geq 0)\\
        \vdots\\
        \Psi_{j-p+1}1(j-p+1\geq 0)
        \end{pmatrix},
    \end{equation}   
    then we can also write $\mathcal{X}_t$ as an infinite sum $\mathcal{X}_t=\sum_{j=0}^{\infty}\Psi_j^{(p)}\epsilon_{t-j-1}$. Without loss of generality, we consider the first entry of $\frac{1}{T}\sum_{t=0}^{T-1}\epsilon_t\mathcal{X}_t^\top$:
    \begin{equation}\label{FstEntryDev}
         \frac{1}{T}\sum_{t=0}^{T-1}\epsilon_{t,1}\sum_{j=0}^{\infty}(\Psi_j)_{1\cdot}\epsilon_{t-j-1}.
    \end{equation}
    In the following, we tackle the infinite sum in \eqref{FstEntryDev}, by focusing our analysis on the finite sum and let the residue converges to 0. Rigorously, for any positive integer $m$, let $$\tilde{\epsilon}=(\epsilon_{-m-1}^\top,\dots,\epsilon_{T-1}^\top)^\top,\quad \eta^{(t)}=((\Psi_{t+m})_{1\cdot},\dots,(\Psi_0)_{1,\cdot},0,\dots,0)^\top\in \mathbb{R}^{(T+m+1)M},
    $$ 
    and $e^{(t)}\in \mathbb{R}^{(T+m+1)M}$ satisfying $e^{(t)}_i=\mathbf{1}(i=(t+m)M+1)$, then we have
    \begin{equation*}
        \begin{split}
            &\frac{1}{T}\sum_{t=0}^{T-1}\epsilon_{t,1}\sum_{j=0}^{\infty}(\Psi_j)_{1\cdot}\epsilon_{t-j-1}\\
            =&\tilde{\epsilon}^\top \left(\frac{1}{T}\sum_{t=0}^{T-1}e^{(t)}\eta^{(t)^\top}\right)\tilde{\epsilon}+\frac{1}{T}\sum_{t=0}^{T-1}\epsilon_{t,1}\sum_{j=t+m+1}^\infty (\Psi_j)_{1\cdot}\epsilon_{t-j-1}\\
            \triangleq& E_1+E_2.
        \end{split}
    \end{equation*}
    We will let $m$ be sufficiently large in later argument. The following arguments are devided into two parts: bounding $E_1$ and $E_2$.
    \begin{itemize}[leftmargin = *]
        \item[(1)] Bounding $E_1$\\
        Since all entries of $\tilde{\epsilon}$ are independent sub-Gaussian with constant parameter, we can apply the following Hanson-Wright inequality:
		\begin{Lemma}\label{HansonWright}
		Let $X=(X_1,\dots,X_n)\in \mathbb{R}^n$ be a random vector with independent components $X_i$ which satisfy $\mathbb{E}(X_i)=0$ and $\|X_i\|_{\psi_2}\leq K$. Let $A$ be an $n\times n$ matrix. Then, for every $t\geq 0$,
		\begin{equation*}
		    \mathbb{P}\left(|X^\top AX-\mathbb{E}X^\top AX|>t\right)\leq 2\exp\left\{-c\min\left(\frac{t^2}{K^4\|A\|_{F}^2},\frac{t}{K^2\|A\|_2}\right)\right\}
		\end{equation*}
		\end{Lemma}
		This lemma is a result in \cite{rudelson2013hanson}.By Lemma \ref{HansonWright}, we only need to bound the norms of $\frac{1}{T}\sum_{t=0}^{T-1}e^{(t)}\eta^{(t)^\top}$. 
        
        First note that 
        $$
        \left\|\frac{1}{T}\sum_{t=0}^{T-1}e^{(t)}\eta^{(t)\top}\right\|_2=\sup_{\|u\|_2=\|v\|_2=1}\frac{1}{T}\sum_{t=0}^{T-1}u^\top e^{(t)}\eta^{(t)\top}v.
        $$ 
        For any $u,v\in \mathbb{R}^{(T+m+1)M}$ with unit $\ell_2$ norm, one can show that
        \begin{equation*}
            \begin{split}
                &\frac{1}{T}\sum_{t=0}^{T-1}u^\top e^{(t)}\eta^{(t)\top}v\\
                =&\frac{1}{T}\sum_{t=0}^{T-1}u_{(t+m)M+1}\sum_{i=0}^{t+m}(\Psi_{t+m-i})_{1\cdot}v^{(i+1)}\\
                \leq &\frac{1}{T}\sum_{t=0}^{T-1}u_{(t+m)M+1}\sum_{i=0}^{t+m}\alpha_{t+m-i}\|v^{(i+1)}\|_2\\
                \leq &\frac{1}{T}(u_{mM+1},\cdots,u_{(T+m-1)M+1})\Gamma \begin{pmatrix}\|v^{(1)}\|_2\\
                \vdots\\
                \|v^{(T+m)}\|_2
                \end{pmatrix}\\
                \leq &\frac{\|\Gamma\|_2}{T},
            \end{split}
        \end{equation*}
        where $v^{(i)}=(v_{(i-1)M+1},\dots,v_{iM})^\top$, $\alpha_i=\left\|\Psi_i\right\|_2\geq \left\|(\Psi_i)_{1\cdot}\right\|_2$, and $\Gamma\in \mathbb{R}^{T\times (T+m)}$ is a matrix with each entry $\Gamma_{ij}=\alpha_{m+i-j}1(m+i-j\geq 0)$.  Since $\Gamma$ is a Toeplitz matrix, we will use the following lemma to bound its $\ell_2$ norm.
         \begin{Lemma}\label{Toeplitz}
         Let $f(\lambda)$ be a Fourier series defined as $f(\lambda)=\sum_{t=-\infty}^{\infty}t_k\exp\{ik\lambda\}$, with $\sum_{k=-\infty}^{\infty}|t_k|<\infty$. We define a sequence of Toeplitz matrices $T_n$ with $(T_n)_{i,j}=t_{i-j}$, then the operator norm of $T_n$ is bounded by
         \begin{equation*}
         \|T_n\|_2\leq 2\text{ess}\sup f.
         \end{equation*}
        where ess $\sup f$ the essential supremum.
        \end{Lemma}
 This is actually Lemma 4.1 in \cite{gray2006toeplitz}, and we directly apply it here. By Lemma \ref{Toeplitz}, 
        $$
        \|\Gamma\|_2\leq 2\sup_{\lambda}\left|\sum_{k=-m}^{\infty}\alpha_{m+k} e^{ik\lambda}\right|\leq 2\sum_{k=0}^\infty \alpha_k\leq \sum_{i=0}^{\infty}\left(\sum_{j=0}^\infty\alpha_{i+j}^2\right)^{\frac{1}{2}}\leq \beta.
        $$
        Thus $\left\|\frac{1}{T}\sum_{t=0}^{T-1}e^{(t)}\eta^{(t)^\top}\right\|_2\leq \frac{\beta}{T}$. While for the Frobenius norm, we have
        \begin{equation*}
        \begin{split}
            \left\|\frac{1}{T}\sum_{t=0}^{T-1}e^{(t)}\eta^{(t)^\top}\right\|_{F}^2=&\text{tr}\left(\left(\frac{1}{T}\sum_{t=0}^{T-1}\eta^{(t)}e^{(t)\top}\right)\left(\frac{1}{T}\sum_{l=0}^{T-1}e^{(t)}\eta^{(t)\top}\right)\right)\\
            =&\frac{1}{T^2}\sum_{t=0}^{T-1}\|\eta^{(t)}\|_2^2\\
           \leq &\frac{1}{T^2}\sum_{t=0}^{T-1}\sum_{i=0}^{t+m}\alpha_i^2\leq \frac{\beta^2}{T}.
        \end{split}
        \end{equation*}
        Therefore, by Lemma \ref{HansonWright}, for any $\delta>0$,
        \begin{equation*}
            \mathbb{P}\left(|E_1|>\delta\right)\leq 2\exp\left\{-cT\min\{\delta,\delta^2\}\right\}.
        \end{equation*}
        \item[(2)] Bounding $E_2$\\
        First note that 
        \begin{equation*}
            \begin{split}
                |E_2|=&\left|\frac{1}{T}\sum_{t=0}^{T-1}\epsilon_{t,1}\sum_{j=t+m+1}^\infty (\Psi_j)_{1\cdot}\epsilon_{t-j-1}\right|\\
                \leq &\frac{1}{2T}\sum_{t=0}^{T-1}\epsilon_{t,1}^2+\frac{1}{2T}\sum_{t=0}^{T-1}\left(\sum_{j=t+m+1}^\infty (\Psi_j)_{1\cdot}\epsilon_{t-j-1}\right)^2.
            \end{split}
        \end{equation*}
        Recall the definition of $\|\cdot\|_{\psi_1}$ and $\|\cdot\|_{\psi_2}$ in the proof of Lemma \ref{subGaussnorm}. Since $\|\epsilon_{t,1}^2\|_{\psi_1}\leq 2\|\epsilon_{t,1}\|_{\psi_2}^2\leq 2\tau^2$, 
        \begin{equation*}
            \mathbb{P}\left(\left|\frac{1}{2T}\sum_{t=0}^{T-1}\epsilon_{t,1}^2\right|>\delta\right)\leq 2\exp\{-cT\min\{\delta,\delta^2\}\},
        \end{equation*}
        by Bernstein type inequality of sub-exponential random variables(see proposition 5.16 in \cite{vershynin2010introduction}).
        
        Now we bound the second term $\frac{1}{2T}\sum_{t=0}^{T-1}\left(\sum_{j=t+m+1}^\infty (\Psi_j)_{1\cdot}\epsilon_{t-j-1}\right)^2$.
         Since
        \begin{equation*}
            \left|\sum_{j=t+m+1}^\infty(\Psi_j)_{1\cdot}\epsilon_{t-j-1}\right|\leq \sum_{j=t+m+1}^{\infty}\alpha_j\|\epsilon_{t-j-1}\|_2,
        \end{equation*}
        one can show that
        \begin{equation*}
            \begin{split}
                &\left\|\frac{1}{2T}\sum_{t=0}^{T-1}\left(\sum_{j=t+m+1}^\infty (\Psi_j)_{1\cdot}\epsilon_{t-j-1}\right)^2\right\|_{\psi_1}\\
                \leq &\left\|\sum_{j=t+m+1}^\infty(\Psi_j)_{1\cdot}\epsilon_{t-j-1}\right\|_{\psi_2}^2\\
                \leq &CM\tau^2\left(\sum_{j=t+m+1}^{\infty}\alpha_j\right)^2,
            \end{split}
        \end{equation*}
        where we apply the fact that $\left\|\left\|\epsilon_t\right\|_2\right\|_{\psi_2}\leq C\sqrt{M}\tau$, which is shown in the proof of Lemma \ref{subGaussnorm}. Thus we have 
        \begin{equation*}
            \begin{split}
                &\mathbb{P}\left(\frac{1}{2T}\sum_{t=0}^{T-1}\left(\sum_{j=t+m+1}^\infty (\Psi_j)_{1\cdot}\epsilon_{t-j-1}\right)^2>\delta\right)\\
                \leq &C\exp\left\{-\frac{c\delta}{M\tau^2\left(\sum_{j=t+m+1}^{\infty}\alpha_j\right)^2}\right\}.
            \end{split}
        \end{equation*}
        due to the tail bound of sub-exponential r.v. (also see \cite{vershynin2010introduction}). Since 
        $$
        \sum_{i=0}^{\infty}\alpha_i\leq \sum_{i=0}^{\infty}\left(\sum_{j=0}^{\infty}\alpha_{i+j}^2\right)^{\frac{1}{2}}\leq \beta,
        $$ 
        $$
        \lim_{m\rightarrow \infty} \left(\sum_{j=t+m+1}^{\infty}\alpha_j\right)^2=0.
        $$
        Let $m$ be sufficiently large such that $\left(\sum_{j=t+m+1}^{\infty}\alpha_j\right)^2\leq \frac{1}{MT}$, then we arrive at the following
        \begin{equation*}
            \mathbb{P}\left(\frac{1}{T}\sum_{t=0}^{T-1}\epsilon_{t,1}(\mathcal{X}_t)_1\right)\leq C\exp\{-cT\min\{\delta,\delta^2\}\}.
        \end{equation*}
        Let $\delta=C\sqrt{\log M}{T}$ and take a union bound over the $pM^2$ entries of $\frac{1}{T}\sum_{t=0}^{T-1}\epsilon_t\mathcal{X}_t^\top$, the conclusion follows.
    \end{itemize}

\end{proof}

\begin{proof}[Proof of Lemma \ref{true_w_dev}]
Without loss of generality, consider 
$$\frac{1}{T}\sum_{t=0}^{T-1}\left(\mathcal{X}_{t,D_m}-w_m^{*\top}\mathcal{X}_{t,D_m^c}\right)_i\mathcal{X}_{t,j}
$$
for any $1\leq i\leq d_m$, and $j\in D_m^c$. Similar from the proof of Lemma \ref{true_w_dev}, We can write it as a quadratic form 
\begin{equation*}
    \frac{1}{T}\sum_{t=0}^{T-1}\mathcal{X}_t^\top \frac{1}{2}\left((W_m^*)_{i\cdot}^\top e_j^\top+e_j(W_m^*)_{i\cdot}\right)\mathcal{X}_t,
\end{equation*}
where $W_m^*$ is defined as in \eqref{W_def}. Since $\frac{1}{2}\left((W_m^*)_{i\cdot}^\top e_j^\top+e_j(W_m^*)_{i\cdot}\right)$ is of rank 2, and we have bounded $\left\|(W_m^*)_{i\cdot}\right\|_2$ in \eqref{Wl2Bnd}, applying Lemma \ref{trl2norm} leads to
\begin{equation*}
    \begin{split}
        \left\|\frac{1}{2}\left((W_m^*)_{i\cdot}^\top e_j^\top+e_j(W_m^*)_{i\cdot}\right)\right\|_{\tr}\leq & 2\left\|\frac{1}{2}\left((W_m^*)_{i\cdot}^\top e_j^\top+e_j(W_m^*)_{i\cdot}\right)\right\|_2\\
        \leq&\left\|(W_m^*)_{i\cdot}\right\|_2\leq C.
    \end{split}
\end{equation*}
Applying Lemma \ref{QuaBound}, and taking a union bound over all entries of 
$$
\frac{1}{T}\sum_{t=0}^{T-1}\left(\mathcal{X}_{t,D_m}-w_m^{*\top}\mathcal{X}_{t,D_m^c}\right)\mathcal{X}_{t},
$$
the conclusion follows.
\end{proof}

\begin{proof}[Proof of Lemma \ref{CovBound}]
Similar from the proof of Lemma \ref{DevBnd}, we consider $\left|\frac{1}{T}\sum_{t=0}^{T-1}X_{ti}X_{tj}-\Upsilon_{ij}\right|$. Since 
$$
\frac{1}{T}\sum_{t=0}^{T-1}X_{ti}X_{tj}=\frac{1}{T}\sum_{t=0}^{T-1}X_t^\top\left(\frac{1}{2}(e_ie_j^\top+e_je_i^\top)\right) X_t,
$$
by Lemma \ref{QuaBound}, we need to bound norms of $\frac{1}{2}(e_ie_j^\top+e_je_i^\top)$, which is of rank at most 2. One can show that
\begin{equation*}
\begin{split}
    \left\|\frac{1}{2}(e_ie_j^\top+e_je_i^\top)\right\|_{\tr}\leq 2 \left\|\frac{1}{2}(e_ie_j^\top+e_je_i^\top)\right\|_2\leq 2\|e_i\|_2\|e_j\|_2,
    \end{split}
\end{equation*}
with Lemma \ref{trl2norm}.
Therefore, by taking a union bound, it is clear that 
$$
\left\|\frac{1}{T}\sum_{t=0}^{T-1}X_tX_t^\top-\Upsilon\right\|_{\infty}\leq C\sqrt{\frac{\log M}{T}},
$$
with probability at least $1-c_1\exp\{-c_2\log M\}$.
\end{proof}

\section{Proof of Lemmas in Section \ref{lemma_proof} and Appendix \ref{lemma_proof2}}
\begin{proof}[Proof of Lemma \ref{convergence.rate}]
	Here we adopt the proof framework for Lemma 4 in \cite{grama2006asymptotic}, but with some small adjustments. First we construct a new martingale difference sequence $(m_{nk},\mathcal{G}_{nk})_{1\leq k\leq n+1}$, sum of whose covariances equal to $I_{d\times d}$. Random projections are used for construction. The following lemma on random projections is stated as Lemma 3 in \cite{grama2006asymptotic}.
	\begin{Lemma}
	Let $V$ and $a_1, \cdots, a_n$ be positive semi-definite $d\times d$ matrices. Set $A_k = a_1 +\cdots+a_k$, for $k =1, \cdots,n$. Then there exist a sequence of integers $1\leq \tau_1\leq \cdots\leq \tau_d\leq n$ and a corresponding sequence $\mathcal{S}_1\supseteq\cdots \supseteq \mathcal{S}_d$ of subspaces of $\mathbb{R}^d$ such that, with $P_k$ defined as the projection matrix of subspace $\mathcal{S}_i$, for $\tau_i\leq k <\tau_{i+1}$ (where $\tau_0=1, \tau_{d+1}=n+1,\mathcal{S}_0=\mathbb{R}^d$), the following statements hold true for $k=1,\cdots,n$:\\
	$(a)V-\widehat{A}_k$ is non-negative definite, where $\widehat{A}_k =P_1a_1P_1+\cdots+P_ka_kP_k$;\\
	$(b)x^\top (\widehat{A}_k-A_k)x=0$, for all $x\in \Pi_k\triangleq \{P_kx: x\in \mathbb{R}^d\}$;\\
	$(c)x^\top(\widehat{A}_k-V+\alpha_k I)x\geq 0$ for all $x\in \Pi_k^\top$, where $\alpha_k=\max\{\|a_{\tau_j}\|_2:\tau_j\leq k\}$.\\
	Meanwhile, $P_k$ is determined by $a_1,\cdots,a_k$ and $V$. 
	\end{Lemma}
	Given this claim, $m_{nk}$ can be constructed as follows:\\
	Recall the martingale sequence we consider is $(\xi_{nk},\mathcal{F}_{nk})_{1\leq k\leq n+1}$, and $a_{nk}=\mathbb{E}\left(\xi_{nk}\xi_{nk}^{\top}\right)$. Apply the fact with $V=I$, $a_k=a_{nk}$, and let $\{P_{nk}\}_{k=1}^n$ be the corresponding projection matrices. Let $D_n=I-\sum_{k=1}^n P_{nk}a_{nk}P_{nk}$, which is non-negative definite. Define
	$$
	M_{k}^n=\sum_{k=1}^n m_{nk},\quad 1\leq k\leq n+1,
	$$
	where
	$$
	m_{nk}=P_{nk}\xi_{nk}, \text{ for }1\leq k\leq n, \quad m_{n,n+1}=D_n^\frac{1}{2}\eta_{n,n+1}.
	$$
	Since $P_{nk}\in \mathcal{F}_{n,k-1}$, $m_{nk}\in \mathcal{F}_{nk}$ for $1\leq k\leq n$.Thus $(m_{nk},\mathcal{G}_{nk})$ is also a martingale difference sequence with $\mathcal{G}_{nk}=\mathcal{F}_{nk}$, when $1\leq k\leq n$, and $\mathcal{G}_{n,n+1}=\sigma(F_{nn},\eta_{n,n+1})$. Meanwhile, 
	$$
	\left\langle M^n\right\rangle _{n+1}=\sum_{k=1}^{n+1}\mathbb{E}(m_{nk}m_{nk}^\top|\mathcal{F}_{n,k-1})=I_{d\times d}.
	$$
	This construction is from \cite{grama2006asymptotic}. They also prove that, for any $\varepsilon, \delta>0$, 
	\begin{equation}\label{ErrorXM}
	   \mathbb{P}\left(\|X_n^n-M_{n+1}^n\|_2\geq \varepsilon\right)\leq C(d,\delta)\varepsilon^{-2-2\delta}\left(L_{\delta}^{n,d}+N_{\delta}^{n,d}\right), 
	\end{equation}
	Since
	\begin{equation}\label{prob_diff}
	\begin{split}
	&-\mathbb{P}(\|X_n^n-M_{n+1}^n\|_2>\varepsilon)-\mathbb{P}(\|Z+\mu\|_2\geq r+2\varepsilon)\\
	&+\mathbb{P}(\|M_{n+1}^n+\mu\|_2\geq r+\varepsilon)-\mathbb{P}(Z\in [r,r+2\varepsilon))\\
	\leq &\mathbb{P}(\|X_n^n+\mu\|_2\geq r)-\mathbb{P}(\|Z+\mu\|_2\geq r)\\
	\leq &\mathbb{P}(\|M_{n+1}^n+\mu\|_2\geq r-\varepsilon)-\mathbb{P}(\|Z+\mu\|_2\geq r-2\varepsilon)\\
	&+\mathbb{P}(\|X_n^n-M_{n+1}^n\|_2>\varepsilon)+\mathbb{P}(Z\in [r-2\varepsilon, r)),
	\end{split}
	\end{equation}
	for any $\mu\in \mathbb{R}^d, r\geq 0, \varepsilon>0$,
	we need to bound $$\mathbb{E}(1(\|Z+\mu\|_2\geq r+2\varepsilon))-\mathbb{E}(1(\|M_{n+1}^n+\mu\|_2\geq r+\varepsilon))$$ and $$\mathbb{E}(1(\|M_{n+1}^n+\mu\|_2\geq r-\varepsilon))-\mathbb{E}(1(\|Z+\mu\|_2\geq r-2\varepsilon)).$$ The following functions are defined as a smooth relaxation for indicator function. Let 
	\begin{equation}\label{f*_def}
	   f_*(z)=\int_{-\infty}^{z-\frac{1}{2}} \phi(t) \mathrm{d}t, \text{ with }\phi(t)=\frac{1}{C}\exp\{-\frac{4}{1-4t^2}\}1_{\left(-\frac{1}{2},\frac{1}{2}\right)}(t), 
	\end{equation}
	where $C$ is a normalizing constant s.t. $\int \phi(t) dt=1$. Then we have $f_*(z)=0$ if $z\leq 0$, $0\leq f_*(z)\leq 1$ if $0\leq z\leq 1$, and $f_*(z)=1$ if $z\geq 1$. $f_*(z)$ is infinitely many times differentiable on $\mathbb{R}$, and since $f_*(z)$ is constant when $z\leq 0$ or $z\geq 1$, for any fixed order, the derivative of $f_*(z)$ is bounded. For any $z\in \mathbb{R}^d$, let 
	\begin{equation}\label{fl_def}
	    f_{l,\mu,r,\varepsilon}(z)=f_*(g_{l,\mu,r,\varepsilon}(z)),
	\end{equation}
	where
	\begin{equation}\label{gl_def}
	    g_{1,\mu,r,\varepsilon}(z)=\frac{\|z+\mu\|_2-r-\varepsilon}{\varepsilon}, \quad g_{2,\mu,r,\varepsilon}(z)=\frac{\|z+\mu\|_2-r+2\varepsilon}{\varepsilon}.
	\end{equation}
	In the following proof, we will denote $f_{l,\mu,r,\varepsilon}(z)$ and $g_{l,\mu,r,\varepsilon}(z)$ as $f_l(z)$ and $g_l(z)$, $l=1,2$ for brevity. Therefore,
	$$\mathbb{E}(1(\|Z+\mu\|_2\geq r+2\varepsilon))-\mathbb{E}(1(\|M_{n+1}^n+\mu\|_2\geq r+\varepsilon))\leq \mathbb{E}(f_1(Z)-f_1(M_{n+1}^n)),$$ $$\mathbb{E}(1(\|M_{n+1}^n+\mu\|_2\geq r-\varepsilon))-\mathbb{E}(1(\|Z+\mu\|_2\geq r-2\varepsilon))\leq \mathbb{E}(f_2(M_{n+1}^n)-f_1(Z)).$$
	Thus,
	\begin{equation*}
    \begin{split}
	&|\mathbb{P}(\|X_n^n+\mu\|_2\geq r)-\mathbb{P}(\|Z+\mu\|_2\geq r)|\\
	\leq &\max_{l=1,2}|\mathbb{E}(f_l(M_{n+1}^n)-f_l(Z))|+\mathbb{P}(\|X_n^n-M_{n+1}^n\|_2>\varepsilon)\\
	&+\mathbb{P}(\|Z+\mu\|_2\in [r-2\varepsilon, r+2\varepsilon]).
    \end{split}
	\end{equation*}
	Actually, when $r\leq 3\varepsilon$, the right hand side of (\ref{prob_diff}) can be substituted by 
	$$\mathbb{P}(\|Z+\mu\|_2< 3\varepsilon),$$
	and 
	\begin{equation}\label{prob_diff_gen}
	    \begin{split}
	        &|\mathbb{P}(\|X_n^n+\mu\|_2\geq r)-\mathbb{P}(\|Z+\mu\|_2\geq r)|\\
	\leq &\max\{|\mathbb{E}(f_1(M_{n+1}^n)-f_1(Z))|+\mathbb{P}(\|X_n^n-M_{n+1}^n\|_2>\varepsilon)\\
	&+\mathbb{P}(\|Z+\mu\|_2\in [r, r+2\varepsilon]), \mathbb{P}(\|Z+\mu\|_2\in [0,3\varepsilon))\}.
	    \end{split}
	\end{equation}
     To bound $\mathbb{E}(f_l(M_{n+1}^n)-f_l(Z))$, we will use the following lemma.
     \begin{Lemma}\label{DerivativeBnd}
     For $f_l(\cdot)$ defined as in \eqref{fl_def}, 
     \begin{align}\label{bounded_deriv}
	    \left|\sum_{1\leq i_1,\cdots,i_k\leq d} y_{i_1}\cdots y_{i_k}\frac{\partial^k}{\partial z_{i_1}\cdots\partial z_{i_k}}f_l(z)\right|\leq C(k)\varepsilon^{-k}\|y\|_2^k,
	\end{align}
	for any $k\in \mathbb{Z}^*$, $y,z \in \mathbb{R}^{d}$, when $l=1$, or when $l=2$ and $r>3\varepsilon$. 
     \end{Lemma}
	 The proof of this lemma is deferred to Appendix \ref{RestProof}. In the following proof, we will always assume the condition $l=1$ or $l=2$ and $r>3\varepsilon$ hold. Therefore, for any $m\in \mathbb{Z}^*$,
	\begin{equation*}
    \begin{split}
	&\left|f_l(z+y)-f_l(y)-\sum_{k=1}^{m}\sum_{1\leq i_1,\cdots,i_k\leq d} y_{i_1}\cdots y_{i_k}\frac{\partial^k}{\partial z_{i_1}\cdots\partial z_{i_k}}f_l(z)\right|\\
	= &\left|\sum_{1\leq i_1,\cdots,i_{m+1}\leq d} y_{i_1}\cdots y_{i_{m+1}}\frac{\partial^{m+1}}{\partial u_{i_1}\cdots\partial u_{i_{m+1}}}f_l(u)\right|\\
	\leq &C(m+1)\varepsilon^{-m-1}\|y\|_2^{m+1},
    \end{split}
	\end{equation*}
	where $u=z+t_1y$ for some $0\leq t_1\leq 1$.
	Meanwhile,
	\begin{align*}
	&\left|f_l(z+y)-f_l(y)-\sum_{k=1}^{m}\sum_{1\leq i_1,\cdots,i_k\leq d} y_{i_1}\cdots y_{i_k}\frac{\partial^k}{\partial z_{i_1}\cdots\partial z_{i_k}}f_l(z)\right|\\
	= &\Bigg|\sum_{1\leq i_1,\cdots,i_{m}\leq d} y_{i_1}\cdots y_{i_{(m)}}\frac{\partial^{m}}{\partial v_{i_1}\cdots\partial v_{i_{m}}}f_l(v)\\
	&-\sum_{1\leq i_1,\cdots,i_m\leq d} y_{i_1}\cdots y_{i_m}\frac{\partial^m}{\partial z_{i_1}\cdots\partial z_{i_m}}f_l(z)\Bigg|\\
	\leq &2C(m)\varepsilon^{-m}\|y\|_2^{m},
	\end{align*}
	where $v=z+t_2y$ for some $0\leq t_2\leq 1$. Thus, for any $\delta>0$,
	\begin{align*}
	&\left|f_l(z+y)-f_l(y)-\sum_{k=1}^{\lceil 2+2\delta\rceil-1}\sum_{1\leq i_1,\cdots,i_k\leq d} y_{i_1}\cdots y_{i_k}\frac{\partial^k}{\partial z_{i_1}\cdots\partial z_{i_k}}f_l(z)|\right|\\
	\leq &C(\delta)\max\{\varepsilon^{-{\lceil 2+2\delta\rceil}+1}\|y\|_2^{\lceil 2+2\delta\rceil-1}, \varepsilon^{-{\lceil 2+2\delta\rceil}}\|y\|_2^{\lceil 2+2\delta\rceil}\}\\
	\leq &C(\delta)\varepsilon^{-2-2\delta}\|y\|_2^{2+2\delta}.
	\end{align*}
	Let $\tilde{w}_{nk}$, $1\leq k\leq n$ be i.i.d. standard Gaussian random vectors that are independent of $\mathcal{G}_{n,n+1}$, $w_{nk}=(b_{nk})^{\frac{1}{2}}\tilde{w}_{nk}$, for $k=1,\cdots,n+1$, where $b_{nk}=\mathbb{E}(m_{nk}m_{nk}^\top|\mathcal{G}_{n,k-1})$. Define
	$$
	W_{n+2}^n=0,\quad W_k^n=\sum_{i=k}^{n+1}w_{ni}, \quad 1\leq k\leq n+1.
	$$
	Then $W_1^n$ follows standard Gaussian distribution. Let $U_k^n=M_{k-1}^n+W_{k+1}^n$, then
	\begin{equation*}
    \begin{split}
	&\left|\mathbb{E}(f_l(M_{n+1}^n)-f_l(Z))\right|\\
	=&\left|\mathbb{E}(f_l(M_{n+1}^n)-f_l(W_1^n))\right|\\
	=&\left|\sum_{k=1}^{n+1} \mathbb{E}(f_l(U_{k}^n+m_{nk})-f_l(U_k^n+w_{nk}))\right|\\
	\leq &\sum_{k=1}^{n+1}\Bigg|\mathbb{E}(f_l(U_k^n+m_{nk})-f_l(U_k^n)\\
	&-\sum_{j=1}^{\lceil 2+2\delta\rceil-1}\sum_{1\leq i_1,\cdots,i_j\leq d} (m_{nk})_{i_1}\cdots (m_{nk})_{i_j}\frac{\partial^j}{\partial z_{i_1}\cdots\partial z_{i_j}}f_l(U_k^n))\Bigg|\\
	+&\sum_{k=1}^{n+1}\Bigg|\mathbb{E}(f_l(U_{k}^n+w_{nk})-f_l(U_k^n)\\
	&-\sum_{j=1}^{\lceil 2+2\delta\rceil-1}\sum_{1\leq i_1,\cdots,i_j\leq d} (w_{nk})_{i_1}\cdots (w_{nk})_{i_j}\frac{\partial^j}{\partial z_{i_1}\cdots\partial z_{i_j)}}f_l(U_k^n))\Bigg|\\
	\leq &\sum_{k=1}^{n+1}C(\delta)\varepsilon^{-2-2\delta}\mathbb{E}(\|m_{nk}\|_2^{2+2\delta}).
    \end{split}
	\end{equation*}
	Generally this inequality holds for $\delta\in (0,\frac{1}{2}]$, since $w_{nk}$ and $m_{nk}$ have the same second order moments, which justifies the fourth line. 
	By the proof of Lemma 4 in \cite{grama2006asymptotic}, 
	$$
	\sum_{k=1}^{n+1}\mathbb{E}(\|m_{nk}\|_2^{2+2\delta})\leq C(d,\delta)(L_\delta^{n,d}+N_\delta^{n,d}),
	$$
	thus 
	\begin{equation}\label{ErrorMZ}
	    \left|\mathbb{E}\left(f_l(M^n_{n+1})-f_l(z)\right)\right|\leq C(d,\delta)\varepsilon^{-2-2\delta}R_{\delta}^{n,d}.
	\end{equation}
	Now we only need to bound $\mathbb{P}\left(\|Z+\mu\|_2\in [r-2\varepsilon,r+2\varepsilon]\right)$ and $\mathbb{P}\left(\|Z+\mu\|_2\in [0,3\varepsilon)\right)$. Assume $\varepsilon\leq 1$, then
	$$
	\mathbb{P}\left(\|Z+\mu\|_2\in [0,3\varepsilon)\right)= \mathbb{P}\left(Z\in \mathbb{B}_{3\varepsilon}(-\mu)\right)\leq C(d)\varepsilon^d\leq C(d)\varepsilon.
	$$
	Meanwhile,
	\begin{equation*}
    \begin{split}
	    &\mathbb{P}(\|Z+\mu\|_2\in [r-2\varepsilon,r+2\varepsilon])\\
	    = &\mathbb{P}(Z\in \mathbb{B}_{r+2\varepsilon}(-\mu)\backslash \mathbb{B}_{r-2\varepsilon}(-\mu))\\
	    \leq &\begin{cases} C(d)\left((r+2\varepsilon)^d-(r-2\varepsilon)^d\right), & r\leq 2\varepsilon+\|\mu\|\\
	    C(d)\exp\{-(r-2\varepsilon-\|\mu\|_2)^2/2\}\left((r+2\varepsilon)^d-(r-2\varepsilon)^d\right), & r>2\varepsilon+\|\mu\|_2
	    \end{cases}\\
	    \leq &C(d,\|\mu\|_2)\varepsilon.
        \end{split}
	\end{equation*}
	The last line is due to that 
	\begin{equation*}
	    \begin{split}
	        (r+2\varepsilon)^d-(r-2\varepsilon)^d\leq &4\varepsilon d(r+2\varepsilon)^{d-1}\leq 4d\varepsilon(4+\|\mu\|_2)^{d-1},
	    \end{split}
	\end{equation*}
	when $r\leq 2\varepsilon+\|\mu\|_2$, and 
	\begin{equation*}
	    \begin{split}
	        &\exp\{-(r-2\varepsilon-\|\mu\|_2)^2/2\}\left((r+2\varepsilon)^d-(r-2\varepsilon)^d\right)\\
	        \leq &4\varepsilon d\sup_{x>0}(x+4\varepsilon+\|\mu\|_2)^{d-1}\exp\{-x^2/2\}\\
	        \leq &4d\sup_{x>0}(x+4+\|\mu\|_2)^{d-1}\exp\{-x^2/2\} \varepsilon.
	    \end{split}
	\end{equation*}
	Here clearly $C(d,\|\mu\|_2)$ is non-decreasing with respect to $\|\mu\|_2$.
	Therefore, by (\ref{prob_diff_gen}), (\ref{ErrorXM}) and (\ref{ErrorMZ}), when $R_{\delta}^{n,d}\leq 1$, for any $\mu\in \mathbb{R}^d$, $r\geq 0$, $0<\delta\leq \frac{1}{2}$, with $\varepsilon=(R_{\delta}^{n,d})^{\frac{1}{3+2\delta}}$, 
	$$
	\mathbb{P}(\|X_n^n+\mu\|_2\geq r)-\mathbb{P}(\|Z+\mu\|_2\geq r)\leq C(d,\delta,\|\mu\|_2)\left(R_\delta^{n,d}\right)^{\frac{1}{3+2\delta}},
	$$
	where $C(d,\delta,\|\mu\|_2)$ is non-decreasing with respect to $\|\mu\|_2$. \\
\end{proof}

\begin{proof}[Proof of Lemma \ref{subGaussnorm}]
First we introduce the following two norms:\\
For any random variable $X$,
\begin{equation*}
\begin{split}
    \|X\|_{\psi_1}=&\sup_{p\geq 1}p^{-1}\mathbb{E}\left(|X|^p\right)^{\frac{1}{p}},\\
\|X\|_{\psi_2}=&\sup_{p\geq 1}p^{-\frac{1}{2}}\mathbb{E}\left(|X|^p\right)^{\frac{1}{p}}. 
\end{split}
\end{equation*}
These two norms are related to sub-exponential and sub-Gaussian random variables, and the following lemma shows the connections between the two norms and the scale factor for sub-Gaussian r.v.
\begin{Lemma}\label{psinorm}
For any sub-Gaussian r.v. $X$ with scale factor $\tau$, the following hold:
$$
c\tau\leq \|X\|_{\psi_2}\leq C\tau,
$$
with some absolute constants $c, C$, and 
$$
\|X\|_{\psi_2}^2\leq \|X\|_{\psi_1}\leq 2\|X\|_{\psi_2}^2.
$$
\end{Lemma}
This is an established result in \cite{vershynin2010introduction}. By Lemma \ref{psinorm}, bounding $\left\|\left\|W_m^*\mathcal{X}_t\right\|_2^2\right\|_{\psi_1}$ would be sufficient, and we start from bounding $\mathbb{E}\left(\exp\left\{\lambda \left(W_m^*\right)_{i\cdot}\mathcal{X}_t\right\}\right)$ for any $\lambda\in \mathbb{R}$. Recall that $\mathcal{X}_t=\Psi_j^{(p)}\varepsilon_{t-j-1}$, with $\Psi_j^{(p)}$ defined as in \eqref{PsipDef}, we can write
\begin{equation*}
\begin{split}
    &(W_m^*)_{i\cdot}\mathcal{X}_t=(W_m^*)_{i\cdot}\sum_{k=0}^{\infty}\Psi_k^{(p)} \epsilon_{t-k-1}=\lim_  {N\rightarrow \infty}\sum_{k=0}^N(W_m^*)_{i\cdot}\Psi_k^{(p)} \epsilon_{t-k-1},\\
    &\exp\left\{\lambda \left(W_m^*\right)_{i\cdot}\mathcal{X}_t\right\}=\lim_{N\rightarrow \infty}\exp\left\{\lambda \sum_{k=0}^N(W_m^*)_{i\cdot}\Psi_k^{(p)} \epsilon_{t-k-1}\right\},
\end{split}
\end{equation*}
and 
\begin{equation*}
    \exp\left\{\lambda \sum_{k=0}^N\left(W_m^*\right)_{i\cdot}\Psi_k^{(p)} \epsilon_{t-k-1}\right\}\leq \exp\left\{|\lambda| \sum_{k=0}^{\infty}\left\|(W_m^*)_{i\cdot}\right\|_2\tilde{\alpha}_k \left\|\epsilon_{t-k-1}\right\|_2\right\},
\end{equation*}
where $\tilde{\alpha}_k$ is defined as $\left\|\Psi_k^{(p)}\right\|_2$. The relationship between $\tilde{\alpha}_k$ and $\alpha_k=\left\|\Psi_k\right\|_2$ can be established as follows:
\begin{equation}\label{PsipNorm}
            \begin{split}
                \tilde{\alpha}_k=\sup_{\|u\|_2=1}\left\|\Psi_k^{(p)}u\right\|_2=&\sup_{\|u\|_2=1}\left(\sum_{n=0}^{(p-1)\wedge j}\left\|\Psi_{k-n}u\right\|_2^2\right)^{\frac{1}{2}}\leq \left(\sum_{n=0}^{p-1}\alpha_{k-n}^2\right)^{\frac{1}{2}},
            \end{split}
        \end{equation}
        if we define $\alpha_i=0$ when $i<0$.
We now prove that $\exp\left\{|\lambda| \sum_{k=0}^{\infty}\left\|(W_m^*)_{i\cdot}\right\|_2\tilde{\alpha}_k \left\|\epsilon_{t-k}\right\|_2\right\}$ is integrable so that we can use Dominated Convergence Theorem. Since $\epsilon_{ti}$'s are all independent sub-Gaussian random variables with parameter $\tau$,
\begin{equation}\label{enormNorm}
    \begin{split}
        \left\|\left\|\epsilon_t\right\|_2\right\|_{\psi_2}\leq  \left\|\|\epsilon_t\|_2^2\right\|_{\psi_1}^{\frac{1}{2}}\leq \left(M\|\epsilon_{ti}^2\|_{\psi_1}\right)^{\frac{1}{2}}\leq C\sqrt{M}\tau,
    \end{split}
\end{equation}
where the second inequality is due to Minkowski's inequality. Thus,
\begin{equation*}
    \begin{split}
        &\mathbb{E}\left(\exp\left\{|\lambda| \sum_{k=0}^{\infty}\left\|(W_m^*)_{i\cdot}\right\|_2\tilde{\alpha}_k \left\|\epsilon_{t-k}\right\|_2\right\}\right)\\
        =&\lim_{N\rightarrow \infty}\mathbb{E}\left(\exp\left\{|\lambda| \sum_{k=0}^{N}\left\|(W_m^*)_{i\cdot}\right\|_2\tilde{\alpha}_k \left\|\epsilon_{t-k}\right\|_2\right\}\right)\\
        \leq&\lim_{N\rightarrow \infty}\exp\left\{CM\lambda^2\left\|(W_m^*)_{i\cdot}\right\|_2^2\sum_{k=0}^N\tilde{\alpha}_k^2\right\}\leq\exp\left\{CM\lambda^2\right\},
    \end{split}
\end{equation*}
where the first equality is due to Monotone Convergence Theorem, and the last line is due to (\ref{Wl2Bnd}) and the fact that
\begin{equation*}
    \begin{split}
      \sum_{k=0}^N\tilde{\alpha}_k^2  \leq \sum_{k=0}^N\sum_{n=0}^{p-1}\alpha_{k-n}^2\leq p\sum_{k=0}^N\alpha_k^2\leq \beta^2.
    \end{split}
\end{equation*} 
Therefore, by Dominated Convergence Theorem, 
\begin{equation*}
    \begin{split}
        &\mathbb{E}\left(\exp\left\{\lambda \left(W_m^*\right)_{i\cdot}\mathcal{X}_t\right\}\right)\\
        =&\lim_{N\rightarrow \infty}\mathbb{E}\left(\exp\left\{\lambda \sum_{k=0}^N(W_m^*)_{i\cdot}\Psi_k^{(p)} \epsilon_{t-k}\right\}\right)\\
        \leq &\exp\left\{C\lambda^2 \|(W_m^*)_{i\cdot}\|_2^2\sum_{k=0}^{\infty}\tilde{\alpha}_k^2\right\}\\
        =&\exp\left\{C\lambda^2\right\}.
    \end{split}
\end{equation*}
By Lemma \ref{psinorm}, $\left\|\left(W_m^*\right)_{i\cdot}\mathcal{X}_t\right\|_{\psi_2}\leq C$, and 
\begin{equation*}
    \begin{split}
        \left\|\left\|W_m^*\mathcal{X}_t\right\|_2^2\right\|_{\psi_1}\leq \sum_{i=1}^{d_m} \left\|\left(\left(W_m^*\right)_{i\cdot}\mathcal{X}_t\right)^2\right\|_{\psi_1}\leq 2\sum_{i=1}^{d_m}\left\|\left(W_m^*\right)_{i\cdot}\mathcal{X}_t\right\|_{\psi_2}^2\leq C.
    \end{split}
\end{equation*}
Thus 
\begin{equation*}
    \mathbb{E}\left(\left\|W_m^*\mathcal{X}_t\right\|_2^p\right)^{\frac{1}{p}}\leq C\sqrt{p}.
\end{equation*}
\end{proof}
\begin{proof}[Proof of Lemma \ref{QuaBound}]
Recall that $\mathcal{X}_t=\sum_{j=0}^{\infty}\Psi_j^{(p)}\epsilon_{t-j-1}$, where $\Psi_j^{(p)}$ is defined in \eqref{PsipDef}. Similar from the proof of Lemma \ref{DevBnd}, for any positive integer $m$, we can write down $\frac{1}{T}\sum_{t=0}^{T-1}\mathcal{X}_t^\top B\mathcal{X}_t$ as the following:
\begin{equation*}
    \begin{split}
        \frac{1}{T}\sum_{t=0}^{T-1}\mathcal{X}_t^\top B\mathcal{X}_t=& \frac{1}{T}\sum_{t=0}^{T-1}\left(\sum_{j=0}^{\infty}\Psi_j^{(p)}\epsilon_{t-j-1}\right)^\top B\left(\sum_{j=0}^{\infty}\Psi_j^{(p)}\epsilon_{t-j-1}\right)\\
        =&\frac{1}{T}\sum_{t=0}^{T-1}\left(\sum_{j=0}^{t+m-1}\Psi_j^{(p)}\epsilon_{t-j-1}\right)^\top B\left(\sum_{j=0}^{t+m-1}\Psi_j^{(p)}\epsilon_{t-j-1}\right)\\
        &+\frac{1}{T}\sum_{t=0}^{T-1}\left(\sum_{j=t+m}^{\infty}\Psi_j^{(p)}\epsilon_{t-j-1}\right)^\top B\left(\sum_{j=t+m}^{\infty}\Psi_j^{(p)}\epsilon_{t-j-1}\right)\\
        &+\frac{2}{T}\sum_{t=0}^{T-1}\left(\sum_{j=0}^{t+m-1}\Psi_j^{(p)}\epsilon_{t-j-1}\right)^\top B\left(\sum_{j=t+m}^{\infty}\Psi_j^{(p)}\epsilon_{t-j-1}\right)\\
        \triangleq &E_1+E_2+E_3.
    \end{split}
\end{equation*}
Then we can bound each $E_i$ from its expectation separately, and $m$ will be chosen to be sufficiently large later.
\begin{itemize}[leftmargin = *]
    \item[(1)] Bounding $E_1-\mathbb{E}(E_1)$\\
    Let $\Theta^{(t)}\in \mathbb{R}^{pM\times (T+m)M}$ and $\tilde{\epsilon}\in \mathbb{R}^{(T+m)M}$ be defined as
    $$
    \Theta^{(t)}=\begin{pmatrix}\Psi_{t+m-1}^{(p)}&\cdots&\Psi_0^{(p)}&0&\cdots&0\end{pmatrix},
    $$
    $$
    \tilde{\epsilon}=\begin{pmatrix}
    \epsilon_{-m}^\top &\cdots &\epsilon_{T-1}^\top
    \end{pmatrix}^\top.
    $$
    Then $E_1=\tilde{\epsilon}^\top \left(\frac{1}{T}\sum_{t=0}^{T-1}\Theta^{(t)\top}B\Theta^{(t)}\right)\tilde{\epsilon}$, and by Lemma \ref{HansonWright} we only need to bound the operator norm and Frobenius norm of $\frac{1}{T}\sum_{t=0}^{T-1}\Theta^{(t)\top}B\Theta^{(t)}$.
    \begin{itemize}[leftmargin = *]
        \item[i.] Bounding $\left\|\frac{1}{T}\sum_{t=0}^{T-1}\Theta^{(t)\top}B\Theta^{(t)}\right\|_2$\\
        For any unit vector $u,v\in \mathbb{R}^{(t+m)M}$, 
        \begin{equation*}
        \begin{split}
            u^\top \frac{1}{T}\sum_{t=0}^{T-1}\Theta^{(t)\top}B\Theta^{(t)}v=&\frac{1}{T}\sum_{t=0}^{T-1}\sum_{i,j=1}^{t+m} u^{(i)\top}\Psi_{t+m-1}^{(p)\top} B\Psi_{t+m-j}^{(p)}v^{(j)}\\
            =&\frac{1}{T}\sum_{i,j=1}^{T+m-1}u^{(i)\top}\left[\sum_{t=(i\vee j-m)\vee 0}^{T-1}\Psi_{t+m-1}^{(p)\top}B\Psi_{t+m-j}^{(p)}\right]v^{(j)}\\
            \leq &\frac{1}{T}\sum_{i,j=1}^{T+m-1}\|u^{(i)}\|_2\|v^{(j)}\|_2\|B\|_2\sum_{l=0}^{\infty}\left\|\Psi_{|i-j|+l}^{(p)}\right\|_2\left\|\Psi_{l}^{(p)}\right\|_2,
        \end{split}
        \end{equation*}
        where $u^{(i)}=(u_{(i-1)M+1},\dots,u_{iM})$. Let $\tilde{\alpha}_i=\left\|\Psi_i^{(p)}\right\|_2$, and $\Gamma\in \mathbb{R}^{(t+m)\times (t+m)}$ be defined as $\Gamma_{ij}=\sum_{k=0}^\infty \tilde{\alpha}_{|i-j|+k}\tilde{\alpha}_{k}$, then 
        \begin{equation*}
        \begin{split}
            u^\top \frac{1}{T}\sum_{t=0}^{T-1}\Theta^{(t)\top}B\Theta^{(t)}v\leq 
            &\frac{\|B\|_2}{T}(\|u^{(1)}\|_2,\dots,\|u^{(t+m)}\|_2)\Gamma \begin{pmatrix}
            \|v^{(1)}\|_2\\
            \vdots\\
            \|v^{(t+m)}\|_2
            \end{pmatrix}\\
            \leq & \frac{\|B\|_2\Lambda_{\max}(\Gamma)}{T}.
        \end{split}
        \end{equation*}
        Thus we only need to bound $\Lambda_{\max}(\Gamma)$. Applying Lemma \ref{Toeplitz}, the largest eigenvalue of Toeplitz matrix $\Gamma$ can be bounded by
        \begin{equation*}
            \begin{split}
                \Lambda_{\max}(\Gamma)\leq &\text{ess}\sup_{\lambda}\left|\sum_{l=-\infty}^{\infty}\sum_{j=0}^{\infty}\tilde{\alpha}_{|l|+j}\tilde{\alpha}_{j}e^{il\lambda}\right|\\
                \leq &\left|\sum_{l=-\infty}^{\infty}\sum_{j=0}^{\infty}\tilde{\alpha}_{|l|+j}\tilde{\alpha}_{j}\right|\\
                \leq &2\sum_{l=0}^{\infty}\left(\sum_{j=0}^{\infty}\tilde{\alpha}_{l+j}^2\right)^{\frac{1}{2}}\left(\sum_{j=0}^{\infty}\tilde{\alpha}_{j}^2\right)^{\frac{1}{2}}.
                \end{split}
        \end{equation*}
        where the third inequality is due to Cauchey-Schwartz inequality. Due to \eqref{PsipNorm}, we can further obtain
        \begin{equation*}
            \begin{split}
                \Lambda_{\max}(\Gamma)\leq &2\sum_{l=0}^{\infty}\left(\sum_{j=0}^{\infty}\sum_{n=0}^{p-1}\alpha_{l+j-n}^2\right)^{\frac{1}{2}}\left(\sum_{j=0}^{\infty}\sum_{n=0}^{p-1}\alpha_{j-n}^2\right)^{\frac{1}{2}}\\
                \leq &2p\left(\sum_{i=0}^{\infty}\alpha_{1-p+i}^2\right)^{\frac{1}{2}}\sum_{l=0}^{\infty}\left(\sum_{i=0}^{\infty}\alpha_{l+1-p+i}^2\right)^{\frac{1}{2}}\leq C(\beta).
            \end{split}
        \end{equation*}
         and we define $\alpha_i=0$ when $i<0$ for convenience. Therefore, 
        $$
        \left\|\frac{1}{T}\sum_{t=0}^{T-1}\Theta^{(t)\top}B\Theta^{(t)}\right\|_2\leq \frac{C\|B\|_2}{T}.
        $$
        \item[ii.] Bounding $\left\|\frac{1}{T}\sum_{t=0}^{T-1}\Theta^{(t)\top}B\Theta^{(t)}\right\|_{F}^2$\\
        First note that
        \begin{equation*}
            \begin{split}
                \left\|\frac{1}{T}\sum_{t=0}^{T-1}\Theta^{(t)\top}B\Theta^{(t)}\right\|_{F}^2\leq \frac{1}{T^2}\sum_{s,t=0}^{T-1}\left|\text{tr}\left(\Theta^{(s)\top}B\Theta^{(s)}\Theta^{(t)\top}B\Theta^{(t)}\right)\right|,
            \end{split}
        \end{equation*}
        and if we write $B=P^\top \Lambda P$ with orthogonal $P$ and diagonal $\Lambda$ (since $B$ is symmetric),
        \begin{equation*}
            \begin{split}
                &\left|\text{tr}\left(\Theta^{(s)\top}B\Theta^{(s)}\Theta^{(t)\top}B\Theta^{(t)}\right)\right|\\
                =&\left|\text{tr}\left(P\Theta^{(s)}\Theta^{(t)\top}B\Theta^{(t)}\Theta^{(s)\top}P^\top \Lambda\right)\right|\\
                \leq &\|B\|_{\tr}\left\|\Theta^{(s)}\Theta^{(t)\top}B\Theta^{(t)}\Theta^{(s)\top}\right\|_2\\
                \leq &\|B\|_{\tr}\|B\|_2\left\|\Theta^{(s)}\Theta^{(t)\top}\right\|_2^2.
            \end{split}
        \end{equation*}
        Meanwhile, due to that $\tilde{\alpha}_i=\left\|\Psi_i^{(p)}\right\|_2$ and \eqref{PsipNorm},
        \begin{equation*}
            \begin{split}
                \sum_{s,t=0}^{T-1}\left\|\Theta^{(s)}\Theta^{(t)\top}\right\|_2^2=&\sum_{s,t=0}^{T-1}\left\|\sum_{i=1}^{t\wedge s +m}\Psi_{t+m-i}^{(p)}\Psi_{s+m-i}^{(p)}\right\|_2^2\\
                \leq &\sum_{s,t=0}^{T-1}\left(\sum_{i=1}^{t\wedge s +m}\tilde{\alpha}_{t+m-i}\tilde{\alpha}_{s+m-i}\right)^2\\
                =&\sum_{s,t=0}^{T-1}\left(\sum_{i=0}^{(t\wedge s)+m-1}\tilde{\alpha}_i\tilde{\alpha}_{|t-s|+i}\right)^2\\
                \leq &\sum_{s,t=0}^{T-1}\left(p\sum_{i=0}^{\infty}\alpha_i^2\right)\left(p\sum_{i=1-p}^{\infty}\alpha_{|t-s|+i}^2\right).
            \end{split}
        \end{equation*}
        Note that $\sum_{i=0}^{\infty}\left(\sum_{j=0}^{\infty} \alpha_{i+j}^2\right)^{\frac{1}{2}}\leq \beta$,
        \begin{equation*}
            \begin{split}
                 \sum_{s,t=0}^{T-1}\left\|\Theta^{(s)}\Theta^{(t)\top}\right\|_2^2\leq &Cp^2\sum_{s,t=0}^{T-1}\left(\sum_{i=1-p}^{\infty}\alpha_{|t-s|+i}^2\right)\\
                \leq &Cp^2\sum_{l=0}^{T-1}2(T-l)\left(\sum_{i=1-p}^{\infty}\alpha_{l+i}^2\right)\\
                \leq &CT\sum_{l=0}^{\infty} \left(\sum_{i=0}^{\infty} \alpha_{l+i}^2\right)\\
                \leq &CT\left(\sum_{l=0}^{\infty} \left(\sum_{i=0}^{\infty} \alpha_{l+i}^2\right)^{\frac{1}{2}}\right)^2\leq CT,
            \end{split}
        \end{equation*}
        where the fourth line is due to Cauchey-Schwartz inequality. Therefore,
        \begin{equation*}
            \left\|\frac{1}{T}\sum_{t=0}^{T-1}\Theta^{(t)\top}B\Theta^{(t)}\right\|_{F}^2\leq \frac{C\|B\|_2\|B\|_{\tr}}{T}.
        \end{equation*}
    \end{itemize}
    Now we apply Lemma \ref{HansonWright}, and arrive at
    \begin{equation*}
        \mathbb{P}\left(\left|E_1-\mathbb{E}(E_1)\right|>\delta\right)\leq 2\exp\left\{-cT\min\left\{\frac{\delta}{\|B\|_2},\frac{\delta^2}{\|B\|_2\|B\|_{\tr}}\right\}\right\}.
    \end{equation*}
    \item[(2)] Bounding $E_2-\mathbb{E}(E_2)$\\
    We will show that $\left|E_2-\mathbb{E}(E_2)\right|$ vanishes when $m$ is large enough. First we bound $\|E_2\|_{\psi_1}$. Since
    \begin{equation*}
        \begin{split}
            \left|E_2\right|\leq \frac{1}{T}\sum_{t=0}^{T-1}\|B\|_2\left(\sum_{j=t+m}^\infty \tilde{\alpha}_j\|\epsilon_{t-j-1}\|_2\right)^2,
        \end{split}
    \end{equation*}
    by \eqref{PsipNorm} and \eqref{enormNorm},
    \begin{equation*}
        \begin{split}
            \|E_2\|_{\psi_1}\leq &\frac{2}{T}\sum_{t=0}^{T-1}\|B\|_2\left(\sum_{j=t+m}^\infty \tilde{\alpha}_j\left\|\left\|\epsilon_{t-j-1}\right\|_2\right\|_{\psi_2}\right)^2\\
            \leq &\frac{CM\|B\|_2}{T}\sum_{t=0}^{T-1}\left(\sum_{j=t+m}^\infty \tilde{\alpha}_j\right)^2\\
            \leq &CM\|B\|_2\left(\sum_{j=m}^\infty \tilde{\alpha}_j\right)^2\\
            \leq &CM\|B\|_2p^2\left(\sum_{j=m-p}^\infty \alpha_j\right)^2.
        \end{split}
    \end{equation*}
    Meanwhile,
    \begin{equation*}
        \begin{split}
            \left|\mathbb{E}(E_2)\right|=&\left|\frac{1}{T}\sum_{t=0}^{T-1}\text{tr}\left(B\sum_{j=t+m}^\infty \Psi_j^{(p)}\Psi_j^{(p)\top}\right)\right|\\
            \leq &\left|\frac{1}{T}\sum_{t=0}^{T-1}\|B\|_{\tr}\sum_{j=t+m}^\infty \tilde{\alpha}_j^2\right|\\
            \leq &p\|B\|_{\tr}\sum_{j=m-p}^\infty \alpha_j^2.
        \end{split}
    \end{equation*}
    For any $\delta>0$, let $m$ be sufficiently large such that $\sum_{j=m-p}^\infty \alpha_j^2<\frac{\delta}{2p\|B\|_{\tr}}$, $\|E_2\|_{\psi_1}\leq \frac{C\|B\|_2}{T}$, then by tail bound of sub-exponential random variable (see \cite{vershynin2010introduction}),
    \begin{equation*}
        \mathbb{P}\left(\left|E_2-\mathbb{E}(E_2)\right|>\delta\right)\leq C\exp\left\{-\frac{c\delta T}{\|B\|_2}\right\}.
    \end{equation*}
    \item[(3)] Bounding $E_3-\mathbb{E}(E_3)$\\
    One can show that 
    \begin{equation*}
        |E_3|\leq \frac{2\|B\|_2}{T}\sum_{t=0}^{T-1}\sum_{j=t+m}^\infty\tilde{\alpha}_j\|\epsilon_{t-j-1}\|_2\sum_{j=0}^\infty\tilde{\alpha}_j\|\epsilon_{t-j-1}\|_2,
    \end{equation*}
    and 
    $$
    \left\|\sum_{j=n}^\infty\tilde{\alpha}_j\|\epsilon_{t-j-1}\|_2\right\|_{\psi_2}\leq C\sqrt{M}\tau \sum_{j=n}^\infty\tilde{\alpha}_j\leq Cp\sqrt{M}\tau \sum_{j=n-p}^\infty \alpha_j.
    $$
    Thus 
    \begin{equation*}
    \begin{split}
        \|E_3\|_{\psi_1}\leq &\frac{4\|B\|_2}{T}\sum_{t=0}^{T-1}\left\|\sum_{j=t+m}^\infty\tilde{\alpha}_j\|\epsilon_{t-j-1}\|_2\right\|_{\psi_2}\left\|\sum_{j=0}^\infty\tilde{\alpha}_j\|\epsilon_{t-j-1}\|_2\right\|_{\psi_2}\\
        \leq &C\|B\|_2\sqrt{M}p\tau \left(\sum_{j=m-p}^\infty \alpha_j\right)\left(\sum_{j=0}^\infty \alpha_j\right)\\
        \leq &C\|B\|_2\sqrt{M}\sum_{j=m-p}^\infty \alpha_j.
    \end{split}
    \end{equation*}
    The first line is due to the following fact:
    For any two sub-Gaussian random variables $X$ and $Y$, $\left\|XY\right\|_{\psi_1}\leq 2\|X\|_{\psi_2}\|Y\|_{\psi_2}$. We can prove this in the following:
    \begin{equation*}
    \begin{split}
        \sup_{q\geq 1}q^{-1}\left(\mathbb{E}|XY|^q\right)^{\frac{1}{q}}\leq &\sup_{q\geq 1}q^{-1}\left(\mathbb{E}|X|^{2q}\right)^{\frac{1}{2q}}\left(\mathbb{E}|Y|^{2q}\right)^{\frac{1}{2q}}\\
        \leq &2\sup_{q\geq 1}q^{-\frac{1}{2}}\left(\mathbb{E}|X|^{q}\right)^{\frac{1}{q}}\sup_{q\geq 1}q^{-\frac{1}{2}}\left(\mathbb{E}|Y|^{q}\right)^{\frac{1}{q}}\\
        =&2\|X\|_{\psi_2}\|Y\|_{\psi_2},
    \end{split}
    \end{equation*}
    where the first line applies Cauchey-Schwartz inequality.
    Thus, with large enough $m$, $\|E_3\|_{\psi_1}\leq \frac{\|B\|_2}{T}$. Also, $\mathbb{E}(E_3)=0$, therefore implies the same bound for $E_3-\mathbb{E}(E_3)$ as the one for $E_2-\mathbb{E}(E_2)$:
    \begin{equation*}
        \mathbb{P}\left(\left|E_3-\mathbb{E}(E_3)\right|>\delta\right)\leq C\exp\left\{-\frac{c\delta T}{\|B\|_2}\right\}.
    \end{equation*}
\end{itemize}
In conclusion, for any $\delta>0$, if we choose some $m$ accordingly,
\begin{equation*}
    \begin{split}
        &\mathbb{P}\left(\left|\frac{1}{T}\sum_{t=0}^{T-1}\mathcal{X}_t^\top B\mathcal{X}_t-\text{tr}(B\Upsilon)\right|>\delta\right)\\
        \leq &\sum_{i=1}^3 \mathbb{P}\left(\left|E_i-\mathbb{E}(E_i)\right|>\frac{\delta}{3}\right)\\
        \leq &C\exp\left\{-cT\min\left\{\frac{\delta}{\|B\|_2},\frac{\delta^2}{\|B\|_2\|B\|_{\tr}}\right\}\right\}.
    \end{split}
\end{equation*}

\end{proof}

\begin{proof}[Proof of Lemma \ref{REC}]
     Here we apply some results in \cite{basu2015regularized} with a little change in notation. These results simplifies the original problem to finding a upper bound for $\left|v^\top (H-\Upsilon)v\right|$ with any fixed unit vector $v$. Specifically, the following lemmas are useful:
     \begin{Lemma}\label{cone_sparse}
     For any $J\subset \{1,\cdots,pM\}$, and $\kappa>0$,
     	$$
        \mathcal{C}(J,\kappa)\cap\{v\in \mathbb{R}^{pM}:\|v\|_2\leq 1\}\subset (\kappa+2)\text{cl}\left\{\text{conv}\left\{\mathcal{K}\left(\left|J\right|\right)\right\}\right\},
        $$
        where $\mathcal{K}(l)=\{v\in\mathbb{R}^{pM}: \|v\|_0\leq l, \|v\|_2\leq 1\}$ for any positive integer $l$.
     \end{Lemma}
     
     \begin{Lemma}\label{conv_sparse}
     	$$
     	\sup_{v\in \text{cl}\{\text{conv}(\mathcal{K}(l))\}}\left|v^\top Dv\right|\leq 3\sup_{v\in \mathcal{K}(2l)}\left|v^\top Dv\right|.
     	$$
     \end{Lemma}
     
     \begin{Lemma}\label{sparseREC}
     	Consider a symmetric matrix $D\in\mathbb{R}^{pM\times pM}$. If for any vector $v\in\mathbb{R}^{pM}$ with $\|v\|_2\leq 1$, and any $\eta\geq0$,
     	$$
     	\mathbb{P}\left(\left|v^\top Dv\right|>\eta\right)\leq c_1\exp\left\{-c_2T\min\left\{\eta,\eta^2\right\}\right\},
     	$$
     	then for any integer $l\geq 1$, 
     	$$
     	\mathbb{P}\left(\sup_{v\in \mathcal{K}(l)}\left|v^\top Dv\right|>\eta\right)\leq c_1\exp\left\{-c_2T\min\left\{\eta,\eta^2\right\}+l\min\left\{\log (pM),\log(21epM/l)\right\}\right\}.
     	$$
     \end{Lemma}
     
     By Lemma \ref{cone_sparse} and Lemma \ref{conv_sparse},
     \begin{equation*}
         \begin{split}
             &\sup\left\{\left|v^\top (H-\Upsilon)v\right|:v\in \mathcal{C}(J,\kappa),\|v\|_2\leq 1\right\}\\
             \leq &\sup\left\{\left|v^\top (H-\Upsilon)v\right|:v\in (\kappa+2)\text{cl}\left\{\text{conv}\{\mathcal{K}(|J|)\}\right\}\right\}\\
             \leq &3(\kappa+2)^2\sup\left\{\left|v^\top (H-\Upsilon)v\right|:v\in \mathcal{K}(2|J|)\right\}.
         \end{split}
     \end{equation*}
     For any unit vector $v\in \mathbb{R}^{pM}$, 
     \begin{equation*}
         v^\top (H-\Upsilon)v=\frac{1}{T}\sum_{t=0}^{T-1}\mathcal{X}_t^\top vv^\top \mathcal{X}_t-\text{tr}\left(vv^\top \Upsilon\right),
     \end{equation*}
     Thus $\left|v^\top (H-\Upsilon)v\right|$ can be bounded by Lemma \ref{QuaBound}. 
     \begin{equation*}
         \left\|vv^\top\right\|_{\tr}=\left\|vv^\top\right\|_2=\|v\|_2^2=1, 
     \end{equation*}
     which implies 
     \begin{equation*}
         \mathbb{P}\left(\left|v^\top (H-\Upsilon)v\right|>\eta\right)\leq c_1\exp\{-c_2T\min\{\eta,\eta^2\}\}.
     \end{equation*}
     By Lemma \ref{sparseREC}, when $|J|\log pM\leq C(\eta)T$,
     \begin{equation*}
         \sup\left\{\left|v^\top (H-\Upsilon)v\right|:v\in \mathcal{K}(2|J|)\right\}\leq \eta,
     \end{equation*}
     with probability at least $1-c_1\exp\{-c_2T\min\{\eta,\eta^2\}\}$. Let $\eta=[6(\kappa+2)^2]^{-1}\Lambda_{\min}(\Upsilon)\geq C(\kappa,\beta)$, then 
     \begin{equation*}
     \begin{split}
         &\inf\left\{v^\top Hv:v\in \mathcal{C}(J,\kappa),\|v\|_2\leq 1\right\}\\
         \geq &\Lambda_{\min}(\Upsilon)-\sup\left\{\left|v^\top (H-\Upsilon)v\right|:v\in \mathcal{C}(J,\kappa),\|v\|_2\leq 1\right\}\\
         \geq &\frac{1}{2}\Lambda_{\min}(\Upsilon)\geq C(\beta),
     \end{split}
     \end{equation*}
     with probability at least $1-c_1\exp\{-c_2T\}$, when $|J|\log pM\leq C(\kappa, \beta)T$, and $c_2$ depends on $\kappa$ and $\beta$. Here we apply Lemma \ref{CovEigenBnd} to lower bound the eigenvalues of $\Upsilon$.
\end{proof}

\section{Proof of Lemma \ref{DerivativeBnd}, \ref{SparsityBound}, \ref{InvBnd}, and \ref{trl2norm}}\label{RestProof}

\begin{proof}[Proof of Lemma \ref{DerivativeBnd}]
Recall that $f_l(z)=f_*(g_l(z))$, with $f_*(z)=\int_{-\infty}^{z-\frac{1}{2}} \phi(z) dz$, $g_1(z)=\left(\|Z+\mu\|_2-r-\varepsilon\right)/\varepsilon$, and $g_2(z)=\left(\|Z+\mu\|_2-r+2\varepsilon\right)/\varepsilon$. In order to bound the partial derivatives of composite function, we apply the following lemma which is a direct result of Proposition 1 and 2 in \cite{hardy2006combinatorics}. 
\begin{Lemma}\label{DerivativeFormula}
Suppose univariate function $f$ and $g$: $\mathbb{R}^n\rightarrow \mathbb{R}$ have derivatives and partial derivatives of orders up to $k$, then $\forall \{i_1,\dots,i_k\}\subset \{1,\dots,n\}$, 
\begin{equation*}
    \frac{\partial^k}{\partial x_{i_1}\cdots \partial x_{i_k}}f(g(x))
	=\sum_{\pi\in \Pi(k)}f^{(|\pi|)}(g(x))\prod_{B\in \pi}\frac{\partial^{|B|} g(x)}{\prod_{j\in B}\partial x_{i_j}},
\end{equation*}
where $\Pi(k)$ is the set of partitions for $\{1,\cdots,k\}$, and $B\in \pi$ is a block in $\pi$. Formally, 
	$$
	\Pi(k)=\{\{B_1,B_2,\cdots,B_n\}: B_i\cap B_j=\emptyset, \cup_i B_i=\{1,2,\cdots, k\}\}.
	$$
\end{Lemma}
By Lemma \ref{DerivativeFormula}, we can write out the $k$th order partial derivatives of $f_l$:  
	\begin{align*}
	\frac{\partial^k}{\partial z_{i_1}\cdots \partial z_{i_k}}f_l(z)
	=\sum_{\pi\in \Pi(k)}f_*^{(|\pi|)}(g_l(z))\prod_{B\in \pi}\frac{\partial^{|B|} g_l(z)}{\prod_{j\in B}\partial z_{i_j}}.
	\end{align*}
	Moreover, we can also write $g_l(z)$ as a composite function $\varphi_l(\psi(z))$, with $\varphi_1(x)=\frac{\sqrt{x}-r-\varepsilon}{\varepsilon}$, $\varphi_2(x)=\frac{\sqrt{x}-r+2\varepsilon}{\varepsilon}$, and $\psi(z)=\|z+\mu\|_2^2$. Then applying Lemma \ref{DerivativeFormula} on $g_l(z)$ gives us
	\begin{equation}\label{g_l_deriv}
	    \frac{\partial^n}{\partial z_{i_1}\cdots \partial z_{i_n}}g_l(z)=\sum_{\pi\in \Pi(n)}\varphi_l^{(|\pi|)}(\psi(z))\prod_{B\in \pi}\frac{\partial^{|B|} \psi(z)}{\prod_{j\in B}\partial z_{i_j}}.
	\end{equation}
	Note that 
	\begin{equation*}
	\frac{\partial^{|B|} \psi(z)}{\prod_{j\in B}\partial z_{i_j}}=\begin{cases}
	z_{i_j}+\mu_{i_j} & \text{if } B=\{j\} \text{ for any $j$}\\
	1(i_j=i_l) & \text{if } B=\{j,l\} \text{ for any $j,l$} \\
	0 & \text{if } |B|>2,
	\end{cases}
	\end{equation*}
	which means that we only need to consider the partitions with all blocks of size $1$ or $2$, when calculating the partial derivative of $g_l(z)$ using \eqref{g_l_deriv}. Also note that we need partitions for blocks within an original partition $\pi$, we define the following partition set $\mathcal{C}(\pi)$ for any partition $\pi=\{B_1,\dots,B_n\}$ of size $n$:
	\begin{equation*}
	    \mathcal{C}(\pi)=\left\{\cup_{i=1}^n \tilde{\pi}_i:\tilde{\pi}_i\in \Pi(B_i) s.t. \forall C\in \tilde{\pi}_i, |C|\leq 2\right\}.
	\end{equation*}
	This set $\mathcal{C}(\pi)$ include the unions of partitions for each block $B_i$ within $\pi$, and each block within the partition of $B_i$ has size bounded by $2$. Let $S(\tilde{\pi})=\{i:\{i\}\in \tilde{\pi}\}$, and $P(\tilde{\pi})=\{\{i,j\}:\{i,j\}\in \tilde{\pi}\}$, then the partial derivative of $f_l(z)$ can be expanded as
	\begin{equation}\label{f_l_deriv}
	    \frac{\partial^k}{\partial z_{i_1}\cdots \partial z_{i_k}}f_l(z)=\sum_{\substack{\pi\in \Pi(n)\\\tilde{\pi}\in \mathcal{C}(\pi)}}f_*^{|\pi|)}(g_l(z))C(\pi,\tilde{\pi})\frac{\Pi_{j\in S(\tilde{\pi})}(z_{i_j}+\mu_{i_j})\Pi_{\{j,l\}\in P(\tilde{\pi})}1(i_j=i_l)}{\varepsilon^{|\pi|} \|z+\mu\|_2^{2|\tilde{\pi}|-|\pi|}},
	\end{equation}
	where we apply the fact that $\varphi_l^{(k)}(x)=\frac{C(k)}{\varepsilon x^{k-\frac{1}{2}}}$.
	For each fixed $\pi\in \Pi(k)$ and $\tilde{\pi}\in \mathcal{C}(\pi)$,
	\begin{equation*}
	\begin{split}
	    &\left|\sum_{1\leq i_1,\cdots,i_k\leq d}y_{i_1}\cdots y_{i_k}\Pi_{j\in S(\tilde{\pi})}(z_{i_j}+\mu_{i_j})\Pi_{\{j,l\}\in P(\tilde{\pi})}1(i_j=i_l)\right|\\
	 =&\left|\left(y^\top (z+\mu)\right)^{|S(\tilde{\pi})|}\|y\|_2^{2|P(\tilde{\pi})|}\right|\leq \|y\|_2^k\|z+\mu\|_2^{|S(\tilde{\pi})|},
	\end{split}
	\end{equation*}
	then combine this with \eqref{f_l_deriv}, we have
\begin{equation*}
    \begin{split}
	&\left|\sum_{1\leq i_1,\cdots,i_k\leq d} y_{i_1}\cdots y_{i_k}\frac{\partial^k}{\partial z_{i_1}\cdots \partial z_{i_k}}f_l(z)\right|\leq \sum_{\substack{\pi\in \Pi(n)\\\tilde{\pi}\in \mathcal{C}(\pi)}}\frac{f_*^{(|\pi|)}(g_l(z))C(\pi,\tilde{\pi})\|y\|_2^k}{\varepsilon^{|\pi|}\|z+\mu\|_2^{k-|\pi|}}.
    \end{split}
	\end{equation*}
	In addition, note that $f_*^{(k)}(x)=\phi^{(k-1)}(x-\frac{1}{2})=0$ when $x\leq 0$ or $x\geq 1$, and is bounded on $(0,1)$.Thus we only have to consider $\|z+\mu\|_2>r+\varepsilon$ when $l=1$ and $\|z+\mu\|_2>r-2\varepsilon$ when $l=2$. If $r>3\varepsilon$ and $l=2$, $\|z+\mu\|_2>r-2\varepsilon>\varepsilon$. Therefore, 
	\begin{equation*}
    \begin{split}
	    &\left|\sum_{1\leq i_1,\cdots,i_k\leq d} y^{(i_1)}\cdots y^{(i_k)}\frac{\partial^k}{\partial z^{(i_1)}\cdots\partial z^{(i_k)}}f_l(z)\right|\\
	\leq &\sum_{\pi\in \Pi(k)}\sum_{(S_i,P_i)_{i=1}^{|\pi|}\in \mathcal{C}(\pi)}\frac{C(|\pi|)\|y\|^k}{\varepsilon^{k}}\leq C(k)\varepsilon^{-k}\|y\|^k.
    \end{split}
	\end{equation*}
\end{proof}

\begin{proof}[Proof of Lemma \ref{SparsityBound}]
Note that
$$
w_{m}^*=\Upsilon_{D_m^c,D_m^c}^{-1}\Upsilon_{D_m^c,D_m}=-(\Upsilon^{-1})_{D_m^c,D_m}\left[(\Upsilon^{-1})_{D_m,D_m}\right]^{-1}.
$$
When $A^*$ is symmetric, $\Upsilon^{-1}=I-(A^*)^2$, thus
\begin{equation*}
    \begin{split}
        w_m^*=\left((A^*)^2\right)_{D_m^c,D_m}\left[I-\left((A^*)^2\right)_{D_m,D_m}\right]^{-1}\in \mathbb{R}^{(M-d_m)\times d_m}.
    \end{split}
\end{equation*}
It is clear that 
\begin{equation*}
    \begin{split}
        s_m=\|w_m^*\|_0\leq&d_m\left|\{i:(w_m^*)_{i\cdot}\neq 0\}\right|\leq d_m\left|\{i:[(A^*)^2]_{i,D_m}\neq 0\}\right|.
    \end{split}
\end{equation*}
Let $R_m =\left|\{i:[(A^*)^2]_{i,D_m}\neq 0\}\right|$ and $C_m=\{j: A^*_{j,D_m}\neq 0\}$, then 
$$
|C_m|\leq d_m \max_{1\leq i\leq M}\|a_i^*\|_0
$$
and
$$
R_m\subset \{i: \text{supp}(A^*_{i\cdot})\cap C_m\neq \emptyset\}.
$$
Therefore,
\begin{equation*}
    \begin{split}
        s_m\leq d_m|R_m|\leq d_m\sum_{j\in C_m}|\text{supp}(A^*_{\cdot j})|\leq d_m^2(\max_{1\leq i\leq M}\|a_i^*\|_0)^2.
    \end{split}
\end{equation*}
\end{proof}

\begin{proof}[Proof of Lemma \ref{InvBnd}]
Let $Y=(B+\Delta)^{-1}$, then immediately we have $YB-I=-Y\Delta$, which is equivalent to $Y-B^{-1}=-Y\Delta B^{-1}$. Thus the $\ell_2$ norm of $Y-B^{-1}$ can be bounded by $\|Y\|_2\|\Delta\|_2\|B^{-1}\|_2$. Moreover, note that $\|Y\|_2\leq \|Y-B^{-1}\|+\|B^{-1}\|$, we have
\begin{equation*}
    \|Y\|_2\leq \|B^{-1}\|_2+\|Y\|_2\|\Delta\|_2\|B^{-1}\|_2,
\end{equation*}
and rearranging terms gives us 
\begin{equation*}
        \|Y\|_2\leq \frac{\|B^{-1}\|_2}{1-\|B^{-1}\|_2\|\Delta\|_2}.
\end{equation*}
Therefore,
\begin{equation*}
    \|Y-B^{-1}\|_2\leq \frac{\|B^{-1}\|^2_2\|\Delta\|_2}{1-\|B^{-1}\|_2\|\Delta\|_2}.
\end{equation*}
\end{proof}

\begin{proof}[Proof of Lemma \ref{trl2norm}]
First note that for any symmetric matrix $U$, we can write it as $U=P^\top \Lambda P$, with orthogonal matrix $P$ and diagonal matrix $\Lambda$.
By the definition of trace norm, 
\begin{equation*}
    \|U\|_{\tr}=\text{tr}\left(\sqrt{U^2}\right)=\text{tr}\left(\sqrt{P^\top \Lambda^2 P}\right)=\text{tr}\left(P^\top \sqrt{\Lambda^2}P\right)=\text{tr}\left(\sqrt{\Lambda^2}\right).
\end{equation*}
If we denote the non-zero eigenvalues of $U$ as $\lambda_1,\dots,\lambda_r$, then 
\begin{equation*}
    \|U\|_{\tr}=\text{tr}\left(\sqrt{\Lambda^2}\right)\leq r\max_i|\lambda_i|\leq r\|U\|_2.
\end{equation*}

\end{proof}

\end{document}